\documentclass[11pt]{amsart}
\usepackage{latexsym, amssymb, url, MnSymbol, amscd, amsfonts, amsxtra, amsmath, verbatim, mathrsfs, hyperref}
\usepackage[pdftex,lmargin=1.25in,rmargin=1.25in,tmargin=1.25in,bmargin=1.25in, marginparwidth=1.1in, marginparsep = 0.1in]{geometry}
\usepackage{todonotes}
\usepackage{marginnote}

\newcommand{\herm}{\mathrm{Herm}}
\newcommand{\hern}{\herm_n}
\newcommand{\Herm}{\herm}

\newcommand{\ci}{C^{\infty}}
\newcommand{\IC}{\mathbb{C}}
\newcommand{\Disc}{\mathrm{Disc}}

\newcommand{\hdrone}{H_{dR}^1}

\newcommand{\IQ}{\mathbb{Q}}
\newcommand{\OK}{\calO_{\cmfield}}
\newcommand{\calO}{\mathcal{O}}

\newcommand{\ZZ}{\mathbb{Z}}

\newcommand{\uA}{\underline{A}}
\newcommand{\uo}{\underline{\omega}}
\newcommand{\uO}{\underline{\Omega}}
\newcommand{\adeles}{\mathbb{A}}
\newcommand{\isomto}{\overset{\sim}{\rightarrow}}
\newcommand{\Isom}{\mathrm{Isom}}
\newcommand{\Gal}{\mathrm{Gal}}

\newcommand{\GL}{GL}

\newcommand{\CO}{\mathcal{O}}

\newcommand{\ord}{\mathrm{ord}}

\newcommand{\Ind}{\mathrm{Ind}}

\newcommand{\cmfield}{K}
\newcommand{\realfield}{K^+}
\newcommand{\diag}{\mathrm{diag}}
\newcommand{\IR}{\mathbb{R}}
\newcommand{\GSp}{GSp}

\newcommand{\Frob}{\mathrm{Frob}}
\newcommand{\ram}{\mathrm{ram}}
\newcommand{\Real}{\mathrm{Re}}
\newcommand{\cpct}{\mathcal{K}}
\newcommand{\Sp}{\mathrm{Sp}}

\newcommand{\GU}{GU}
\newcommand{\End}{\mathrm{End}}
\newcommand{\Res}{\mathrm{Res}}
\newcommand{\tr}{\mathrm{trace}}
\newcommand{\reflex}{E}
\newcommand{\moduli}{M}
\newcommand{\moduliint}{\mathcal{\moduli}}
\newcommand{\Spec}{\mathrm{Spec}}
\newcommand{\Sh}{\mathrm{Sh}}
\newcommand{\Mat}{\mathrm{Mat}}
\newcommand{\im}{\mathrm{Im}}
\newcommand{\re}{\mathrm{Re}}
\newcommand{\CA}{\mathcal{A}}
\newcommand{\CE}{\mathcal{E}}
\newcommand{\CH}{\mathcal{H}}
\newcommand{\Torus}{T}
\newcommand{\Borel}{B}
\newcommand{\Nilp}{N_\Borel}
\newcommand{\Levi}{H}
\renewcommand{\MR}[1]{ }

\makeindex

\theoremstyle{plain}
\newtheorem{thm}{Theorem}[subsection]
\newtheorem{prop}[thm]{Proposition}
\newtheorem{cor}[thm]{Corollary}
\newtheorem{lem}[thm]{Lemma}

\theoremstyle{definition}
\newtheorem{defi}[thm]{Definition}
\newtheorem{example}[thm]{Example}

\theoremstyle{remark}
\newtheorem{rmk}[thm]{Remark}

\newcommand*{\TitleFont}{%
      \usefont{T1}{cmr}{bx}{sc}%
      \fontsize{11}{15}%
      \selectfont}
\title[Automorphic forms on unitary groups]{\TitleFont \uppercase{\bf 
 Automorphic forms on unitary groups}}
\author{E. E. Eischen}\thanks{The author is grateful for support from NSF Grant DMS-1751281.
}

\address{E. Eischen\\
Department of Mathematics\\
University of Oregon\\
Fenton Hall\\
Eugene, OR 97403-1222\\
USA}
\email{eeischen@uoregon.edu}

\begin{document}
\bibliographystyle{amsalpha}

\maketitle

\vspace{-0.4in}
\begin{abstract}
This manuscript provides a more detailed treatment of the material from my lecture series at the 2022 Arizona Winter School on Automorphic Forms Beyond $\GL_2$.  The main focus of this manuscript is automorphic forms on unitary groups, with a view toward algebraicity results.  I also discuss related aspects of automorphic $L$-functions in the setting of unitary groups.
\end{abstract}

\setcounter{tocdepth}{2}

\tableofcontents

\section{Introduction}\label{sec:intro}
This manuscript 
has 
two main purposes:
\begin{enumerate}
\item{Introduce {\em automorphic forms on unitary groups} from several perspectives.}
\item{Illustrate some strategies in number theory (especially concerning algebraicity of values of automorphic $L$-functions), exploiting the fact that structures associated to unitary groups provide a convenient setting in which to work.}
\end{enumerate}

We start with some familiar examples.  Since it can be easy to get bogged down in sophisticated machinery, it is prudent to keep some familiar, yet relevant, examples in mind from the beginning.  In the mid-1700s, Leonhard Euler  proved that the values of the Riemann zeta series $\zeta(s) = \sum_{n\geq 1}n^{-s}$ at positive even integers are algebraic (rational, in fact) up to a well-defined transcendental factor.  More precisely, he proved that for each positive integer $k$,
\begin{align*}
\zeta(2k) = (-1)^{k+1}(2\pi)^{2k}\frac{B_{2k}}{2 (2k)!},
\end{align*}
where $B_{2k}\in\IQ$ is the $2k$th Bernoulli number (the $2k$th coefficient in the Taylor series expansion $
\frac{te^t}{e^t-1} = \sum_{n=0}^\infty B_n\frac{t^n}{n!}$).  
A century later, Ernst Kummer proved that far beyond merely being rational, the numbers $-\frac{B_{2k}}{2k}$ arising in the values of $\zeta(2k)$, which turn out to be the values at $1-2k$ of the analytic continuation of $\zeta(s)$, satisfy striking congruences mod powers of a prime number $p$  \cite{kummer}.  Rather than viewing these properties as a cute curiosity, Kummer was interested in information they encoded about cyclotomic fields.  Indeed, he showed that $p$ does not divide the class number of the cyclotomic field $\IQ\left(e^{2\pi i/p}\right)$ if and only if $p$ does not divide the numerators of the Bernoulli numbers $B_2, B_4, \ldots, B_{p-3}$, in which case he could prove special cases of Fermat's Last Theorem \cite{kummer36, kummer34}.   

The Riemann zeta function is the Dirichlet $L$-function attached to the trivial character, but one could also ask what happens if we replace the trivial character by, say, some nontrivial algebraic character.  In this case, once again, we obtain algebraic values at certain integer inputs, and the values have algebraic meaning.  Furthermore, in the twentieth century, these $L$-functions began to play a significant role in Galois-theoretic statements.  For example, picking up where Kummer had left off a century earlier, Kenkichi Iwasawa linked the behavior of Galois modules over towers of cyclotomic fields to $p$-adic zeta functions ($p$-adic analytic functions encoding the congruences first observed by Kummer).  It turned out that the congruences observed by Kummer encoded not only information about the sizes of class groups but also about structures of collections of class groups, viewed as Galois modules \cite{iw, iw2}.

So far, our discussion has only concerned $L$-functions attached to characters, but automorphic forms on unitary groups are already lurking nearby.  From class field theory, we have a correspondence between Hecke characters and representations of abelian Galois groups.  In fact, a Hecke character of $\adeles_L^\times$,\index{$\adeles_L^\times$ ideles of a number field $L$} the ideles of a number field $L$, is an automorphic form on $\GL_1(\adeles_L)$ or, equivalently when $L$ is a CM field, on idelic points of the {\em general unitary group $GU(1)$ of rank one}.  In other words, we have a correspondence between Galois representations of abelian extensions and automorphic forms (and of their $L$-functions), at least in this simple case.  The values of such $L$-functions can be shown to be algebraic, up to a well-determined transcendental factor.  

One of the most powerful techniques we have for proving such algebraicity results (as well as for proving analogues of Kummer's congruences) is to express the values of these $L$-functions in terms of modular forms, i.e.\ automorphic forms on $\GL_2$.  These are closely related to automorphic forms on the {\em general unitary group $GU(1, 1)$ of signature $(1, 1)$}, thanks to Isomorphism \ref{equ:GL2U11} below.
It turns out that structural aspects of modular forms are useful for proving results about these $L$-functions, including about the rationality and congruences exhibited by their values at certain points.  In the simplest example of this phenomenon, the rationality of $\zeta(1-2k)$ follows from the rationality of the Fourier coefficients in the nonconstant terms in the Fourier expansion of the Eisenstein series of weight $2k$ and level $1$,
\begin{align}\label{equ:G2k}
G_{2k}(q) = \zeta(1-2k)+2\sum_{n\geq 1}\sigma_{2k-1}(n)q^n,
\end{align}
where $q=e^{2\pi i z}$, $z$ is a point in the upper half plane, and $\sigma_{2k-1}(n) =\sum_{d\divides n}d^{2k-1}$.  This is the simplest implementation of an idea of Erich Hecke (that was later fleshed out by Helmut Klingen and Carl Siegel,
as explained in \cite[Section 1.3]{BCG} and \cite{IO}) to study algebraicity of values of zeta functions by exploiting properties of Fourier coefficients of modular forms \cite{klingen, siegel1, siegel2}.  

Moving a step further, we could investigate properties of the values of a {\em Rankin--Selberg convolution} of a weight $k$ holomorphic cusp form $f(q) = \sum_{n\geq 1}a_n q^n$ that we assume to be primitive (i.e.\ $a_1=1$ and $f$ is a common eigenfunction of the Hecke operators of level $N$) and a weight $\ell< k$ holomorphic modular form $g(q) = \sum_{n\geq 0}b_n q^n$ with algebraic Fourier coefficients $a_n$ and $b_n$ for all $n$.  That is, we consider the zeta series
\begin{align*}
D(s, f, g) = \sum_{n=1}^\infty \frac{a_n b_n}{n^s}.
\end{align*}
Goro Shimura proved that for all integers $k, \ell, m$ satisfying $\ell<k$ and $\frac{k+\ell-2}{2}<m<k$,
 \begin{align}\label{equ:Dmfgalg}
 \frac{D(m, f, g)}{\langle f, f\rangle}\in \pi^{k}\bar{\IQ},
 \end{align}
 where $\bar{\IQ}$ denotes an algebraic closure of $\IQ$ and $\langle, \rangle$ denotes the Petersson inner product \cite[Theorem 3]{shimura-RS}. 
 A key step in proving Expression \eqref{equ:Dmfgalg} is to express the values of $D(s, f, g)$ as a Petersson inner product involving an Eisenstein series.  More precisely, Shimura and Robert Rankin proved in \cite{shimura-RS, rankin} that 
\begin{align}\label{equ:RSconvmf}
D(k-1-r, f, g) = c\pi^k\langle\tilde{f}, g\delta_\lambda^{(r)} E\rangle,
\end{align}
 where $c=\frac{\Gamma(k-\ell-2r)}{\Gamma(k-1-r)\Gamma(k-\ell-r)}\frac{(-1)^r 4^{k-1}N}{3}\prod_{p\divides N}\left(1+p^{-1}\right)$ (with $N$ the level of the modular forms), $E$ denotes a particular holomorphic Eisenstein series of weight $\lambda:=k-\ell-2r$ and level $N$, $\tilde{f}(z):=\overline{f\left(-\bar{z}\right)},$ and $\delta_\lambda^{(r)}$ is a {\em Maass--Shimura} differential operator that raises the weight of a modular form of weight $\lambda$ by $2r$.  In fact, $\delta_\lambda^{(r)} E$ is also an Eisenstein series but, unlike the Eisenstein series we have encountered so far in this manuscript, is not holomorphic (although it is {\em nearly holomorphic}).  Expression \eqref{equ:Dmfgalg} then follows from the algebraicity of the Fourier coefficients of the Eisenstein series $E$, the fact that the Maass--Shimura operator preserves certain properties of algebraicity, and the decomposition over $\bar{\IQ}$ of the space of level $N$ modular forms of weight $k$ into an orthogonal basis of cusp forms (so the pairing in Equation \eqref{equ:RSconvmf} becomes a scalar multiple of $\langle f,  f\rangle$).  In Section \ref{autLfcnssection}, we will see a vast extension of this idea of expressing $L$-functions in terms of Eisenstein series (and other automorphic forms) to glean information about rationality properties of the values of $L$-functions.

Much more broadly, one might ask about analogous behavior for $L$-functions associated to other arithmetic data.  In the 1970s, Pierre Deligne formulated vast conjectures about certain values of $L$-functions at integer points \cite{deligne}.  His conjectures concern $L$-functions attached to {\em motives} $M$, which include $L$-functions attached to algebraic Hecke characters (i.e.\ the example mentioned above) and to holomorphic modular forms (the next natural example to consider, given their connection with $2$-dimensional Galois representations).  Roughly speaking, he conjectured that if an integer $m$ is {\em critical} for the motive $M$, then $L(m, M)$ is a rational multiple of a {\em period} 
associated to $M$.  The details of how to formulate our aforementioned results about Artin $L$-functions in this setting are the subject of \cite[Section 6]{deligne}.  More generally, though, this conjecture extends well beyond the low-dimensional examples discussed so far.  In parallel with the case mentioned above, Ralph Greenberg also extended Iwasawa's conjectures (concerning $p$-adic aspects) to this more general setting \cite{green3, green2, green1}.  Meanwhile, the Bloch--Kato conjectures predict the meaning of the order of vanishing at $s=0$ of $L(s, M)$ \cite{BK}.

Given that we just named three significant sets of conjectures about $L$-functions, it is natural to seek tools to prove them.  At present, our main route is through automorphic forms and automorphic representations (generated by automorphic forms).  That is, rather than dive in and prove such conjectures directly for the data to which the $L$-functions are attached, the most fruitful approach to date has been to work with automorphic forms associated to that data.  In our familiar examples above, we saw that automorphic forms on $\GL_1$ and $\GL_2$ played key roles.  We also noted, though, that we could instead view these automorphic forms as being defined on unitary groups.  
In fact, it turns out that the relative simplicity of the above cases (relative to aspects of the {\em higher-dimensional situation}, which is absolutely not a suggestion that the already sophisticated $1$- and $2$-dimensional cases should be considered simple!) obscures the fact that unitary groups can be more effective than $\GL_n$ (the group with which most people would probably be inclined to start) for extending certain techniques from dimensions $1$ and $2$ to higher-dimensional settings.

\subsection{Why work with unitary groups?}
Unitary groups form a particularly convenient class of groups with which to work, due to certain algebraic and geometric properties.  In particular, in analogue with the case of modular curves in the setting of modular forms, unitary groups have associated moduli spaces with integral models.  This enables us to study algebraic aspects of automorphic forms (which arise as sections of vector bundles over Shimura varieties, in analogue with modular forms that arise as sections of a line bundle over modular curves).  While the locally symmetric spaces for $\GL_n$ for $n\geq 3$ lack the structure of Shimura varieties, systems of Hecke eigenvalues for $\GL_n$ can be realized in the cohomology of unitary Shimura varieties \cite{clozel, scholze, HLTT, sugwoo}.  Related to this, we also have substantial additional results about Galois representations in this setting (e.g.\ \cite{skinnerGalois, ch, ch2, chenevier-harris, harris-takagi}), which enable us to study $L$-functions of Galois representations by instead studying $L$-functions of certain {\em cuspidal} automorphic representations.

In addition, thanks to representation-theoretic properties of unitary groups, we have convenient models for the $L$-functions associated to certain automorphic representations of unitary groups.  These models are useful both for proving analytic properties and for extracting algebraic information (and even $p$-adic properties, as seen in \cite{HELS}).  In fact, additional automorphic forms ({\em Eisenstein series}, of which the function $G_{2k}$ from Equation \eqref{equ:G2k} is the simplest case) come into play in the study of these $L$-functions, so that we can eventually turn questions about $L$-functions into questions about properties of automorphic forms. 

Working with unitary groups has enabled major developments (which go beyond the scope of this manuscript but several of which are mentioned here as motivation for studying automorphic forms on unitary groups), including a proof of the main conjecture of Iwasawa Theory for $\GL_2$ \cite{SkUr} and the rationality of certain values of automorphic $L$-functions (including \cite{shar, harriscrelle, harrisbirkhauser, harrisannals, guerberoff, guerberofflin}), as well as progress toward cases of the Bloch--Kato conjecture (including \cite{SkiUrb, klosin1, klosin2, wan2019iwasawa}), and the Gan--Gross--Prasad conjecture (many recent developments, including \cite{xue1, xue2, weizhang, yifengliu, zhiweiyun, jiangzhang, hongyuhe, bp1, bp2}). 

Because of their well-established power but also because many challenges remain, automorphic forms on unitary groups continue to play a significant role in research in number theory.  Current research spans a variety of topics.  In addition to applications to sophisticated problems like those mentioned above, there are also fundamental challenges associated with studying unitary Shimura varieties themselves (which arise independently of the automorphic theory but substantially impact the automorphic theory and associated $L$-functions).  While unitary groups are related to symplectic groups and there is overlap in approaches to the two classes of groups, unitary groups generally present more challenges (including in terms of geometry, e.g.\ as documented in \cite{RSZ} and \cite{EiMa}) than one sees in the symplectic setting.  At a more basic-sounding level, there is still much progress to be made even for computing examples on low-rank unitary groups (although some progress has been made under certain conditions, e.g.\ in \cite{loeffler, bwilliams}).

Given all this, this manuscript focuses
 on the following topics:
\begin{itemize}
\item{Introduction to automorphic forms on unitary groups from several perspectives (including analytic, algebraic, and geometric), as well as connections between those perspectives}
\item{Aspects of $L$-functions associated to certain automorphic representations of unitary groups, with a view toward algebraicity results}
\item{Constructions of examples of automorphic forms on unitary groups}
\end{itemize}

\subsubsection{Why present this particular set of topics (and not others) here?}\label{sec:whythesetopics}
One could fill hundreds of pages discussing automorphic forms and their roles (even if one restricts to unitary groups) and still have a lot of ground left to cover.  For this manuscript, the author has aimed to select a cohesive and relevant collection of topics that give readers a taste of some fundamental parts of this area while also meeting three key criteria: 
\begin{enumerate}
\item{These topics have arisen repeatedly in recent research developments, including those carried out by this manuscript's author.}
\item{These topics are closely tied to questions that the author gets asked.}
\item{These topics are appropriate for graduate students and others getting started in this area.}
\end{enumerate}
To assist readers who hope to explore further, this manuscript cites many papers that go into further details about specific topics.  There are also a number of excellent, more extensive resources (such as books) that cover a lot of related material on automorphic forms, and readers are also encouraged to delve into those resources, for example \cite{getz, bellaiche, bumpbook, automorphicproject, shar, parisbookproject}.

Given the three key criteria here, those familiar with the author's research might wonder why the manuscript barely mentions $p$-adic automorphic forms and $p$-adic methods.  While $p$-adic methods meet (at least) the first two key criteria, it is imperative to understand the basics of analytic and algebraic aspects of automorphic forms on unitary groups before moving on to the $p$-adic setting.  An additional prerequisite for successfully working with $p$-adic automorphic forms is a strong understanding of $p$-adic modular forms (the $\GL_2$ setting), which is at odds with the focus of this workshop (namely, the {\em beyond} $GL_2$ setting).  In fact, as discussed in \cite[Section 1]{bordeaux}, all known constructions of $p$-adic $L$-functions appear to be adaptations of the specific techniques employed in the proof of the algebraicity of the values of the corresponding $L$-function.  (For example, Serre's approach to proving Kummer's congruences and constructing $p$-adic zeta functions in \cite{serre} builds directly on the ideas of Hecke, Klingen, and Siegel employing the constant terms of Eisenstein series summarized above.  For additional examples, see Remark \ref{rmk:padicrecipe}.) So this manuscript covers a subset of the prerequisites necessary for embarking on a study of $p$-adic methods for automorphic forms on unitary groups.  

One advantage of several of the topics presented here is that the broad strategies or ideas here might be transferrable to other groups, even if the technical details of how to carry out those strategies are specific to the group at hand.  For example, the recipe for proving algebraicity of critical values of $L$-functions employed by Michael Harris in the setting of unitary groups (including in \cite{harriscrelle, harrisbirkhauser}) has also been successfully adapted to certain other cases.  Indeed, his strategy is an extension of the one introduced by Shimura for proving algebraicity of critical values of Rankin--Selberg convolutions of modular forms (on $\GL_2$) mentioned above.  Technical ingredients (like Shimura varieties) that play a crucial role in Harris's work are not necessarily available for other groups, but nevertheless, the recipe can be inspiring for how to proceed in other settings.  To give a specific example, note that the overall strategy for proving algebraicity of critical values of Spin $L$-functions for $\GSp_6$ in \cite{ERS} is an (admittedly substantial, when it comes to certain technical issues) adaptation of Harris's (and Shimura's) approach, even though this setting lacks what appear at first to be crucial ingredients from the unitary setting. 
Thus, while a primary goal is to prepare readers to work with automorphic forms on unitary groups, a significant secondary objective is to help readers gain some intuition that might apply more broadly.

\subsection{Who is this manuscript for?}
This manuscript is for graduate students and others who are getting starting started doing research involving automorphic forms on unitary groups.  It introduces aspects of some core topics that can serve as a launchpad for mathematicians hoping to explore further.  The manuscript is especially aimed at those looking to understand connections between algebraic, geometric, and analytic aspects of automorphic forms on unitary groups.

Deservedly, there has been much recent attention on the setting of unitary groups.  Questions I have received suggest, though, that there is substantial demand for an accessible entry point, as it can be challenging to enter such a developed and dynamic field.  A key aim of this manuscript is to provide a welcoming entrance to some of the extensive work in this area.  Exercises related to the material presented here can be found in the problem sets prepared by Lynnelle Ye.

\subsubsection{Recommended material to learn first}
This manuscript is written for readers who are already familiar with automorphic forms on $\GL_2$, i.e.\ modular (and Hilbert modular) forms.  You will get the most out of this manuscript if you already know about modular forms and some of their uses in various settings.  This includes the classical analytic perspective (e.g.\ as in \cite[Chapter VII]{serrebook} and \cite[Chapter III]{koblitzbook}), the algebraic geometric perspective (e.g.\ as in \cite[Section 1]{katzmodular} and \cite[Chapters 1 and 2]{gorenbook}), and the automorphic or representation-theoretic perspective (e.g.\ as in \cite{gelbartbook} and \cite{bumpbook}).  Not knowing {\em some} of this background material will not be problematic, but not knowing {\em most} of it will make it difficult to have the intuition necessary to follow portions of the material presented here.  If you find that you need to develop your understanding of the $\GL_2$ setting further before proceeding, there are also other excellent resources for learning the fundamentals of modular forms, including \cite{miyakebook, DSbook, diamondimbook, 123book, steinbook, ribetsteinbook} and the online resources from the {\em Arizona Winter Semester 2021: Virtual School in Number Theory}.\footnote{\url{https://www.math.arizona.edu/~swc/aws/2021/index.html}}  

If you know the $\GL_2$ case well, then you will have the intuition necessary to pause periodically and think about how the material presented here specializes to the $n=1$ case.  Especially when you are stuck, it can be useful to specialize to the setting of modular forms to see what insight that more familiar setting offers.  Sometimes this will be particularly helpful (for instance, when studying algebraicity of values of $L$-functions, which follows from an extension of Shimura's approach in the setting of modular forms).  Even in cases where the $n=1$ case is too simple to offer much insight into the case of higher rank groups (e.g.\ singular forms are just constant functions in the $n=1$ case), it can still offer a helpful reality check along the way.  What might at first look like abstract formalities or a mess of heavy notation will often instead become a natural extension of what you already know from the setting of modular forms for $\GL_2$.

\subsection{Acknowledgements}
This manuscript would not be possible without the expertise and perspective I gained while working with my collaborators on research projects concerning automorphic forms on unitary groups.  In particular, I would like to thank Ana Caraiani, Jessica Fintzen, Maria Fox, Alex Ghitza, Michael Harris, Jian-Shu Li, Zheng Liu, Elena Mantovan, Angus McAndrew, Chris Skinner, Ila Varma, and Xin Wan.  Conversations I had with each of these collaborators substantially shaped my understanding of automorphic forms on unitary groups.  I am especially grateful to Mantovan for her insightful answers to my questions about geometry, which have clarified my understanding of some of the geometric aspects of the material presented here.  I am also especially grateful to Skinner and Harris, as well as my postdoctoral mentor Matthew Emerton, for helping me get started in this area and patiently answering my questions.  I hope that this manuscript serves as a resource for some of their (and others') future students (and perhaps saves them from some of the tedious sort of questioning to which I subjected them when I did not have access to such a written resource).

I am grateful to the people who provided feedback on portions
of earlier versions of this manuscript, including Utkarsh Agrawal, Francis Dunn, Bence Forr\'as, Maria Fox (who also was the discussion leader), Sean Haight, Andy Huchala, Angus McAndrew, Phil Moore, Yogesh More, Sam Mundy (who also was the project assistant), Samantha Platt, Wojtek Wawr\'ow, Pan Yan, Lynnelle Ye (who also wrote accompanying problem sets), and the anonymous referees.  
I am also grateful to the organizers of the Arizona Winter School (Alina Bucur, Bryden Cais, Brandon Levin, Mirela Ciperiani, Hang Xue, and  David Zureick-Brown) for organizing the excellent workshop and inviting me to be a lecturer.

\section{Unitary groups and PEL data}
Before introducing automorphic forms on unitary groups, it is prudent to establish some basic information about unitary groups and PEL data.  This section includes properties of associated moduli spaces and certain representations that will occur in our definitions of automorphic forms in Section \ref{autformsunitarygpssection}.

\subsection{A first glance at unitary groups}\label{unitarygpsbasicsection}
Readers who are already familiar with unitary groups are encouraged to skip this short introduction to unitary groups and start with Section \ref{PELdatasection}, where we associate a unitary group to PEL data.  This section briefly introduces unitary groups to make sure all readers, including beginners, start off having seen at least basic definitions of the groups with which we work.

Let $\cmfield$ be a quadratic imaginary extension of a totally real field $\realfield$.  Let $V$ be an $n$-dimensional vector space over $\cmfield$, and let $\langle, \rangle$ be a nondegenerate $\cmfield$-valued Hermitian pairing on $V$, i.e.\ $\langle, \rangle: V\times V\rightarrow \cmfield$ is linear in the first variable and conjugate-linear in the second variable, and
\begin{align*}
\langle v, w\rangle = \overline{\langle w, v\rangle}
\end{align*}
for all $v$ and $w$ in $V$.  Note that we can linearly extend any such Hermitian pairing to a $\realfield$-algebra $S$ and the $S$-module $V\otimes_{\realfield}S$.

\begin{defi}\label{unitarygpdefn}
The {\em unitary group} associated to $\left(V, \langle, \rangle\right)$ is the algebraic group $U:=U(V, \langle, \rangle)$\index{$U(V, \langle, \rangle)$}\index{$U$} whose $S$-points, for each $\realfield$-algebra $S$, are given by
\begin{align*}
U(S) = \left\{g\in \GL_{\cmfield\otimes_{\realfield}S}(V\otimes_{\realfield} S) \mid \langle gv, gw\rangle = \langle v, w\rangle\right\}.
\end{align*}
 \end{defi}
 This is a group scheme defined over $\realfield$.
 Given $g\in U(S)$, we define $g^\dag$
 to be the unique element of $U(S)$ such that
 \begin{align*}
 \langle gv, w\rangle = \langle v, g^\dag w\rangle
 \end{align*}
 for all $v, w\in V.$
Note that since $\langle, \rangle$ is nondegenerate, we also have that $gg^\dag$ is the identity element in $U$, and
\begin{align*}
g^{\dag\dag} = g.
\end{align*}
(We have $\langle v, g^{\dag\dag}w\rangle = \langle g^\dag v, w\rangle = \overline{\langle w, g^\dag v\rangle} = \overline{\langle gw, v\rangle} = \langle v, gw\rangle$.  The fact that $\langle, \rangle$ is nondegenerate then shows $g = g^{\dag\dag}$.)

\begin{defi}
The subgroup $SU:=SU(V, \langle, \rangle)$\index{$SU(V, \langle, \rangle)$} of the unitary group $U$ with determinant $1$ is called a {\em special unitary group}.  (It is a subgroup of the special linear group.)
\end{defi}

The unitary group $U$ is a subgroup the group of unitary similitudes defined as follows.
\begin{defi}
The {\em general unitary group} or {\em unitary similitude group} or {\em group of unitary similitudes} associated to $\left(V, \langle, \rangle\right)$ is the algebraic group $\GU:=GU(V, \langle, \rangle)$\index{$GU(V, \langle, \rangle)$} whose $S$-points, for each $\realfield$-algebra $S$, are given by
\begin{align*}
 \GU(S) = \left\{g\in \GL_{\cmfield\otimes_{\realfield}S}(V\otimes_{\realfield} S) \mid \langle gv, gw\rangle = \nu(g) \langle v, w\rangle, \nu(g)\in \GL_1(S)\right\}.
 \end{align*}
 The homomorphism $\nu: \GU\rightarrow \GL_1$ is called a {\em similitude factor}.  
\end{defi}

We have an exact sequence
\begin{align*}
1\rightarrow U\rightarrow \GU\xrightarrow{g\mapsto \nu(g)} \GL_1 \rightarrow 1.
\end{align*}
Sometimes (e.g.\ when realizing $GU$ as an algebraic group), it is convenient to identify $GU$ with the group of tuples
\begin{align*}
\left\{(g, \nu(g))\in \GL_{\cmfield\otimes_{\realfield}S}(V\otimes_{\realfield} S)\times GL_1(S) \mid \langle gv, gw\rangle = \nu(g) \langle v, w\rangle\right\}.
\end{align*}

If we choose an ordered basis $v_1, \ldots, v_n$ for $V$, then we may identify $\langle, \rangle$ with an $n\times n$-matrix $A$ with coefficients in $S$ and $V$ with $S^n$ (via $v_i\mapsto e_i$ with $e_i$ the $i$-th standard basis vector, viewed as a column vector for the moment), via
\begin{align*}
\langle  v_i,  v_j\rangle = {}^t e_i A e_j.
\end{align*}
Note that our convention in this manuscript is always to write the superscript $t$ on the left side of the matrix of which we are taking the transpose.\index{${ }^t$ superscript on lefthand side of matrix denotes transpose of that matrix, i.e.\ ${ }^tA$ denotes the transpose of a matrix $A$}

\begin{rmk}\label{rmk:preservemx}
If $S=\IR$, then an ordered choice of basis identifies $V\otimes_{\realfield}\IR$ with $\IC^n$, viewed for the moment as row vectors.  If $\langle, \rangle$ is a Hermitian pairing on $V$, then there is a Hermitian matrix $A$ (i.e.\ $A = A^\ast:={ }^t\bar{A}$, where the lefthand superscript ${ }^t$ denotes the transpose and $\bar{ }$ denotes the complex conjugate\index{$^\ast$ in the righthand superscript of a matrix denotes the transpose conjugate of that matrix, i.e.\ $A^\ast = { }^t\bar{A}$ for a matrix $A$}) such that
\begin{align*}
\langle v, w\rangle =v A w^\ast
\end{align*}
for all $v, w\in V$.  
 After a change of basis, the Hermitian matrix corresponding to the nondegenerate pairing $\langle, \rangle$ can be written in the form
\begin{align*}\index{$I_{a, b}$}
I_{a, b}:=\begin{pmatrix}
1_a & 0\\
0 & -1_b
\end{pmatrix},
\end{align*}
with $a+b = n$.  The tuple $(a, b)$ is called the {\em signature} of $\langle, \rangle$.  If $ab=0$, we say the unitary group preserving $\langle, \rangle$ is {\em definite}.

When $a=b$, we can also choose a basis with respect to which the matrix corresponding to $\langle, \rangle$ is $i\eta$, where
\begin{align*}\index{$\eta=\eta_a$}
\eta:=\eta_a:=\begin{pmatrix}0&-1_a\\ 1_a & 0\end{pmatrix}.
\end{align*}

When the group under consideration has signature $(a, b)$, it is conventional to write $SU(a, b)$,\index{$SU(a, b)$} $U(a, b)$,\index{$U(a, b)$} or $\GU(a, b)$.\index{$GU(a, b)$}  Sometimes, this notation is reserved for the group of matrices preserving $I_{a, b}$ if $a\neq b$ and $\eta_n$ if $a=b$, and this is the convention we will employ going forward.  It is also conventional to write $U(A)$\index{$U(A)$ where $A$ denotes a matrix} for the matrix group preserving a Hermitian matrix $A$ (and $SU(A)$ for the subgroup of matrices of determinant $1$ and $GU(A)$ for the corresponding similitude group).

In Section \ref{PELdatasection}, we will introduce groups $G$ and $G_1$ that are defined over $\IQ$ and are closely related to $\Res_{\realfield/\IQ}GU$ and $\Res_{\realfield/\IQ}U$, respectively.  (Given a group $H$ defined over $\realfield$, $\Res_{\realfield/\IQ}$\index{$\Res_{\realfield/\IQ}$} denotes the restriction of scalars functor, i.e.\ $\left(\Res_{\realfield/\IQ}H\right)(S) := H(S\otimes_\IQ\realfield)$ for each $\IQ$-algebra $S$.)  Note that 
\begin{align*}
\left(\Res_{\realfield/\IQ}GU\right)(\IR) &\cong \prod_{\tau: \realfield\hookrightarrow \IC}GU\left(a_\tau, b_\tau\right)\\
\left(\Res_{\realfield/\IQ}U\right)(\IR) &\cong \prod_{\tau: \realfield\hookrightarrow \IC}U\left(a_\tau, b_\tau\right).
\end{align*} 
In this case, the signature is the tuple $(a_\tau, b_\tau)_\tau$.  Given such a tuple $(a, b)$, we write $U_{a, b}(\IR)$\index{$U_{a, b}(\IR)$} for $\prod_{\tau}U(a_\tau, b_\tau)$ (and similarly for $SU_{a, b}$\index{$SU_{a, b}(\IR)$} and $GU_{a, b}$\index{$GU_{a, b}(\IR)$}).  If the unitary group is definite of signature $(n, 0)$ or $(0, n)$, we often write $n$ in place of the pair $(n,0)$ or $(0, n)$.
\end{rmk}

\subsubsection{A close relationship between $GL_2$ and $GU(1,1)$}\label{sec:GL2U11}
In Section \ref{sec:intro}, we noted that automorphic forms on $GU(1, 1)$ and $GL_2$ are closely related.  Later, we will see that the symmetric space for $GU(1, 1)$ is a finite set of copies of the upper half plane (the symmetric space for $GL_2$).  There is also a close relationship between automorphic forms on $GU(1, 1)$ and modular forms.  The relationship stems from the isomorphism
\begin{align}\label{equ:GL2U11}
GU(1, 1)\cong \left(\GL_2\times\Res_{\cmfield/\realfield}\mathbb{G}_m\right)/\mathbb{G}_m,
\end{align}
where $\mathbb{G}_m$ is embedded as $\alpha\mapsto (\diag(\alpha, \alpha), \alpha^{-1})$ and, following the usual convention, $\mathbb{G}_m$\index{$\mathbb{G}_m$} denotes the multiplicative group.

\subsubsection{Unitary groups over local fields}
Above, we defined the signature of a unitary group associated to a Hermitian pairing $\langle, \rangle$ on a $\IC$-vector space.  More generally, we will encounter $V_v:=V\otimes_{\realfield}\realfield_v$, where $v$ is a finite place of $\realfield$.  If $v$ splits as $w\bar{w}$ in $\cmfield$, then the decomposition
\begin{align*}
\cmfield\otimes_{\realfield}\realfield_v\cong \cmfield_w\oplus\cmfield_{\bar{w}}.
\end{align*}
induces a decomposition $V_v = V_w\oplus V_{\bar{w}}$.  Note that the nontrivial element of $\Gal(\cmfield/\realfield)$ swaps the two summands in the direct sum here, and $U(\realfield_v)$ fixes each summand.  Furthermore, still assuming that $v$ splits as $w\bar{w}$ in $\cmfield$, we get isomorphisms
\begin{align}\label{UGLniso}
U(\realfield_v)\cong \GL_n(\cmfield_w)\cong \GL_n(\realfield_v).
\end{align}
On the other hand, in the case where $v$ is inert, $\cmfield_w/\realfield_v$ is a quadratic extension of $p$-adic fields (for $v$ a prime over $(p)$), in which case the structure is described in, e.g.\, \cite[Section 1]{harris-unitary}.

\subsection{PEL data}\label{PELdatasection}

In this section, we introduce {\em data of PEL type}, along with corresponding moduli spaces.  In analogue with the setting of modular forms, which can be viewed as sections of a line bundle over a modular curve (a moduli space of elliptic curves with additional structure), we will later define automorphic forms as sections of a vector bundle over a higher-dimensional generalization of the modular curve (namely Shimura varieties, which serve as moduli spaces for certain abelian varieties with additional structure).  As we shall see in Section \ref{sec:moduliproblem}, the data of PEL type introduced below correspond to a moduli problem.  

In analogue with the moduli problem of classifying pairs consisting of an elliptic curve and a level structure, our moduli problem concerns tuples of abelian varieties with not only a level structure but also additional structures (polarization and endomorphism).  Similarly to the case of modular forms, realization of our automorphic forms in terms of algebraic geometry will enable us to work over base rings (and schemes) beyond just $\IC$.  This is essential for considering questions about algebraicity or rationality.

For an excellent introduction to Shimura varieties, the reader is encouraged to consult \cite{lan}.  Note that because we are concerned with unitary groups in this manuscript, we have specialized from the beginning to PEL data that will correspond to unitary Shimura varieties.
For detailed references specific to the case of unitary groups, 
the reader is encouraged to consult \cite[Section 5]{kottwitz} and \cite[Section 5.1]{lan}.

We consider tuples $\mathfrak{D}:=(D, \ast, \CO_D, V, \langle, \rangle, L, h)$\index{$\mathfrak{D}$} consisting of:
\begin{itemize}
\item{A finite-dimensional semisimple $\IQ$-algebra $D$\index{$D$}, each of whose simple factors has center CM field $\cmfield$\index{$\cmfield$}}
\item{A positive involution $\ast$\index{$\ast$} on $D$ over $\IQ$, by which we mean an anti-automorphism (i.e.\ reverses the order of multiplication in $D$) of order $2$ such that $\tr_{D\otimes_\IQ \IR}(x x^\ast)>0$ for all nonzero $x\in D\otimes_\IQ\IR$}
\item{A $\ast$-stable $\ZZ$-order $\CO_D$\index{$\CO_D$} in $D$}
\item{A nonzero finitely generated left $D$-module $V$\index{$V$}}
\item{A nondegenerate $\IQ$-valued alternating form $\langle, \rangle$\index{$\langle, \rangle$} on $V$ such that $\langle bv, w\rangle = \langle v, b^\ast w \rangle$ for all $b\in D$ and $v,w\in V$}
\item{A lattice $L\subseteq V$ preserved by $\CO_D$ that is self-dual with respect to $\langle, \rangle$.}
\item{A $\ast$-homomorphism $h:\IC\rightarrow \End_{D\otimes_\IQ\IR}(V\otimes_\IQ\IR)$\index{$h$} 
 (and by {\em $\ast$-homomorphism}, we mean an $\IR$-algebra homomorphism such that $h(z)^\ast  = h(\bar{z})$ for all $z\in\IC$), such that the symmetric $\IR$-linear bilinear form $\langle \cdot, h(i) \cdot \rangle$ on $V_\IR:=V\otimes_\IQ\IR$\index{$V_\IR$} is positive definite}
\end{itemize}
\begin{rmk}
To see that the pairing $(\cdot, \cdot):=\langle \cdot, h(i) \cdot\rangle$ is symmetric, observe that for all $v, w\in V$,
\begin{align*}
(v, w):=\langle v, h(i) w\rangle=\langle h(-i) v, w\rangle = -\langle h(i)v, w\rangle = \langle w, h(i) v\rangle = (w, v).
\end{align*}
\end{rmk}
\begin{rmk}
We view $K$ as embedded in $D$ via the diagonal embedding of $K$ into each of the simple factors of $D$.
\end{rmk}

Such a tuple $\mathfrak{D}$ is called a {\em PEL type datum} or {\em a datum of PEL type}.  This is the setup given in \cite[Section 5]{kottwitz}.  (N.B.\ Because we have specialized to the unitary case and so are, for example, excluding the Siegel case, we have required that $\cmfield$ be a CM field and not a totally real field.)
One can also consider an {\em integral PEL datum}, as in, for example, \cite[Section 5.1.1]{lan}.  To distinguish the previous case from the integral case, we sometimes call the previous case a {\em rational} PEL datum.  In the integral case, we consider tuples $(\CO, \ast, L, \langle, \rangle, h)$ consisting of:
\begin{itemize}
\item{An order $\CO$ in a finite-dimensional semisimple $\IQ$-algebra $D$}
\item{A positive involution $\ast$ on $\CO$}
\item{A $\CO$-module $L$ that is finitely generated as a $\ZZ$-module}
\item{A nondegenerate alternating pairing $\langle, \rangle: L\times L\rightarrow 2\pi i \ZZ$ such that $\langle bv, w\rangle = \langle v, b^\ast w \rangle$ for all $b\in \CO$ and $v,w\in L$}
\item{A $\ast$-homomorphism $h:\IC\rightarrow \End_{\CO\otimes_\ZZ\IR}\left(L\otimes_\ZZ\IR\right)$ such that the symmetric $\IR$-linear bilinear form $(2\pi i)^{-1}\langle \cdot, h(i) \cdot \rangle$ on $L\otimes_\ZZ\IR$ is positive definite.}
\end{itemize}

To our integral PEL datum, following \cite[Section 5.1.3]{lan}, we attach a group scheme $G$ over $\ZZ$ whose $R$-points, for each ring $R$, are given by
\begin{align}\label{GofR}\index{$G$}
G(R):=\left\{(g, r)\in \End_{\CO\otimes_\ZZ R}(L\otimes_\ZZ R)\times R^\times \mid \langle gv, gw\rangle = r\langle v, w\rangle \mbox{ for all } v, w\in L\otimes_\ZZ R\right\}.
\end{align}

\begin{rmk}
Note that from an integral PEL type datum, one obtains a (rational) PEL datum as above by tensoring the integral data with $\IQ$ (and renormalizing the bilinear forms).  As explained in \cite[Section 5.1.1]{lan}, in the moduli problem that we will discuss below, rational PEL data are most naturally suited to working with isogeny classes of abelian varieties, while integral PEL data are most readily suited to the language of isomorphism classes of abelian varieties.
\end{rmk}

Let $\realfield$\index{$\realfield$} be the fixed field of $\ast$.  Then $\cmfield/\realfield$ is a quadratic imaginary extension.  Let
\begin{align*}
n=\dim_\cmfield V.\index{$n$}
\end{align*}
Note that $V_\IC:=V_\IR\otimes_\IR\IC$\index{$V_\IC$} decomposes as
\begin{align*}\index{$V_1$}\index{$V_2$}
V_\IC = V_1\oplus V_2,
\end{align*}
where $V_1$ is the submodule on which $h(z)$ (more precisely, $h(z)\times 1$) acts by $z$ and $V_2$ is the submodule on which $h(z)$ (more precisely, $h(z)\times 1$) acts by $\overline{z}$.  The {\em reflex field} $\reflex$\index{$\reflex$} of $\mathfrak{D}$ is defined to be the field of definition of the isomorphism class of the $\IC$-representation $V_1$ of $D$.  For more details about how to view the reflex field and its geometric significance, see \cite[Section 5.1.1]{lan} or \cite[Section 1.2.5]{la}.

\begin{rmk}
Let $p$ be a prime number.  For defining PEL type moduli problems over $\CO_{\reflex, (p)}$, it is also useful to replace our PEL datum with a {\em $p$-integral PEL datum}, i.e.\ (as in \cite[Equation (5.1.2.2)]{lan})
$(\CO\otimes_\ZZ\ZZ_{(p)}, \ast, L\otimes_\ZZ\ZZ_{(p)}, \langle, \rangle, h)$ 
with $L\otimes_\ZZ\ZZ_{(p)}$ required to be self-dual under the resulting Hermitian pairing on $L\otimes_\ZZ\IQ_p$ and $p\ndivides \Disc(\CO)$.
\end{rmk}

Note that the involution $\ast$ on $D$ induces an involution on $C:=\End_D(V)$.\index{$C$}  To the PEL datum $\mathfrak{D}$, we associate an algebraic group $G$ over $\IQ$ whose $R$-points are given by
\begin{align*}
G(R) = \left\{x\in C\otimes_\IQ R\mid xx^\ast \in R^\times \right\}
\end{align*}
for any $\IQ$-algebra $R$.  Note that this agrees with the definition of $G$ coming from the integral PEL data in Equation \eqref{GofR}.  The {\em similitude factor} of $G$ is the homomorphism
\begin{align*}\index{$\nu$}
\nu: G\rightarrow \mathbb{G}_m,
\end{align*}
defined by $g\mapsto g g^\ast$.  We let $G_1$ be the group whose $R$-points are given by
\begin{align*}\index{$G_1$}
G_1(R):=\ker(\nu) = \left\{x\in C\otimes_\IQ R\mid xx^\ast =1 \right\}.
\end{align*}

\subsubsection{Decompositions and signatures associated to PEL data}\label{decompsection}
Following the conventions and perspective of \cite[Section 2.1.5]{EiMa2}, we briefly summarize some key decompositions and define the signature of a PEL datum.  
While these decompositions are basic and the definition of {\em signature} occurs in each of the author's papers in this area, it apparently took many iterations (for this author, at least) to arrive at what feels like an ``optimally'' concise and useful setup for them, hence the citation of a relatively recent paper for this background material.

For any number field $L$, we let $\mathcal{T}_L$\index{$\mathcal{T}_L$ where $L$ denotes a number field} 
denote the set of embeddings $L\hookrightarrow \bar{\IQ}$.  Given $\tau\in \mathcal{T}_L$, we denote its composition with complex conjugation by $\tau^\ast$.\index{$\tau^\ast$}  Going forward, we fix a {\em CM type} $\Sigma_\cmfield$\index{$\Sigma_\cmfield$} for $\cmfield$ (i.e.\ $\Sigma_\cmfield\subseteq \mathcal{T}_\cmfield$ contains a choice of exactly one representative from each pair of complex conjugate embeddings $\tau, \tau^\ast\in \mathcal{T}_\cmfield$).

The decomposition $\cmfield\otimes_\IQ\IC = \oplus_{\tau\in \mathcal{T}_\cmfield}\IC$ (identifying $a\otimes b$ with $(\tau(a)b)_{\tau\in\mathcal{T}_\cmfield}$) induces decompositions
\begin{align*}
V_i &= \oplus_{\tau\in \mathcal{T}_\cmfield}V_{i, \tau}, \mbox{ for $i=1, 2$ }\\
V_\IC &= \oplus_{\tau\in \mathcal{T}_\cmfield}V_\tau,
\end{align*}
with $V_\tau = V_{1, \tau}\oplus V_{2, \tau}$ for all $\tau\in \mathcal{T}_\cmfield$.  Here, the subscript $\tau$ denotes the submodule on which each $a\in \cmfield$ acts as scalar multiplication via $\tau(a)$.

\begin{defi}
The {\em signature} of the unitary PEL datum $\mathfrak{D}$ is $\left(a_\tau\right)_{\tau\in \mathcal{T}_\cmfield}$, where
\begin{align*}\index{$a_\tau$}
a_\tau:=\dim_\IC V_{1, \tau}
\end{align*}
for all $\tau\in \mathcal{T}_\cmfield$.
 We also sometimes speak of the {\em signature} at $\tau\in \Sigma_\cmfield$ or at $\sigma\in \mathcal{T}_{\realfield}$, by which we mean $(a_\tau, a_{\tau^\ast})$, given the unique $\tau\in \Sigma_\cmfield$ such that $\tau|_{\realfield} = \sigma$.  For such $\tau\in \Sigma_\cmfield$, we also sometimes set $a_\sigma^+:=a_\tau^+:=a_\tau$ and $a_\sigma^-:=a_\tau^-:=a_{\tau^\ast}$.  It is also common to write $(a_\tau, b_\tau)$\index{$b_\tau$} in place of $(a_\tau, a_{\tau^\ast})$.
\end{defi}
Note that for each $\tau\in \mathcal{T}_\cmfield$, we have
\begin{align*}
a_\tau+a_{\tau^\ast} = n.
\end{align*}

More generally, we record some basic facts about decompositions of modules that will be useful to us later.  Given a number field $L$, we denote by $L^\Gal$\index{$\Gal$ in the superscript of a number field denotes the Galois closure of a number field} the Galois closure of $L$ in $\bar{\IQ}$, and we denote by $\CO_L$\index{$\CO_L$ the ring of integers in a number field $L$} the ring of integers in $L$.  If $R$ is an $\CO_{{\realfield}^\Gal}$-algebra and the discriminant of $\realfield/\IQ$ is invertible in $R$, then we have an isomorphism
\begin{align*}
\CO_{\realfield}\otimes_\ZZ R&\isomto \oplus_{\tau\in \mathcal{T}_{\realfield}}R\\
a\otimes r&\mapsto(\tau(a) r)_{\tau \in \mathcal{T}_{\realfield}}.
\end{align*}
Given an $\CO_{\realfield}\otimes_\ZZ R$-module $M$ and $\tau\in \mathcal{T}_{\realfield}$, we denote by $M_\tau$ the submodule on which each $a\in \CO_{\realfield}$ acts as multiplication by $\tau(a)$, and we have an $\CO_{\realfield}\otimes_\ZZ R$-module isomorphism
\begin{align*}
M\isomto\oplus_{\tau\in \mathcal{T}_{\realfield}}M_\tau
\end{align*}
If $R$ is, furthermore, an $\CO_{\cmfield^\Gal}$-algebra, then we can further decompose $M$ as
\begin{align*}
M = \oplus_{\tau\in\mathcal{T}_\cmfield}M_\tau = \oplus_{\tau\in\Sigma_\cmfield}M_\tau\oplus M_{\tau^\ast} = \oplus_{\sigma\in\mathcal{T}_{\realfield}}M_\sigma^+\oplus M_\sigma^-,
\end{align*}
where for each $\tau\in \mathcal{T}_\cmfield$, $M_\tau$ denotes the submodule of $M$ on which each element $a\in \OK$ acts via scalar multiplication by $\tau(a)$, and for each $\sigma\in \mathcal{T}_{\realfield}$, $M_\sigma^+$ (resp. $M_\sigma^-$) is the submodule of $M_\sigma$ on which each element $a\in \CO_\cmfield$ acts as multiplication by $\tau(a)$ (resp. $\tau^\ast (a)$) for $\tau\in \Sigma_\cmfield$ the unique element of $\Sigma_\cmfield$ such that $\tau|_{\realfield} = \sigma$.  For such $\tau\in\Sigma_\cmfield$, we also sometimes write $M_\tau^\pm$ in place of $M_\sigma^\pm$.  We also set
\begin{align}\label{equ:pmconvention}
M^\pm=\oplus_{\sigma\in\mathcal{T}_{\realfield}}M_\sigma^\pm.
\end{align}

\subsubsection{PEL data arising from unitary groups}\label{unitarycon}
Following the conventions of \cite[Section 2.2]{HELS}, we say that a PEL datum $\mathfrak{D}$ like above is {\em of unitary type} if the following three conditions hold:
\begin{itemize}
\item{$D = \cmfield\times\cdots \times \cmfield$, i.e.\ $D$ is a direct product of finitely many copies of $\cmfield$}
\item{$\ast$ acts as complex conjugation on each factor $\cmfield$ in $D = K\times\cdots \times K$}
\item{$\CO_D\cap \cmfield$ is the ring of integers $\CO_\cmfield$ in $\cmfield$, where $\cmfield$ is identified with its diagonal embedding in $D = K\times\cdots \times K$}
\end{itemize}

Fix a totally imaginary element $\alpha$ of $\CO_\cmfield$.    
As explained in \cite[Section 2.3]{HELS}, given a collection of $m$ unitary groups preserving Hermitian pairings $\langle, \rangle_{W_1}, \ldots, \langle, \rangle_{W_m}$ on $\cmfield$-vector spaces $W_1, \ldots, W_m$, respectively, we obtain a PEL datum $\mathfrak{D} = \left(D, \ast, \CO_D, V, \langle, \rangle, L, h \right)$ of unitary type as follows:
\begin{itemize}
\item{Let $D = \cmfield^m$.}
\item{Let $\ast$ be the involution on $D$ that acts as complex conjugation on each factor $\cmfield$.}
\item{Let $\CO_D = \CO_\cmfield^m\subseteq \cmfield^m$.}
\item{Let $V = \oplus_i W_i$.}
\item{Let $\langle (v_1, \ldots, v_m), (w_1, \ldots w_m)\rangle = \sum_i\langle v_i, w_i\rangle_i$, where 
\begin{align*}
\langle, \rangle_i:=\tr_{\cmfield/\IQ}\left(\alpha\langle, \rangle_{V_i}\right).
\end{align*}}
\item{Let $L = \oplus L_i$, where $L_i\subseteq V_i$ is an $\CO_\cmfield$-lattice such that $\langle L_i, L_i\rangle \subseteq\ZZ$.}
\item{Let $h=\prod_i h_i: \IC\rightarrow  \End_{\realfield\otimes_\IQ\IR}(V\otimes_\IQ\IR) = \prod_i \End_{\realfield\otimes_\IQ\IR}(W_i\otimes_\IQ\IR)$, where
\begin{align*}
h_i: \IC\rightarrow \End_{\realfield\otimes_\IQ\IR}(W_i\otimes_\IQ\IR)
\end{align*}
is defined by $h_i = \prod_{\tau \in \Sigma_\cmfield}h_{i, \tau}: \IC\rightarrow \End_{\realfield\otimes_\IQ\IR}(V_i\otimes_\IQ\IR) =  \prod_{\tau \in \Sigma_\cmfield}\End_{\IR}(V_i\otimes_{\cmfield, \tau}\IC)$
and
\begin{align*}
h_{i, \tau}:\IC\rightarrow \End_{\IR}(V_i\otimes_{\cmfield, \tau}\IC)
\end{align*}
is defined as follows.  Choose an ordered basis
\begin{align}\label{orderedbasisB}
\mathscr{B}
\end{align}
 for $W_i\otimes_{\cmfield, \tau}\IC$ with respect to which the matrix for $\langle, \rangle$ is of the form $\diag(1_r, -1_s)$ (with $r$ and $s$ dependent on $i$ and $\tau$), and identify $W_i\otimes_{\cmfield, \tau}\IC$ with $\IC^{r+s}$, as well as $\End_{\IR}(W_i\otimes_{\cmfield, \tau}\IC)$ with $M_{(r+s)\times(r+s)}(\IC)$, via this choice of basis.  We then define
\begin{align*}
h_{i, \tau}(z) = \diag(z1_r, \bar{z}1_s)
\end{align*}
for each $z\in\IC$.
 }
\end{itemize}

\subsubsection{Moduli problem}\label{sec:moduliproblem}
In this section, we will describe a moduli problem that classifies abelian varieties together with the structure of a {\em polarization}, {\em endomorphism}, and {\em level structure}. The letters {\em PEL} above are an abbreviation for {\em Polarization, Endomorphism, Level structure}, part of the data in this moduli problem.  The moduli problem here is analogous to the one encountered in the context of modular forms, where we classify elliptic curves with level structure.  To define a moduli space in our setting, we also need to include a polarization and endomorphism, although these extra structures will generally not come into play much later in this manuscript.
Our formulation here most closely follows the conventions in \cite[Section 2.2]{CEFMV}, \cite[Section 2.1]{HELS}, and \cite[Section 5.1]{lan}.  

Let $G$ be as in Equation \eqref{GofR}, and let $\cpct$\index{$\cpct$} be an open compact subgroup of $G(\adeles_f)$, where $\adeles_f$\index{$\adeles_f$} denotes the finite adeles in the adeles $\adeles$\index{$\adeles$} over $\IQ$.  

We first consider the moduli problem that associates to each pair $(S, s)$ consisting of a connected, locally Noetherian scheme $S$ over $\reflex$ and a geometric point $s$ of $S$ the set of {\em equivalence classes} of tuples $(A, \lambda, \iota, \eta)$ consisting of:
\begin{itemize}
\item{An abelian variety $A$ over $S$ of dimension $g:=n[\realfield: \IQ]$ (where $n$ is the dimension of the $\cmfield$-vector space $V$ in the PEL datum $\mathfrak{D}$)}
\item{A polarization $\lambda: A\rightarrow A^\vee$ (where $A^\vee$ denotes the dual abelian variety)}
\item{An embedding $\iota: \CO_D\otimes_\ZZ\IQ\hookrightarrow \End(A)\otimes_\ZZ\IQ$ of $\IQ$-algebras that satisfies the {\em Rosati condition}, i.e.\ 
\begin{align*}
\lambda\circ\iota(b^\ast) = (\iota(b))^\vee \circ\lambda
\end{align*}
for all $b\in\CO_D$
}
\item{A $\cpct$-level structure $\eta$, i.e.\ a $\pi_1(S, s)$-fixed orbit of $\CO_D$-linear isomorphisms 
\begin{align*}
L\otimes_\ZZ\adeles_f\isomto H_1(A, \adeles_f)
\end{align*}
that maps $\langle, \rangle$ to a $\adeles_f^\times$-multiple of the pairing on $H_1(A, \adeles_f)$ defined by the $\lambda$-Weil pairing
}
\end{itemize} 
In addition, we require that the above tuples satisfy {\em Kottwitz's determinant condition} (as explained in, for example, \cite[Section 5]{kottwitz}, \cite[Definition 1.3.4.1]{la}, and \cite[Section 2.2]{CEFMV}).  Tuples $(A, \lambda,  \iota,\eta)$ and $(A', \lambda', \iota', \eta')$ are considered {\em equivalent} if there is an isogeny $\phi: A\rightarrow A'$ such that $\lambda$ is a nonzero rational multiple of $\phi^\vee\circ\lambda'\circ\phi$, $\iota'(b)\circ \phi = \phi\circ \iota(b)$ for all $b\in \CO_D$, and $\eta' = \phi\circ \eta$.

If $\cpct$ is sufficiently small (or more precisely, if $\cpct$ is {\em neat}, in the sense of \cite[Definition 1.4.1.8]{la}), then this moduli problem is representable by a smooth, quasi-projective scheme $\moduli_\cpct$\index{$\moduli_\cpct$} over $\reflex$.  This is a result of \cite[Corollary 7.2.3.10]{la} and \cite[Section 5]{kottwitz}, which also explain that a $p$-integral version of this moduli problem is representable by a smooth, quasi-projective scheme over $\CO_\reflex\otimes\ZZ_{(p)}$.  (N.B.\ Given a scheme $S = \Spec(R)$, we sometimes write $R$ in place of $\Spec(R)$ when the meaning is clear from context.)  More precisely, assume that in addition to the conditions above, $\cpct = \cpct^p\cpct_p$ with $\cpct_p\subseteq G(\IQ_p)$ hyperspecial.  (i.e.\ We require that there exists a smooth group scheme $\mathcal{G}$ that is a model of $G$ over $\ZZ_p$, such that the special fiber $\mathcal{G}$ is reductive, and such that $\cpct_p = \mathcal{G}(\ZZ_p)$, as discussed in much more detail in \cite[Section 2.4]{getz}.  For geometric motivation for the condition of being hyperspecial, see also \cite{milne, lan}.  For the origins, see Tits's original article in \cite{tits}.)  Then there is a smooth, quasi-projective scheme $\moduliint_\cpct$\index{$\moduliint_\cpct$} over $\CO_\reflex\otimes\ZZ_{(p)}$ that to each pair $(S, s)$ consisting of a connected, locally Noetherian scheme $S$ over $\CO_{\reflex, (p)}$ and a geometric point $s$ of $S$ associates the set of {\em equivalence classes} of tuples $(A,  \lambda, \iota, \eta)$ (where tuples $(A, \lambda, \iota, \eta)$ and $(A', \lambda', \iota', \eta')$ are considered equivalent if there is a prime-to-$p$ isogeny $\phi$ meeting the conditions from above and furthermore $\lambda$ is a nonzero prime-to-$p$ multiple of $\phi^\vee\circ\lambda'\circ\phi$) consisting of:
\begin{itemize}
\item{An abelian variety $A$ over $S$ of dimension $g$}
\item{A prime-to-$p$ polarization $\lambda: A\rightarrow A^\vee$ (where $A^\vee$ denotes the dual abelian variety)}
\item{An embedding $\iota: \CO_D\otimes_\ZZ\ZZ_{(p)}\hookrightarrow \End(A)\otimes_\ZZ\ZZ_{(p)}$ of $\ZZ_{(p)}$-algebras that satisfies the Rosati condition
}
\item{A $\cpct^p$-level structure $\eta$, i.e.\ a $\pi_1(S, s)$-fixed orbit of $\CO_D$-linear isomorphisms 
\begin{align*}
L\otimes_\ZZ\adeles^{p, \infty}\isomto H_1\left(A, \adeles^{p, \infty}\right)
\end{align*}
that maps $\langle, \rangle$ to a $(\adeles^{\infty, p})^\times$-multiple of the pairing on $H_1(A, \adeles_f)$ defined by $\lambda$-Weil pairing
}
\end{itemize}
(The notation $\adeles^{p, \infty}$\index{$\adeles^{p, \infty}$} means the adeles away from $p$ and $\infty$.)  Once again, these tuples are also required to satisfy {\em Kottwitz's determinant condition} (as explained in, for example, \cite[Section 5]{kottwitz}, \cite[Definition 1.3.4.1]{la}, and \cite[Section 2.2]{CEFMV}).  
Note that $\moduliint_\cpct\times_{\Spec \CO_{\reflex, (p)}} \Spec\reflex = \moduli_\cpct$.

\begin{rmk}
The condition that $\cpct$ is {\em neat} guarantees the representability of our moduli problem by a smooth moduli space.  For the purposes of this manuscript, the details of what it means to be neat are unimportant.  For the sake of completeness, though, we briefly recall from \cite[Definition 1.4.1.8]{la} that $\cpct$ is defined to be neat if each of its elements $g = \left(g_p\right)$ is neat.  That is, the group $\cap_p \epsilon_p$, where $\epsilon_p$ denotes the group of algebraic eigenvalues of $g_p$ (viewed as an element of $GL(L\otimes\IQ_p)$), is torsion free.
\end{rmk}

\subsubsection{Compactifications}
The moduli spaces $\moduliint_\cpct$ have toroidal compactifications (as proved in \cite{la}), over which one can define modular forms.  There are also minimal compactifications of the spaces $\moduliint_\cpct$, as constructed in \cite{la} and summarized in \cite{lan}.  See \cite[Section 5.1.4]{lan} for a summary of key developments for compactifications leading up to Lan's work on toroidal and minimal compactifications in \cite{la}, in particular connections with the earlier work of Gerd Faltings and Ching-Li Chai in the setting of Siegel moduli problems in \cite[Chapters III--V]{FC}.

\subsubsection{Complex points and connection with Shimura varieties}\label{shconnsec}
We briefly summarize the connection between complex points of our moduli space $\moduli_\cpct$ and complex points of unitary Shimura varieties.  Let $\mathcal{H}$ be the orbit of $h$ under conjugation by $G(\IR)$, and let \index{$\cpct_\infty$}$\cpct_\infty\subseteq G(\IR)$ be the centralizer of $h$.  We give $\mathcal{H}$ the structure of a real manifold via the identification $G(\IR)/\cpct_\infty$ with $\mathcal{H}$ via $g\mapsto ghg^{-1}$.  Furthermore, $h$ induces a complex structure on $\mathcal{H}$.  When $(G, \mathcal{H})$ is a Shimura datum (in the sense of \cite[Section 2.3]{lan}), we say that $(G, \mathcal{H})$ is a {\em PEL-type Shimura datum}, and the Shimura variety $\Sh_\cpct$\index{$\Sh_\cpct$} associated to $(G, \mathcal{H})$ is called a {\em PEL-type Shimura variety}.  (See \cite[Section 2.4]{lan} for an excellent introduction to Shimura varieties.)

\begin{rmk}
As noted in \cite[Section 5.1.3]{lan}, it is not necessarily the case that $(G, \mathcal{H})$ is a Shimura datum.  When $\CO$ is an order in an imaginary quadratic field and $a\geq b$, though, $(G, \mathcal{H})$ is a Shimura datum, $G(\IR)\cong GU(a, b)$, and $\mathcal{H}\cong \mathcal{H}_{a, b}$,
where
\begin{align}\label{equ:Habdef}\index{$\mathcal{H}_{a, b}$}
\mathcal{H}_{a, b} := \left\{z\in \Mat_{a\times b}(\IC)\mid 1-{ }^t\bar{z}z >0\right\},
\end{align}
with $>0$\index{$>0$ means positive definite} meaning positive definite and $\Mat_{a\times b}$\index{$\Mat_{a\times b}$} meaning $a\times b$ matrices.
For details, see \cite[Example 5.1.3.5]{lan}.
\end{rmk}

Note that the complex points of $\Sh_\cpct$ can be identified with the double coset space
\begin{align*}\index{$X_\cpct$}
X_\cpct:= G(\IQ)\backslash \left(\mathcal{H}\times G\left(\adeles_f\right)/\cpct\right),
\end{align*}
with $G(\IQ)$ acting diagonally on $\mathcal{H}$ and $G\left(\adeles_f\right)$ on the left and $\cpct$ acting on $G\left(\adeles_f\right)$ on the right, as detailed in \cite[Section 8]{kottwitz} and \cite[Section 2]{lan}.  Even if $(G, \mathcal{H})$ is not a Shimura datum, there is an open and closed immersion
\begin{align*}
X_\cpct\hookrightarrow \moduliint_\cpct(\IC),
\end{align*}
and there is an open and closed subscheme of $\moduliint_\cpct$ that is an integral model of $\Sh_\cpct$.  In our setting (i.e.\ the unitary setting), $\moduliint_\cpct(\IC)$ is a disjoint union of finitely many copies of $X_\cpct$, as explained in \cite[Section 2.3.2]{CEFMV} and \cite[Section 5.1.3]{lan}.

Following the conventions for terminology introduced in \cite[Section 2.3]{HELS}, we call $\moduliint_\cpct$ the {\em moduli space (of PEL-type)} and $\Sh_\cpct$ the {\em Shimura variety (of PEL type)}.  Note that due to our insistence that the center of $D$ be a CM field, we have actually narrowed the set of cases under consideration in this manuscript to the unitary setting (what is often called Type A), rather than the broader PEL setting (which includes symplectic and orthogonal groups as well).

\begin{rmk}
Observe that each element $h\in G(\adeles_f)$ acts on the right on $\mathcal{H}\times G(\adeles_f)$ via $(z, g)\mapsto (z, gh)$, which induces a map 
\begin{align}\index{$[h]$}\label{equ:bracketh}
[h]: X_{h\cpct h^{-1}}\rightarrow X_\cpct.
\end{align}
  In turn, this provides a right action of $G(\adeles_f)$ on the collection $\left\{X_\cpct\right\}_\cpct$, which is useful for relating these geometric spaces to automorphic representations.  (This provides some motivation for this adelic formulation, which might otherwise seem unnecessary at first glance.)
\end{rmk}

We briefly review the structure of the double coset $X_\cpct$.  For more details, see, e.g.\, \cite[Section 2.2]{lan}.  It turns out that we can express $G(\adeles_f)$ as a finite disjoint union
\begin{align*}
G(\adeles_f) = \sqcup_{i\in I} G(\IQ)^+g_i\cpct,
\end{align*}
with $g_i\in G(\adeles_f)$ indexed by a finite set $I$ and $G(\IQ)^+:=G(\IQ)\cap G(\IR)^+$ (with $G(\IR)^+$ the connected component of the identity), and let $\mathcal{H}^+$\index{$\mathcal{H}^+$} be a connected component of $\mathcal{H}$ on which $G(\IR)^+$ acts transitively.  Then (as in \cite[Equations (2.2.21)]{lan}), we have
\begin{align*}
X_\cpct &= G(\IQ)^+\backslash \left(\mathcal{H}^+\times G(\adeles_f)/\cpct\right)\\
&= \sqcup_{i\in I}G(\IQ)^+\backslash \left(\mathcal{H}^+\times G(\IQ)^+g_i\cpct/\cpct\right)\\
&=\sqcup_{i\in I}\Gamma_i\backslash \mathcal{H}^+,
\end{align*}
where
\begin{align*}\index{$\Gamma_i$}
\Gamma_i:=\left(g_i\cpct g_i^{-1}\right)\cap G(\IQ)^+.
\end{align*}
As explained in \cite[Section 2.2]{lan}, each $\Gamma_i$ is a congruence subgroup of $G(\ZZ)$, i.e.\ $\Gamma_i$ contains the {\em principal congruence subgroup} (i.e.\ the kernel of $G(\ZZ)\rightarrow G(\ZZ/N\ZZ)$) for some positive integer $N$.

\begin{rmk}\label{PELclassificationH}
In the case of the usual upper half plane $\mathfrak{h}$ and elliptic curves, we associate to each point $z\in \mathfrak{h}$, an elliptic curve whose complex points are identified with the torus $\IC/(\ZZ z+\ZZ)$.  The analogue in our setting is as follows (and discussed in detail in \cite[Section 4]{shar}).  Given a unitary PEL datum like above and $z\in \mathcal{H}$, Shimura defines a map $p_z$\index{$p_z$} (see \cite[Section 4.7]{shar} or \cite[Section 2.3.2]{EDiffOps}) so that $\IC^g/p_z(L)$ is an abelian variety, and he explains (in \cite[Theorem 4.8]{shar}) how to assign a polarization, endomorphism, and level structure so that $\Gamma_\cpct\backslash \mathcal{H}$ classifies PEL tuples with level structure corresponding to $\cpct$.  (Here, \index{$\Gamma_\cpct$}$\Gamma_\cpct:=\cpct\cap GU^+(\IQ)$.)  For additional details on the moduli problem over $\IC$, see also \cite[Section 2.3.2]{EDiffOps} or \cite[Sections 3.1.5 and 3.2]{lan}.  We give a more precise description of the lattice $p_z(L)$ in Remark \ref{nnclass}, for signature $(n,n)$.
\end{rmk}

\subsubsection{Hermitian symmetric space associated to a unitary group}\label{sec:hermsymmsp}
The symmetric domain $\mathcal{H}$ for $G$ as above (with signature $(a_\tau, b_\tau)$ for each $\tau\in \Sigma_\cmfield$) is
\begin{align*}
\prod_{\tau\in \Sigma_{\cmfield}}\mathcal{H}_{a_\tau, b_\tau},
\end{align*}
where $\mathcal{H}_{a, b}$ is defined as in Equation \eqref{equ:Habdef}.
The elements $g  = (g_\tau)_{\tau\in\Sigma_{\cmfield}}\in \prod_{\tau\in\Sigma_{\cmfield}}GU^+(a_\tau, b_\tau)$ (where the superscript $+$ denotes positive determinant) act on $z = (z_\tau)_{\tau\in \Sigma_{\cmfield}}\in\prod_{\tau\in\Sigma_{\cmfield}}\mathcal{H}_{a_\tau, b_\tau}$ by
\begin{align*}
g z = \left(g_\tau z_\tau\right)_{\tau\in \Sigma_{\cmfield}},
\end{align*}
where for
\begin{align*}
g_\tau=\begin{pmatrix}a & b\\ c& d\end{pmatrix}
\end{align*}
with $a\in \GL_{a_\tau}(\IC)$ and $d\in GL_{b_\tau}(\IC)$, we have
\begin{align*}
g_\tau z_\tau := (az_\tau+b)(cz_\tau +d)^{-1}.
\end{align*}
The stabilizer of $0\in \mathcal{H}_{a_\tau, b_\tau}$ is the product of definite unitary groups $U(a_\tau)\times U(b_\tau)$ embedded diagonally in $U(a_\tau, b_\tau)$.  So we can identify $\mathcal{H}_{a_\tau, b_\tau}$ with $U(a_\tau, b_\tau)/(U(a_\tau)\times U(b_\tau)).$

In the case of definite unitary groups (i.e.\ $a_\tau b_\tau = 0$ for all $\tau$), $\mathcal{H}=\prod_{\tau\in \Sigma_{\cmfield}}\mathcal{H}_{a_\tau, b_\tau}$ consists of a single point, which parametrizes an isomorphism class of abelian varieties isogenous to $a_\tau+b_\tau$ (which is the same for all $\tau$) copies of a CM abelian variety, as explained in \cite[Section 4.8]{shar}.

In the special case where $a_\tau = b_\tau = n$ (i.e.\ the case considered by Hel Braun when she introduced Hermitian modular forms in \cite{braun1, braun2, braun3}), it is often convenient to work with an unbounded realization of the space $\mathcal{H}$, namely Hermitian upper half space:
\begin{align}\index{$\mathcal{H}_n$}
\mathcal{H}_n&:=\left\{Z\in \Mat_{n\times n}(\IC) \mid i({ }^t\bar{Z}-Z)>0\right\},\nonumber\\
&=\left\{Z\in \Mat_{n\times n}(\IC)=\Herm_n(\IC)\otimes_\IR\IC \mid \im(Z)>0\right\}\label{herm2equ}
\end{align}
where $\Herm_n$ denotes $n\times n$ Hermitian matrices and $\im(Z)$ denotes the {\em Hermitian imaginary part}.  
\begin{rmk}
The equality in Equation \eqref{herm2equ} follows from the identification
\begin{align*}
\Mat_{n\times n}(\IC)\cong\Herm_n(\IC)\otimes_\IR\IC\\
Z\mapsto \re(Z)+i \im(Z)
\end{align*}
with $\re(Z)$ and $\im(Z)$ the Hermitian real and imaginary parts, respectively, i.e.\
\begin{align*}
\re(Z)&:= \frac{1}{2}(Z+{ }^t\bar{Z})\\
\im(Z)&:=\frac{1}{2i}(Z-{ }^t\bar{Z}).
\end{align*}
\end{rmk}
The action here of $GU(\eta_n)$ on $\mathcal{H}_n$ is similar to the action given above on $\mathcal{H}_{a, b}$, i.e.\ $g= \begin{pmatrix} A & B\\ C& D\end{pmatrix}$ acts on $z\in \mathcal{H}_n$ by
\begin{align*}
gz = (Az+B)(Cz+D)^{-1}.
\end{align*}
The point $0$ in $\mathcal{H}_{n,n}$ corresponds to the point $i1_n\in \mathcal{H}_n$.  Note that the stabilizer of $i1_n\in \mathcal{H}_{n}$ is the product of definite unitary groups $U(n)\times U(n)$ embedded diagonally in $U(n,n)$.  So we can identify $\mathcal{H}_{n, n}$ with $U(n, n)/(U(n)\times U(n)).$
\begin{rmk}\label{nnclass}
If we choose a basis $\mathscr{B}$ as in \eqref{orderedbasisB} and the lattice $L$ is as in Section \ref{unitarycon}, then if the signature at each place is $(n,n)$, we can define $p_z: L^{[\realfield:\IQ]}\rightarrow\IC^g$ from Remark \ref{PELclassificationH} by
\begin{align*}
p_z(x) = \left([z_\tau \hspace{0.1in} 1_n]{ }^t\bar{x}, [{ }^t z_\tau  \hspace{0.1in}1_n]{ }^tx\right)_{\tau\in\mathcal{T}_{\realfield}}
\end{align*}
for each $x\in L$ and $z = \left(z_\tau\right)_\tau\in\mathcal{H}$ (with $ [ z_\tau  \hspace{0.1in}1_n]$ and $ [{ }^t z_\tau  \hspace{0.1in}1_n]$ denoting $n\times 2n$ matrices).
\end{rmk}

\subsection{Weights and representations}\label{sec:weightsandreps}

We briefly summarize key information about weights and representations, following \cite[Section 2.2]{EiMa2} and \cite[Sections 2.3--2.5]{EiMa}.  We will use the conventions established here in our definitions of automorphic forms on unitary groups in Section \ref{autformsunitarygpssection}.

The setup here is also similar to \cite[Sections 2.1--2.2]{EFGMM}, and \cite[Sections 2.3--2.4]{EFMV}.  For a more detailed treatment, the reader might consult \cite[Chapter II.2]{jantzen} or \cite[Sections 4.1 and 15.3]{fultonharris}.  Let $\Levi$\index{$\Levi$} be the subgroup of $G_1(\IC)$ that preserves the decomposition $V_\IC = V_1\oplus V_2$.  Let $\Borel$\index{$\Borel$} be a Borel subgroup of $\Levi$, and let $\Torus\subset\Borel$ be a maximal torus.  A choice of basis for $V_\IC$ that preserves the decomposition $V_\IC = V_1\oplus V_2$ identifies $\Levi$ with $\prod_{\tau\in \Sigma_{\cmfield}}\left(\GL_{a_\tau}\times \GL_{a_{\tau^\ast}}\right) = \prod_{\tau\in\mathcal{T}_\cmfield}\GL_{a_\tau}$.  We write $\Levi = \prod_\tau\Levi_\tau$, $\Borel = \prod_\tau\Borel_\tau$, and $\Torus = \prod_\tau\Torus_\tau$ (with each of these products over $\tau\in \mathcal{T}_\cmfield$).  We choose such a basis so that, furthermore, $\Borel_\tau$ is identified with the group of upper triangular matrices in $\GL_{a_\tau}$ for each $\tau\in \mathcal{T}_\cmfield$ and $T_\tau=\mathbb{G}_m^{a_\tau}$ is identified with the group of diagonal matrices in $\GL_{a_\tau}$ for each $\tau\in\mathcal{T}_\cmfield$.  Let $\Nilp$\index{$\Nilp$} denote the unipotent radical of $\Borel$.

We denote by $X^\ast:=X^\ast(\Torus)$\index{$X^\ast$} the group of characters of $\Torus$.  Via the isomorphism $\Borel/\Nilp\cong \Torus$, we also view $X^\ast$ as characters of $\Borel$.  We define
\begin{align*}\index{$X^+$}
X^+:=X^+(\Torus):=\left\{(\kappa_{1, \tau}, \ldots, \kappa_{a_\tau, \tau})_{\tau\in \mathcal{T}_\cmfield}\in \prod_{\tau\in\mathcal{T}_\cmfield}\ZZ^{a_\tau, \tau}\mid \kappa_{i, \tau}\geq \kappa_{i+1, \tau} \mbox{ for all } i\right\}.
\end{align*}
We identify each tuple $\kappa = (\kappa_{1, \tau}, \ldots, \kappa_{a_\tau, \tau})_{\tau\in \mathcal{T}_\cmfield}\in X^+$ with the {\em dominant weight} in $X^\ast$ defined by
\begin{align*}
\prod_{\tau\in\mathcal{T}_\cmfield}\diag(t_{1, \tau}, \ldots, t_{a_\tau, \tau})\mapsto \prod_{\tau\in\mathcal{T}_\cmfield}\prod_{i=1}^{a_\tau}t_{i, \tau}^{\kappa_{i, \tau}}.
\end{align*}
If $\kappa_{i, \tau}\geq 0$ for all $i, \tau$ and, furthermore $\kappa_{i, \tau}>0$ for some $i, \tau$, then we call $\kappa$ {\em positive}.

Suppose $R$ is a $\ZZ_p$-algebra or a field of characteristic $0$.  To each dominant weight $\kappa$, there is an associated representation \index{$\rho_\kappa = \rho_{\kappa, R}$}$\rho_\kappa = \rho_{\kappa, R}$ of $\Levi(R)$, which is obtained by application of a {\em $\kappa$-Schur functor} $\mathbb{S}_\kappa$\index{$\mathbb{S}_\kappa$} (as explained in, e.g.\, \cite[Section 15.3]{fultonharris} and summarized in \cite[Section 2.4.2]{EFMV}).  We call $\rho_\kappa$ the {\em representation of (highest) weight $\kappa$}. If $R$ is of characteristic $0$ or of sufficiently large characteristic, the representation $\rho_\kappa$ is irreducible (as explained in, e.g.\, \cite[Chapter II.2]{jantzen}).  Given a locally free sheaf of modules $\mathcal{F}$ over a scheme $T$, we denote by $\mathbb{S}_\kappa(\mathcal{F})$ the locally free sheaf of modules defined by $\mathbb{S}_\kappa(\mathcal{F)}(\Spec R): = \mathbb{S}_\kappa(\mathcal{F}(\Spec R))$ for each affine open $\Spec R$ of $T$.
\begin{rmk}
In practice, the representations $\rho_\kappa$ with which we work are built from compositions of tensor products, symmetric products, and alternating products.  
For an explicit description over $\IC$ of the highest weight representations and highest weight eigenvectors, see \cite[Section 12.6]{shar}.  More generally, see also \cite[Remark 3.5]{CEFMV}, which concerns any field of characteristic $0$.  In that setting, $\rho_\kappa$ is isomorphic to 
\begin{align*}
\Ind_{\Borel^-}^\Levi(-\kappa) = \left\{f: N^-\backslash \Levi\rightarrow \mathbb{A}^1 \mid f(th) = \kappa(t)^{-1}f(h) \mbox{ for all } t\in \Torus\right\},
\end{align*}
where $\Borel^-$ is the opposite Borel, identified with the group of lower triangular matrices, and $N^-$ is its unipotent radical.  Here, $\rho(g)f(h): = f(hg)$ for all $g, h\in \Levi$.
\end{rmk}

\section{Automorphic forms on unitary groups}\label{autformsunitarygpssection}
In analogue with the case of modular forms on $\GL_2$, there are various ways to define automorphic forms on unitary groups.  Each formulation has its own merits, depending on what one is trying to do.  We introduce the following viewpoints: analytic functions on the symmetric spaces $\mathcal{H}$ (Section \ref{sec:hermaut}), functions on abelian varieties with PEL structure (Section \ref{sec:abaut}), global sections of a sheaf on $\moduli_\cpct$ (Section \ref{sec:sheafaut}), and functions on a unitary group (over $\IR$ in Section \ref{sec:Raut} and adelically in Section \ref{sec:adeleaut}).  We conclude the section with brief discussions of $q$-expansions (Section \ref{sec:qexp}) and automorphic representations (Section \ref{sec:autrep}).

\subsection{As analytic functions on Hermitian symmetric spaces}\label{sec:hermaut}
We now introduce automorphic forms as analytic functions on Hermitian symmetric spaces, a generalization of the formulation of modular forms as functions on the upper half plane.  This perspective was first introduced by Braun in \cite{braun1, braun2, braun3} and was widely employed by Shimura, e.g.\ in  \cite[Section 5]{shar}.

In this section, we set
\begin{align*}
\mathcal{H} &= \prod_{\tau\in\Sigma_{\cmfield}}\mathcal{H}_{m_\tau, n_\tau}\\
GU_{m,n}^+(\IR)&:=\prod_{\tau\in\Sigma_{\cmfield}}GU^+(m_\tau, n_\tau).\index{$GU_{m,n}^+$}
\end{align*}

\begin{defi}
Suppose $mn\neq 0$.
The {\em factors of automorphy} or {\em automorphy factors} associated to $g = \left(g_\tau\right)_{\tau\in\Sigma_\cmfield}\in GU_{m,n}^+(\IR)$, with $g_\tau= \begin{pmatrix}a_\tau& b_\tau\\ c_\tau& d_\tau\end{pmatrix}$ and $z= \left(z_\tau\right)_{\tau\in\Sigma_\cmfield}\in \mathcal{H}$ at a place $\tau\in \Sigma_\cmfield$ are:
\begin{itemize}\index{$\lambda_{g_\tau}(z_\tau) = \lambda_\tau(g, z) = \lambda(g_\tau, z_\tau)$}\index{$\mu_{g_\tau}(z_\tau) = \mu_\tau(g, z) = \mu(g_\tau, z_\tau)$}
\item{$\lambda_{g_\tau}(z_\tau): = \lambda_\tau(g, z) := \lambda(g_\tau, z_\tau):=(\overline{b_\tau}{ }^tz_\tau+\overline{a_\tau})$ and $\mu_{g_\tau}(z_\tau) = \mu_\tau(g, z) = \mu(g_\tau, z_\tau):=(c_\tau z_\tau+d_\tau)$ if we work with the bounded domain $\mathcal{H} = \mathcal{H}_{m,n}$}
\item{$\lambda_{g_\tau}(z_\tau): = \lambda_\tau(g, z) := \lambda(g_\tau, z_\tau):=(\overline{c_\tau}{ }^tz_\tau+\overline{d_\tau})$ and $\mu_{g_\tau}(z_\tau) = \mu_\tau(g, z) = \mu(g_\tau, z_\tau):=(c_\tau z_\tau+d_\tau)$ if we work with the unbounded domain $\mathcal{H} = \mathcal{H}_{n}$ (and require $m=n$ at all places $\tau$)}
\end{itemize}
\end{defi}
\begin{defi}
Suppose now that $mn = 0$.  The {\em factors of automorphy} or {\em automorphy factors} associated to $g = \begin{pmatrix}a& b\\ c& d\end{pmatrix}\in GU_{m,n}^+(\IR)$ and $z\in \mathcal{H}$ at a place $\tau\in \Sigma_\cmfield$ are:
\begin{itemize}
\item{$\lambda_{g_\tau}(z_\tau): = \lambda_\tau(g, z) := \lambda(g_\tau, z_\tau):=\overline{g_\tau}$ and $\mu_{g_\tau}(z_\tau) = \mu_\tau(g, z) = \mu(g_\tau, z_\tau):=1$ if $n=0$}
\item{$\lambda_{g_\tau}(z_\tau): = \lambda_\tau(g, z) := \lambda(g_\tau, z_\tau):=1$ and $\mu_{g_\tau}(z_\tau) = \mu_\tau(g, z) = \mu(g_\tau, z_\tau):=g$ if $m=0$}
\end{itemize}
\end{defi}
\begin{defi}The {\em scalar factor of automorphy} or {\em scalar automorphy factor} associated to $g = \begin{pmatrix}a& b\\ c& d\end{pmatrix}\in GU_{m,n}^+(\IR)$ and $z\in \mathcal{H}$ at a place $\tau\in \Sigma_\cmfield$ is:
\begin{align*}\index{$j_{g_\tau}(z_\tau)=j_\tau(g, z) = j(g_\tau, z_\tau)$}
j_{g_\tau}(z_\tau):=j_\tau(g, z): = j(g_\tau, z_\tau):=\det(\mu_\tau(g, z)).
\end{align*}
\end{defi}
\begin{rmk}
Note that
\begin{align*}
\det(\lambda_\tau(g,z)) = \det(\overline{g_\tau})\nu(g_\tau)^{-n_\tau}j_\tau(g, z)
\end{align*}
for all $g\in GU_{m,n}^+(\IR)$ and $z\in \mathcal{H}$.  We also note that
\begin{align*}
\lambda_\tau(gh, z) &= \lambda_\tau(g, hz)\lambda_\tau(g, z)\\
\mu_\tau(gh, z) &= \mu_\tau(g, hz)\mu_\tau(g, z)
\end{align*}
for all $g, h\in GU_{m,n}^+(\IR)$ and $z\in \mathcal{H}$.
\end{rmk}
For each $g = \left(g_\tau\right)_{\tau\in\Sigma_\cmfield}$ and $z=(z_\tau)\in \prod_{\tau\in \Sigma_{\cmfield}}\mathcal{H}_\tau$, we define
\begin{align*}\index{$M_g(z) = M(g, z)$}
M_g(z) := M(g, z):=\left(\lambda_\tau(g, z), \mu_\tau(g, z)\right)_{\tau\in \Sigma_\cmfield}.
\end{align*}
\begin{rmk}\label{Mgrmk}
We note that ${}^tM_\gamma(z)$ maps the lattice $p_{\gamma z}(L)$ from Remark \ref{PELclassificationH} to $p_z(L)$ and defines an isomorphism from $\underline{A}_{\gamma z}$ to $\underline{A}_z$ for each $\gamma\in\Gamma_\cpct$, .
\end{rmk}

Note that $M_g(z) \in \prod_{\tau\in \Sigma_\cmfield}\GL_{a_\tau}(\IC)\times \GL_{b_\tau}(\IC)$.  Let $H(\IC)=\prod_{\tau\in \Sigma_\cmfield}\GL_{a_\tau}(\IC)\times \GL_{b_\tau}(\IC)$, and let $\rho: H(\IC)\rightarrow \GL(X)$ be an algebraic representation with $X$ a finite-dimensional $\IC$-vector space.
Following the conventions of \cite[Equations (5.6a) and (5.6b)]{shar}, we define\index{$f\mid\mid_\rho g$}\index{$f\mid_\rho g$}
\begin{align*}
(f\mid\mid_\rho g)(z) := \rho(M_g(z))^{-1}f(g z)\\
f\mid_\rho g:=f\mid\mid_\rho \left(\nu(g)^{-1/2}g\right),
\end{align*}
where $\nu(g)^{-1/2}g:=(\nu(g_\tau)^{-1/2}g_\tau)_{\tau\in\Sigma_\cmfield}$.  Let $\Gamma$ be a congruence subgroup of $GU^+(\IQ)$.

\begin{defi}\label{autdef} With the conventions and notation above, we define an {\em automorphic form of weight $\rho$ and level $\Gamma$} to be a function
\begin{align*}
f:\mathcal{H}\rightarrow X
\end{align*}
such that all of the following conditions hold:
\begin{enumerate}
\item{$f$ is holomorphic.}\label{condholo}
\item{$f\mid\mid_\rho \gamma = f$ for all $\gamma\in \Gamma$.}\label{cond:auto}
\item{If $\realfield=\IQ$ and the signature is $(1,1)$, then $f$ is holomorphic at every cusp.}\label{kcond}
\end{enumerate}
\end{defi}
\begin{rmk}\label{koocherrmk}
By Koecher's principle (\cite[Theorem 2.3 and Remark 10.2]{lankoecher}), if we are not in the situation of Condition \eqref{kcond} (i.e.\ $\realfield=\IQ$ with signature $(1,1)$), then holomorphy at the boundary is automatic.
\end{rmk}
We will use the terminology {\em automorphic function} to refer to a function that meets Condition \eqref{cond:auto} but not necessarily Conditions \eqref{condholo} and \eqref{kcond}.
\begin{rmk}
The definition of {\em automorphic form} in \cite[Section 5.2]{shar} uses $\mid_\rho$ in place of $\mid\mid_\rho$.  In the text following that definition, though, Shimura explains that he works almost exclusively with $\mid\mid_\rho$ in \cite{shar}.  Similarly, we will work with $\mid\mid_\rho$ here.  Our choice in this manuscript is motivated by the fact that $\mid\mid_\rho$ is what arises naturally from algebraic geometry.  So this definition will be consistent with our later algebraic geometric definition of {\em automorphic form}.
\end{rmk}

\begin{rmk}
In the case of signature $(n,n)$, automorphic forms on unitary groups were first introduced by Hel Braun in \cite{braun1, braun2, braun3}, where she considered them as functions on Hermitian symmetric spaces, in analogue with the formulation of Siegel modular forms as functions on Siegel upper half space.  She called them {\em Hermitian modular forms}, in analogue with Siegel modular forms.
\end{rmk}

\begin{example}
It is instructive to compare the case of $\GL_2$ (i.e.\ the familiar case of classical modular forms) with the case of $\GSp_2$ and $\GU(1, 1)$.  What do you notice about the symmetric spaces in each of these cases?  What do the automorphic forms on each of these groups have to do with each other?
\end{example}

\begin{example}
When $K$ is a quadratic imaginary field and $\langle, \rangle$ is of signature $(2, 1)$ on a $K$-vector space $V$, the group $U(V, \langle, \rangle)$ is also called a {\em Picard modular group}, and the automorphic forms on it are called {\em Picard modular forms}.
\end{example}

\subsection{As functions on a space of abelian varieties with PEL structure}\label{sec:abaut}
In analogue with the case of modular forms, where we move from functions on the upper half plane to functions on a space of elliptic curves with additional structure, we now reformulate our definition of automorphic forms in terms of functions on a space of abelian varieties with additional structure.  Given an abelian variety $\underline{A} = \underline{A}_z$ (i.e.\ an abelian variety $A$ together with a polarization, endomorphism, and level structure) parametrized by $z\in \Gamma_\cpct\backslash\mathcal{H}$ (like in Remark \ref{PELclassificationH}) for some neat open compact subgroup $\cpct$, let 
\begin{align*}
\underline{\Omega}:=\underline{\Omega}_{A/\IC} = H^1(A, \ZZ)\otimes\IC.
\end{align*}
Note that the action of $\CO$ on $\underline{\Omega}$ coming from $h$ induces a decomposition of modules
\begin{align*}
\underline{\Omega} = \underline{\Omega}^+\oplus\underline{\Omega}^- = \oplus_{\tau\in \mathcal{T}_\cmfield}\underline{\Omega}_\tau,
\end{align*}
where $\underline{\Omega}^{\pm} = \oplus_{\tau\in \Sigma_{\cmfield}}\Omega_\tau^{\pm}$ and the rank of $\Omega_\tau^{\pm}$ is $a_\tau^{\pm}$, where $(a_\tau^+, a_\tau^-)$ is the signature of the PEL data at $\tau$.  (This is an instance of the decomposition described in Equation \eqref{equ:pmconvention}.)  For each $\tau\in\mathcal{T}_\cmfield$, we write $\underline{\Omega}_\tau$ for $\underline{\Omega}_\tau^+$ if $\tau\in \Sigma_\cmfield$ and for $\underline{\Omega}_\tau^-$ if $\tau\nin\Sigma_\cmfield$.  Note that $\CO$ acts on $\underline{\Omega}_\tau$ via $\tau$.
For each $\tau\in \mathcal{T}_\cmfield$, we define
\begin{align*}\index{$\CE_{\underline{A}}$}
\CE_{\underline{A}} = \oplus_{\tau\in\mathcal{T}_\cmfield}\Isom_\IC(\Omega_\tau, \IC^{a_\tau}),
\end{align*}
where $(a_\tau, a_{\tau^\ast})$ is the signature at $\tau$ for each $\tau\in \Sigma_\cmfield$ and $\Isom_\IC$ denotes the $\IC$-module of isomorphisms as $\IC$-modules.
We have an action of $\Levi(\IC)=\prod_{\tau\in\mathcal{T}_\cmfield}\GL_{a_\tau}(\IC)$ on $\CE_{\underline{A}}$ given by
\begin{align*}
\left((g_\tau)_{\tau\in \mathcal{T}_\cmfield}(\ell_\tau)_{\tau\in \mathcal{T}_\cmfield}\right) \left((x_\tau)_{\tau\in \mathcal{T}_\cmfield}\right) = (\ell_\tau({ }^tg_\tau x_\tau))_{\tau\in \mathcal{T}_\cmfield}
\end{align*}
for each $(g_\tau)_{\tau\in \mathcal{T}_\cmfield}\in \Levi(\IC)$ and each $(\ell_\tau)_{\tau\in \mathcal{T}_\cmfield}\in \CE_{\underline{A}}$.

\begin{rmk}
Note that the map $p_z$ from Remark \ref{PELclassificationH} induces a choice of basis for $H_1(A_z, \ZZ)$ (and hence an element $\ell_z\in \CE_{\underline{A}_z}$), via the identification of $H_1(A_z, \ZZ)$ with $p_z(L)$.
\end{rmk}

The following lemma is similar to \cite[Lemma 3.7]{CEFMV}.
\begin{lem}\label{lemfFFf}
Let $\cpct$ be a neat open compact subgroup of $GU(\adeles^\infty)$, and let $\Gamma=\Gamma_\cpct:=\cpct\cap GU^+(\IQ)$.  Let $\rho: H(\IC)\rightarrow GL(X)$ be an algebraic representation, with $X$ a finite-dimensional $\IC$-vector space.  There is a one-to-one correspondence between the following two sets:
\begin{itemize}
\item{the set of $X$-valued automorphic functions $f$ on $\mathcal{H}$ of weight $\rho$ and level $\Gamma$}
\item{the set of $X$-valued functions $F$ on the set of pairs $(\underline{A}, \ell)$, with $\underline{A}$ an abelian variety parametrized by $\Gamma\backslash\mathcal{H}$ and $\ell\in \CE_{\underline{A}}$, that satisfy
\begin{align}\label{Fcondaell}
F(\underline{A}, g\ell) = \rho({}^tg)^{-1}F(\underline{A}, \ell)
\end{align}
for all $g\in H(\IC)$.
}
 \end{itemize}
The bijection identifies such a function $F$ of $(\underline{A}, \ell)$ with the automorphic function $f_F$ defined by $f_F(z)=F(\underline{A}_z, \ell_z)$.
\end{lem}
\begin{proof}
Given a function $F$ of $(\underline{A}, \ell)$ as in the statement of the lemma, note that by Remark \ref{Mgrmk} (which says that the action of $\Gamma$ preserves the isomorphism class of $(\underline{A}, \ell)$), we have
\begin{align*}
F(\underline{A}_{\gamma z}, \ell_{\gamma z}) = F(\underline{A}_z, { }^tM_\gamma(z)^{-1}\ell_z) = \rho(M_\gamma(z))F(\underline{A}_z, \ell_z)
\end{align*}
for all $\gamma\in \Gamma$.
So the function $f_F$ is well-defined (i.e.\ independent of the isomorphism class of $\underline{A}$).  Now, we define a map $f\mapsto F_f$ that is inverse to the map $F\mapsto f_F$.  Given $\underline{A}$ parametrized by $\Gamma_{\cpct}\backslash\mathcal{H}$, let $z$ be such that $\underline{A}$ is isomorphic to $\underline{A}_z$, and let $g\in\Levi(\IC)$ be such that $\ell=g\ell_z$.  It is straightforward now to check that given an automorphic function $f$ as above, the function 
$F_f(\underline{A}, \ell):=\rho({ }^tg)^{-1}f(z)$
satisfies Condition \eqref{Fcondaell} and provides the desired inverse map.
\end{proof}
Thus, we may view automorphic functions as certain functions that assign an element of $X$ to each pair $(\underline{A}, \ell)$.  (N.B.\ If instead we want to restrict the discussion to automorphic forms, then we also need to check the holomorphy conditions from Definition \ref{autdef}.)  Note that as an intermediate step in reformulating our modular functions, we could also have defined them as functions on lattices.  For details of that perspective, see \cite[Theorem 2.4]{EDiffOps}.  We can also define automorphic forms as certain rules that take complex abelian varieties as $\underline{A}$ as input, which is the content of Lemma \ref{CECCdefnlemma} and follows immediately from Lemma \ref{lemfFFf}.  First, though, we need to introduce the contracted product
\begin{align*}\index{$\CE_{\uA, \rho}$}
\CE_{\uA, \rho}:=\CE_{\uA}\times^H X:=\CE_{\uA}\times X/\sim,
\end{align*}
where the equivalence $\sim$ is given by
\begin{align*}
(\ell, v)\sim (g\ell, \rho({ }^tg^{-1})v)
\end{align*}
for all $g\in H(\IC)$.

\begin{lem}\label{CECCdefnlemma}[Lemma 3.9 of \cite{CEFMV}]
Let $\rho$, $X$, $\cpct$, and $\Gamma$ be as in Lemma \ref{lemfFFf}.  There is a one-to-one correspondence between the following two sets:
\begin{itemize}
\item{the set of $X$-valued automorphic functions $f$ on $\mathcal{H}$ of weight $\rho$ and level $\Gamma$}
\item{the set of rules that assign to each point $\uA$ parametrized by $\Gamma\backslash\mathcal{H}$ an element $\tilde{F}(\uA)\in\CE_{\uA, \rho}.$
}
\end{itemize}
\end{lem}
The correspondence between $F$ and $\tilde{F}$ is given by $(\ell, F(\uA, \ell)) = \tilde{F}(\uA)$.
\begin{rmk}\label{ceccrmk}
Note that we could reformulate the elements $\tilde{F}$ from Lemma \ref{CECCdefnlemma} as the global sections of a vector bundle over $\Gamma\backslash\CH$.  When we formulate automorphic forms algebraic geometrically below, this is a perspective we will introduce.
\end{rmk}

For the algebraic theory, it will be useful to develop a definition of automorphic forms that also works over other base rings.  Our formulation in terms of abelian varieties provides some inspiration for how we might extend our discussion to other base rings.
\begin{rmk}
In this section, we will exclude the possibility of $\realfield=\IQ$ with signature $(1,1)$ so that we do not need to worry about holomorphy at cusps (which is essentially the case of classical modular forms).  In all other cases, holomorphy at cusps is automatic, as explained in Remark \ref{koocherrmk}.
\end{rmk}

Let $\reflex'$ be a finite extension of $E$ that contains $\tau(\cmfield)$ for all $\tau\in \mathcal{T}_\cmfield$.  For each scheme $S$ over $\Spec\left(\CO_{\reflex', (p)}\right),$ we define
\begin{align*}\index{$\moduliint_{\cpct, S}$}
\moduliint_{\cpct, S}:=\moduliint_\cpct\times_{\Spec\left(\CO_{\reflex, (p)}\right)} S.
\end{align*}
Following the conventions for decompositions introduced in Section \ref{decompsection}, we can decompose the sheaf of relative differentials $\underline{\Omega}_{\uA/S}$\index{$\underline{\Omega}_{\uA/S}$} of an $S$-point $\underline{A}$ of $\moduliint_{\cpct}$ as
\begin{align*}
\uO_{\uA/S} = \oplus_{\tau\in\mathcal{T}_\cmfield}\uO_{\uA/S, \tau} = \oplus_{\sigma\in\mathcal{T}_{\realfield}}\left(\uO_{\uA/S, \sigma}^+\oplus\uO_{\uA/S, \sigma}^-\right).
\end{align*}
We define sheaves
\begin{align*}\index{$\CE_{\uA/S}$}
\CE_{\uA/S} :=\oplus_{\tau\in\mathcal{T}_\cmfield}\Isom_{\CO_S}\left(\uO_{\uA/S, \tau},\CO_S^{a_\tau}\right)\\
\CE_{\uA/S}^\pm:=\oplus_{\tau\in \Sigma_\cmfield}\Isom_{\CO_S}\left(\uO_{\uA/S, \tau}^\pm, \CO_S^{a_\tau^\pm}\right)\index{$\CE_{\uA/S}^\pm$}
\end{align*}

  For the remainder of this section, let $R$ be an $\CO_{\reflex', (p)}$-algebra, let $M_\rho$ be a finite free $R$-module, and let $\rho: H(R)\rightarrow \GL_R(M_\rho)$ be an algebraic representation of $H(R)$.  For any $R$-algebra $R'$, we extend the action of $H(R)$ linearly to an action of $H(R')$ on $(M_\rho)_{R'}:=M_\rho\otimes_RR'$.  We define a contracted product
 \begin{align*}\index{$\CE_{\uA/R', \rho}$}
 \CE_{\uA/R', \rho}:=\CE_{\uA/R}\times^H(M_\rho)_{R'}:=\left(\CE_{\uA/R'}\times\left(M_\rho\right)_{R'}\right)/\sim,
 \end{align*}
 where the equivalence $\sim$ is defined by
 \begin{align*}
 (\ell, m)\sim (g\ell, \rho({ }^tg^{-1})m)
 \end{align*}
 for all $g\in H$.

Definitions \ref{algdef1}, \ref{algdef2}, \ref{algdef3} are the analogues in our setting of Nick Katz's definitions of modular forms in \cite[Sections 1.1 -- 1.5]{katzmodular}.  Notice the parallels between the functions introduced in Lemma \ref{lemfFFf}, which are defined over $\IC$, and those defined by Definition \ref{algdef1}, which allows us to work over other rings as well.
\begin{defi}\label{algdef1}
An {\em automorphic form of weight $\rho$ and level $\cpct$ defined over $R$} is a function $f$ that, for each $R$-algebra $R'$, assigns an element of $(M_\rho)_{R'}$ to each pair $(\uA, \ell)$ consisting of an $R'$-point $\uA$ of $\moduliint_\cpct(R')$ and $\ell\in\CE_{\uA/R'}$ such that both of the following conditions hold:
\begin{enumerate}
\item{$f(\uA, g\ell) = \rho({ }^tg^{-1})f(\uA, \ell)$ for all $g\in H(R')$ and all $\ell\in \CE_{\uA/R'}$}
\item{If $R'\rightarrow R''$ is a homomorphism of $R$-algebras, then
\begin{align*}
f(\uA\times_{R'} R'', \ell\otimes_{R'}1)  = f(\uA, \ell)\otimes_{R'}1_{R''}\in (M_\rho)_{R''},
\end{align*}
i.e.\ the definition of $f(\uA, \ell)$ commutes with extension of scalars for all $R$-algebras.
}
\end{enumerate}
\end{defi}
Definition \ref{algdef1} is the generalization to our setting of the second definition of modular forms Katz gives in \cite[Section 1.1]{katzmodular}.  Definition \ref{algdef2} is the generalization to our setting of the first definition of modular forms Katz gives in \cite[Section 1.1]{katzmodular}.  Similarly to Katz's situation, our two definitions here are equivalent and are simply two ways to formulate an automorphic form.

\begin{defi}\label{algdef2}
An {\em automorphic form of weight $\rho$ and level $\cpct$ defined over $R$} is a rule
\begin{align*}
\uA\mapsto \tilde{f}(\uA)\in\CE_{\uA/R', \rho},
\end{align*}
for each $R$-algebra $R'$ and $R'$-point $\uA$ in $\moduliint_\cpct(R')$, that commutes with extension of scalars for all $R$-algebras, i.e.\
\begin{align*}
\tilde{f}(\uA\times_{R'} R'')  = \tilde{f}(\uA)\otimes_{R'}1_{R''},
\end{align*}
for all $R$-algebra homomorphisms $R'\rightarrow R''$.
\end{defi}
The equivalence between Definitions \ref{algdef1} and \ref{algdef2} is given by
\begin{align*}
\tilde{f}(\uA) = (\ell, f(\uA, \ell)).
\end{align*}
\begin{rmk}\label{rmk:isomexpl}
For readers seeing these definitions for the first time, it is a useful exercise to see what they say in the case of modular forms and then compare them with Katz's definitions of modular forms in \cite[Section 1.1]{katzmodular}.  At first glance, the appearance of $\mathcal{E}_{\uA/R}$ might look surprising to readers who are only familiar with the modular forms setting and have not yet considered the setting of higher rank groups.  Note, though, that giving an element of $\mathcal{E}_{\uA/R}$ is equivalent to choosing an ordered basis for $\Omega_{\uA/R}$.  In the setting of modular forms, $\uA$ is replaced by an elliptic curve, in which case an ordered basis is simply a choice of a nonvanishing differential.  In other words, in the setting of modular forms, the setup here says to consider functions on pairs $(\underline{E}, \omega)$ consisting of an elliptic curve (of some specified level) and a nonvanishing differential, which coincides precisely with the conventional algebraic geometric formulation of modular forms presented in \cite[Section 1.1]{katzmodular}.
\end{rmk}

\subsection{As global sections of a sheaf}\label{sec:sheafaut}

Those familiar with the algebraic geometric definition of modular forms are likely accustomed to defining a modular form as a global section of a certain sheaf over a moduli space, like in \cite[Section 1.5]{katzmodular}.  This is the formulation we introduce now.  Let $\pi:\CA\rightarrow \moduliint_{\cpct, \CO_{E', (p)}}$\index{$\pi$}\index{$\CA$} denote the universal object over $\moduliint_{\cpct, \CO_{E', (p)}}$, and define
\begin{align*}\index{$\uo$}
\uo:=\pi_\ast\uO_{\CA/\moduliint}.
\end{align*}
Following the conventions for decompositions introduced in Section \ref{decompsection}, we have a decomposition
\begin{align*}
\uo = \oplus_{\tau\in\Sigma_\cmfield}\left(\uo_\tau^+\oplus\uo_\tau^-\right) = \oplus_{\tau\in\mathcal{T}_\cmfield}\uo_\tau.
\end{align*}
We define a sheaf $\CE = \CE_\cpct$ on $\moduliint_{\cpct, \CO_{E', (p)}}$ by
\begin{align*}\index{$\CE = \CE_\cpct$}
\CE :=  \CE_\cpct:=\oplus_{\tau\in\mathcal{T}_\cmfield}\Isom_{\CO_\moduliint}\left( \uo_\tau, (\CO_{\moduliint})^{a_\tau}\right).
\end{align*}
For motivation for introducing the sheaf $\CE$, see Remark \ref{rmk:isomexpl}.  Following the conventions introduced in Section \ref{decompsection}, we also define sheaves
\begin{align*}\index{$\CE^\pm$}
\CE^\pm = \oplus_{\tau\in \Sigma_\cmfield}\CE_\tau^\pm.
\end{align*}

For any representation $(\rho, M_\rho)$ of $H$,
we define
\begin{align*}\index{$\uo^\rho=\CE^\rho$}
\uo^\rho:=\CE^\rho:=\CE\times^H M_\rho
\end{align*}
to be the sheaf on $\moduliint_{\cpct}$ for which, for each $\CO_{E', (p)}$-algebra $R$,
\begin{align*}
\CE^\rho(R)=\left(\CE(R)\times M_\rho\otimes R\right)/\sim,
\end{align*}
with the equivalence $\sim$ given by
\begin{align*}
(\ell, m)\sim(g\ell, \rho({}^tg^{-1})m)
\end{align*}
for all $g\in H$.

\begin{defi}\label{algdef3}
An {\em automorphic form of weight $\rho$ and level $\cpct$ defined over $R$} is a global section of the sheaf $\uo^\rho=\CE_{\cpct, \rho} = \CE^\rho$ on $\moduliint_{\cpct, R}$.
\end{defi}
It is a straightforward exercise to check that Definition \ref{algdef3} is equivalent to Definition \ref{algdef2}.

\begin{rmk}
When $\rho$ is an irreducible representation of positive dominant weight $\kappa$, we often write $\uo^\kappa$\index{$\uo^\kappa$} or $\CE^\kappa$\index{$\CE^\kappa$} in place of $\CE^\rho$.  Also, note that if $\kappa =\left(\kappa_\tau\right)_{\tau\in \mathcal{T}_\cmfield}$, then
\begin{align*}
\uo^\kappa = \boxtimes_{\tau\in\mathcal{T}_\cmfield}\uo_\tau^{\kappa_\tau}
\end{align*}
\end{rmk}

\begin{rmk}
As noted in \cite[Section 2.5]{EiMa}, $\uo^\kappa$ can be canonically identified with $\mathbb{S}_\kappa(\uo)$.
\end{rmk}

\begin{rmk}
If $f$ is an automorphic form of weight $\rho_\kappa$, we often say that $f$ is {\em of weight} $\kappa$.
\end{rmk}

\begin{rmk}
By Koecher's principle (see \cite[Theorem 2.3 and Remark 10.2]{lankoecher}), automorphic forms defined over $\moduliint_\cpct$ extend uniquely to automorphic forms over a toroidal compactification of $\moduliint_\cpct$, so long as we exclude the solitary case of $[\realfield:\IQ]=1$ with signature $(1,1)$ (the familiar case of modular forms for $GL_2$).
\end{rmk}

\begin{rmk}
The algebraic automorphic forms defined above over $R=\IC$ are in bijection with finite sets of holomorphic automorphic forms  as defined earlier.  More precisely, given an automorphic form on $\moduli_\cpct(\IC)$, we get global sections over each of the connected components of $\moduli_\cpct(\IC)$.  Recall from Section \ref{shconnsec} each of these connected components is isomorphic to a quotient $\Gamma\backslash\mathcal{H}$.  By GAGA and Lemma \ref{CECCdefnlemma} (together with Remark \ref{ceccrmk}), we then see that each algebraic automorphic form on $\moduli_\cpct(\IC)$ gives rise to a set of holomorphic automorphic forms on $\mathcal{H}$, one for each component of $\moduli_\cpct(\IC)$.
\end{rmk}

\begin{rmk}
The moduli space $\moduliint_\cpct$ acts as a sort of bridge that allows us to move between different base rings.  This, in turn, is useful for studying algebraic aspects of automorphic forms, including certain values of automorphic forms {\em a priori} arising over $\IC$, a crucial ingredient in our study of $L$-functions below.  In another direction, this space also serves as a starting point for defining $p$-adic automorphic forms.
\end{rmk}

\subsection{As functions on a unitary group}
From the perspective of representation theory, it is useful to define our automorphic forms as functions on a unitary group.  Within this perspective, there are several equivalent ways of viewing automorphic forms, namely as $\IC$-valued functions on $G(\IR)$, as $\IC$-valued functions on $G(\adeles)$, and as vector-valued functions on $G(\IR)$ or $G(\adeles)$.  The relationship between the $\IC$- and vector-valued approaches for functions on $G(\IR)$ is also the subject of \cite[Remark 1.5(2)]{boreljacquet}. The formulation of automorphic forms as functions on groups is a direct generalization of the formulation for $\GL_2$ (i.e.\ classical modular forms).  We begin by reviewing the (likely more familiar) case of $\GL_2$.  The setup from Sections \ref{shconnsec} and \ref{sec:hermsymmsp} enables us to extend that approach to the setting of unitary groups.

\subsubsection{Review of the case of $GL_2$}
Before proceeding with unitary groups, we first briefly review how to translate the definition of modular forms as functions on the upper half plane $\mathfrak{h}$ into the definition in terms of functions on a group, first as functions of $SL_2(\IR)$, and then as functions of $\GL_2(\adeles)$.  Although the main focus of this manuscript is automorphic forms on unitary groups, readers might find it helpful first to recall the situation for modular forms.
In particular, the definitions of automorphic forms as functions on unitary groups in Sections \ref{sec:Raut} and \ref{sec:adeleaut} arise similarly to how they arise for modular forms and $GL_2$.

Since $SL_2(\IR)$ acts transitively on $\mathfrak{h}$ and $SO_2(\IR) = \left\{\alpha_\theta :=\begin{pmatrix}\cos \theta & \sin \theta\\
-\sin \theta & \cos \theta
\end{pmatrix}\right\}$ is the stabilizer of $i\in \mathfrak{h}$, we can identify $SL_2(\IR)/SO_2(\IR)$ with $\mathfrak{h}$.  Given a modular form $f$ of weight $k$ and level $\Gamma$ on $\mathfrak{h}$, we define $\phi_f: \Gamma\backslash SL_2(\IR)\rightarrow \IC$ by
\begin{align*}
\phi_f(g):=j(g, i)^{-k}f(g i ),
\end{align*}
where $j(g, z)$ is the canonical automorphy factor, i.e.\ $j(g, z) = cz+d$ for $z\in \mathfrak{h}$ and $g = \begin{pmatrix}a& b \\ c& d\end{pmatrix}\in SL_2(\IR)$.  Then $\phi_f$ is an example of an automorphic form of weight $k$ and level $\Gamma$ on $SL_2(\IR)$, i.e.\ $\phi_f$ is a smooth function that is left $\Gamma$-invariant (i.e.\ $\phi_f(\gamma g ) = \phi_f(g)$ for all $\gamma\in \Gamma$ and $g\in SL_2(\IR)$) and such that each $\alpha_\theta\in SO_2(\IR)$ acts on the right by multiplication by $e^{ki \theta}$, i.e.\ $\phi_f(g\alpha_\theta) = e^{ki\theta}\phi_f(g)$.  

Taking this a step further, we can replace $SL_2(\IR)$ by $GL_2^+(\IR)/\IR_{>0}$, where $+$ denotes positive determinant (or by $GL_2(\IR)/\IR^\times$, noting that in each case we are taking the quotient of the group by its center).  In this case, we define $\phi_f: \left(\Gamma\cdot Z(G)\right)\backslash G(\IR)\rightarrow \IC$, with $G$ denoting $GL_2$ or $GL_2^+$ (or even $SL_2$, like above) and $Z(G)$ denoting the center of $G$, by
\begin{align*}
\phi_f(g):=\det(g)^{k/2}j(g, i)^{-k}f(g i ),
\end{align*}
and $\phi_f(g)$ satisfies the same conditions as in the previous paragraph, but with $SL_2$ replaced by $G$.

It turns out, though, to be convenient to define automorphic forms as functions of $GL_2(\adeles)$.  First, we note that if for each prime number $p$, $\cpct_p\subset GL_2(\ZZ_p)$ is a compact open subgroup such that $\det(\cpct_p) = \ZZ_p^\times$ and $\cpct_p = \GL_2(\ZZ_p)$ for all but finitely many $p$, then 
\begin{align}\label{equ:gl2adel}
\GL_2(\adeles) = \GL_2(\IQ)\cdot\left(\GL_2^+(\IR)\times \prod_p\cpct_p\right).
\end{align}
  If $\cpct = \prod_p\cpct_p$ and 
  \begin{align}\label{equ:GammaK}
  \Gamma  = \GL_2(\IQ)\cap (GL_2^+(\IR)\times \cpct),
  \end{align}
   then Equation \eqref{equ:gl2adel} induces bijections
\begin{align*}
\Gamma\backslash\mathfrak{h}&\leftrightarrow Z\left(\GL_2\left(\adeles\right)\right)\GL_2(\IQ)\backslash \GL_2(\adeles)/\left(SO_2(\IR)\times\cpct\right)\\
\Gamma\backslash \GL_2^+(\IR)&\leftrightarrow \GL_2(\IQ)\backslash \GL_2(\adeles)/\cpct.
\end{align*}
Note $Z\left(\GL_2\left(\adeles\right)\right)\GL_2(\IQ)\backslash \GL_2(\adeles)/\left(SO_2(\IR)\times\cpct\right)=\IR^\times\GL_2(\IQ)\backslash \GL_2(\adeles)/\left(SO_2(\IR)\times\cpct\right) = \IR_{>0}\GL_2(\IQ)\backslash \GL_2(\adeles)/\left(SO_2(\IR)\times\cpct\right)$.
These identifications enable us to reformulate $f$ (and $\phi_f$) from above as a function $\varphi_f: \GL_2(\IQ)\backslash \GL_2(\adeles)\rightarrow \IC$ defined by
\begin{align*}
\varphi_f(\gamma g_\infty (g_p)_p) :=\phi_f(g_\infty),
\end{align*}
for all $\gamma\in GL_2(\IQ)$, $g_\infty\in GL_2^+(\IR)$, and $g_p\in \cpct_p$.
So $\varphi_f$ is a function of $GL_2(\adeles)$ that is left-invariant under $\GL_2(\IQ)$, right invariant under $\cpct$, and satisfies $\varphi_f(g\alpha_\theta) = e^{ki\theta}\phi_f(g)$ for all $\alpha_\theta\in SO_2(\IR)$ (and is smooth as a function of $\GL_2(\IR)$).  One can also extend this treatment to include a nebentypus character.  We also note that the familiar congruence subgroups $\Gamma_0(N)$, $\Gamma_1(N)$, and $\Gamma(N)$ arise in Equation \eqref{equ:GammaK} when $\cpct$ is the subgroup $\cpct(N)$, $\cpct_0(N)$, and $\cpct_1(N)$ of $GL_2(\hat{\ZZ})$, respectively, consisting of matrices congruent mod $N$ to elements of  $\Gamma_0(N)$, matrices in $\cpct(N)$ of the form $\begin{pmatrix}* & *\\ 0 & 1\end{pmatrix}$, and matrices in $\cpct(N)$ of the form $\begin{pmatrix}* &0\\ 0 & 1\end{pmatrix}$, respectively.

\subsubsection{As functions of unitary groups over $\IR$}\label{sec:Raut}
Given the expression of the symmetric space $\mathcal{H}$ in terms of a quotient of $G(\IR)$, it is natural also to define automorphic forms as certain functions on $G(\IR)$.
The group $\cpct_\infty$ is the stabilizer of a fixed point $\mathbf{i}\in \CH$, and similarly to \cite[Section A8.2]{shar}, we identify $\cpct_\infty$ with its image in $H(\IC)$ under the embedding $\cpct_\infty\hookrightarrow H(\IC)$ that maps $k\in \cpct_\infty$ to $M_k(\mathbf{i})=M(k, \mathbf{i})$.
\begin{rmk}
We use the notation $\mathbf{i}$ for the fixed point to emphasize the connection with the $SL_2$ case, where $\CH$ is the upper half plane, $\cpct$ is the group $SO_2(\IR)$, and $\mathbf{i}$ is the number $i$.  More generally, when working with $\mathcal{H}_n$, $\mathbf{i}$ can be realized as the diagonal matrix $i1_n$, and $\cpct_\infty$ as $U(n)\times U(n)$.
\end{rmk}
As discussed in \cite[Section A8.2]{shar}, an automorphic form $f$ viewed as a function on $\CH$ gives a corresponding function $f^\rho$ on $G(\IR)$ as in Definition \ref{realautdefn} below via
\begin{align*}
f^\rho(g) := (f\mid\mid_\rho g)(\mathbf{i})= \rho(M_g(\mathbf{i}))^{-1}f(g\mathbf{i})
\end{align*}
for each $g\in G(\IR)$.  In the other direction, given $f^\rho$ on $G(\IR)$ as in Definition \ref{realautdefn}, we define an automorphic form on $\CH$ of weight $\rho$ by
\begin{align*}
f(z)=f(g_z\mathbf{i}):= \rho(M_g(\mathbf{i}))f^\rho(g_z),
\end{align*}
for each $z\in \CH$, where $g_z\in G(\IR)$ is such that $g_z\mathbf{i} = z$.  Let $\rho:H(\IC)\rightarrow GL(V)$ be a finite-dimensional $\IC$-representation of $H(\IC)$.
 We arrive at the following definition, which is similar to the definition of an automorphic form over any reductive group.
\begin{defi}\label{realautdefn}
An {\em automorphic form} of weight $\rho$ and level $\Gamma$ is a holomorphic function
\begin{align*}
f:G(\IR)\rightarrow V
\end{align*}
that is of moderate growth (in the sense made explicit in Definition \ref{defi:moderategrowth} below) and such that
\begin{align*}
f(\gamma gk) = \rho(k)^{-1}f(g)
\end{align*}
for each $g\in G(\IR)$ and each $k\in \cpct_\infty$.
\end{defi}
\begin{rmk}
In the formulation of Definition \ref{realautdefn}, the holomorphy condition is equivalent to being killed by certain differential operators (as detailed in \cite[Section A8.2]{shar}).
\end{rmk}

\subsubsection{As functions of unitary groups over the adeles}\label{sec:adeleaut}
Given our expression in Section \ref{shconnsec} of $\moduliint_\cpct(\IC)$ as a quotient $X_\cpct$ of the adelic points of a unitary group (and also given our expression of $X_\cpct$ in terms of quotients of $\mathcal{H}$), it is natural to reformulate our definition of {\em automorphic form} in terms of functions on adelic points of a unitary group.  More generally, an automorphic form on a reductive linear algebraic group $G$ (unitary group or not) can be formulated adelically as a function on $G(\adeles)$ meeting certain conditions.  Indeed, this is the context in which we will be able to work with automorphic representations. 

The adelic formulation is particularly convenient in the context of $L$-functions.  For example, each Euler factor at a place $v$ in the Euler product for an automorphic $L$-function corresponds with information from the automorphic form at the place $v$.  This is seen in Tate's thesis, as well as in the discussion of the doubling method in Section \ref{autLfcnssection} below.  (The usefulness of the adelic formulation is also seen in Shimura's computation of Fourier coefficients in \cite{sh}, as well as the related computations of Fourier coefficients in \cite{apptoSHL, apptoSHLvv}, where the global Fourier transform factors - under certain conditions - as a product of local Fourier transforms.)  The adelic formulation also provides a convenient setting for viewing automorphic forms of different levels at once, for example in a collection of Shimura varieties $\Sh_\cpct$.  

Automorphic representations, discussed briefly in Section \ref{sec:autrep}, are realized in terms of the right regular action of $G(\adeles)$ on certain functions $\phi$ of $G(\adeles)$, i.e.\ given $g\in G(\adeles)$, $(g\phi) (h):= \phi(hg)$.

\begin{defi}
Given a subgroup $\mathcal{U}$ of $G(\adeles)$, we say that a function $\phi$ on $G(\adeles)$ is {\em right $\mathcal{U}$-finite}, if the right translates of $\phi$ by elements of $\mathcal{U}$ span a finite-dimensional vector space.  Similarly, we say that $\phi$ is {\em $Z(\mathfrak{g})$-finite}, where $Z(\mathfrak{g})$ denotes the center of the complexified Lie algebra $\mathfrak{g}$\index{$\mathfrak{g}$} of $G(\IR),$ if the translates of $\phi$ by elements of $Z(\mathfrak{g})$ span a finite dimensional vector space.
\end{defi}

\begin{defi}
We say that a function $\phi$ on $G(\adeles)$ is {\em smooth} if it is $\ci$ as a function of $G(\IR)$ and locally constant
 as a function of $G(\adeles_f)$.
\end{defi}

\begin{defi}\label{defi:moderategrowth}
We say that a $\IC$-valued function $\phi$ on $G(\adeles)$ is of {\em moderate growth} if there exist numbers $n, C\geq 0$, such that
\begin{align*}
|\phi(g)|\leq C|| g||^n
\end{align*}
for all $g\in G(\adeles)$.  Here, $||g||:=\prod_v\max_{1\leq i, j\leq n}\left(\max\left( |g_{ij}|_v, |g_{ij}^{-1}|_v\right)\right)$, where we have identified $G$ with its image in $\GL_n$ under an embedding $G\hookrightarrow \GL_n$ and $g_{ij}$ denotes the $ij$th entry of the associated matrix in $GL_n$.  More generally, a tuple of $\IC$-valued functions $\left(\phi_1, \ldots, \phi_d\right)$ is of {\em moderate growth} if each function $\phi_i$ is of moderate growth.
\end{defi}

\begin{defi}
For any reductive group $G$ (unitary group or not), an {\em automorphic form} on $G(\adeles)$ is defined to be a function $\phi$ on $G(\adeles)$ such that
\begin{align*}
\phi(g) = \phi(\alpha g)
\end{align*}
for all $\alpha \in G(\IQ)$ (so we may view $\phi$ as a function on $G(\IQ)\backslash G(\adeles)$) and such that $\phi$ additionally is smooth, right $(\cpct\times\cpct_\infty)$-finite, of moderate growth, and $Z(\mathfrak{g})$-finite.  We say that $\phi$ is of {\em level} $\cpct$ if $\phi$ is fixed by $\cpct$.
\end{defi}

\begin{example}
Observe that an adelic automorphic form on a definite unitary group of signature $(1, 0)$ is a Gr\"ossencharacter on $\adeles_\cmfield^\times$.  (Given a number field $L$, we denote by $\adeles_L$\index{$\adeles_L$ the ring of adeles for a number field $L$} the ring of adeles for $L$.)
\end{example}
We can also consider automorphic forms of weight $\kappa$ as follows.  
Given an irreducible $\IC$-representation $(V_\kappa, \rho_\kappa)$ of $\cpct_\infty$ of highest weight $\kappa$, an automorphic form of {\em weight} $\kappa$ is an automorphic form $\phi: G(\adeles)\rightarrow V_\kappa$ such that $\phi(gu) = \rho_\kappa(u)^{-1} \phi(g)$ for all $g\in G(\adeles)$ and $u\in \cpct_\infty$.  We also can consider automorphic forms {\em with nebentypus character} $\psi: T(\hat\ZZ)\rightarrow \bar{\IQ}^\times$ factoring through $T(\ZZ/m\ZZ)$ for some integer $m$ and $\Torus$ a maximal torus like in Section \ref{sec:weightsandreps}, i.e.\ automorphic forms $\phi$ for which $\phi(gt) = \psi(t)\phi(g)$ for all $t\in T(\hat\ZZ)$.

\begin{defi}
An automorphic form $\varphi$ on $G(\adeles)$ is called a {\em cusp form} (or {\em cuspidal automorphic form}) if
\begin{align*}
\int_{N(\IQ)\backslash N(\adeles)}\varphi(n g) dn = 0
\end{align*}
for each $g\in G(\adeles)$ and each unipotent radical $N(\IQ)$ of each proper parabolic subgroup of $G(\IQ)$. 
\end{defi}

\begin{rmk}
We make a brief note about conventions.  Up to this point and sometimes going forward, we work with a group $G$ defined over $\IQ$.  At times, though, it will convenient to work with groups defined over other fields (e.g.\ $G_1$ defined over $\realfield$).  This will be the case, for example, in our initial discussion of the doubling method in Section \ref{sec:doubling}, when we introduce the original approach from \cite{PSR}.  If we replace $G$ by a group $H$ defined over a number field $L$, we consider the $\adeles_L$-points in place of the $\adeles$-points, and we write $H_v$ for the component of $H$ at $v$.
\end{rmk}

\subsection{Representations of $G(\adeles)$ and automorphic representations}\label{sec:autrep}
In Section \ref{autLfcnssection}, we consider $L$-functions attached to certain {\em automorphic representations}.  Here, we briefly establish some fundamental information about automorphic representations.  The Langlands conjectures (introduced by Robert Langlands) predict a precise correspondence between Galois representations and automorphic representations.  In particular, the predictions include that the $L$-function associated to a Galois representation is the same as a particular $L$-function associated to the corresponding automorphic representation.

The information summarized here is not specific to unitary groups but will be helpful to have available as we move forward.   For more detailed introductions to automorphic representations (and their role in the Langlands program), see, for example, \cite{arthurgelbart, gelbartelem, gelbartALM, kudlamfar, knappintro, bumpbook}.

We denote by $\mathcal{A}(G)$\index{$\mathcal{A}(G)$} the space of automorphic forms on $G(\adeles)$, and we denote by $\mathcal{A}_0(G)$\index{$\mathcal{A}_0(G)$} the space of cusp forms on $G(\adeles)$.  Let $Z$ denote the center of $G$, and let $\chi$ be a character of $Z(\IQ)\backslash Z(\adeles)$.  We denote by $\mathcal{A}(G)_\chi$\index{$\mathcal{A}(G)_\chi$} the submodule of automorphic forms $\phi$ such that $\phi(z g) = \chi(z)\phi(g)$, and we denote by $\mathcal{A}_0(G)_\chi$\index{$\mathcal{A}_0(G)_\chi$} the submodule of cuspidal automorphic forms in $\mathcal{A}(G)_\chi$.  Note that $G(\adeles)$ acts on $\mathcal{A}(G)$ and $\mathcal{A}_0(G)$ by right translation, i.e.\ $g\in G(\adeles)$ acts on an automorphic form $\phi$ via
\begin{align*}
(g\phi) (h):= \phi(hg)
\end{align*} for all $h\in G(\adeles)$.  Moreover, we have that $\mathcal{A}(G)$ and $\mathcal{A}_0(G)$ are $(\mathfrak{g}, \cpct_\infty)$-modules (in the sense of, e.g.\, \cite[Definition 2.3]{kudlamfar}), so are $(\mathfrak{g}, \cpct_\infty)\times G(\adeles_f)$-modules.

\begin{defi}
A $(\mathfrak{g}, \cpct_\infty)\times G(\adeles_f)$-module $(\pi, W)$ is {\em admissible} if each irreducible representation of $\cpct\times\cpct_\infty$ occurs with finite multiplicity in $W$. 
\end{defi}

\begin{defi}
An irreducible representation $\pi$ is called an  {\em automorphic representation} if it is an admissible representation such that $\pi$ is a subquotient of $\mathcal{A}(G)$.  
\end{defi}

\begin{defi}
 An automorphic representation is {\em cuspidal} (with central character $\chi$) if it occurs as a submodule of $\mathcal{A}_0(G)_\chi$ for some character $\chi$.
\end{defi}

\noindent We have 
\begin{align*}
\mathcal{A}_0(G) = \oplus_\pi m(\pi)\pi
\end{align*}
with the sum over all (isomorphism classes of) irreducible admissible cuspidal representations $\pi$ and $m(\pi)$ the multiplicity of $\pi$, which is always finite.

Given an irreducible representation $\pi$, we can write $\pi = \pi_\infty\otimes \pi_f$, with $\pi_\infty$ an irreducible representation at archimedean components and $\pi_f$ an irreducible representation of $G(\adeles_f)$.

\begin{defi}Let $\cpct$ be a compact open subgroup of $G(\adeles_f)$.
An automorphic representation $\pi$ is of {\em level} $\cpct$ if $\pi_f^\cpct\neq 0$ (where the superscript $\cpct$ denotes the subspace of $\cpct$-fixed vectors).
\end{defi}

Every irreducible admissible representation $\pi$ can be decomposed as a restricted tensor product
\begin{align*}
\pi=\otimes_v'\pi_v
\end{align*}
 of irreducible admissible representations $\pi_v$, as made precise in Proposition \ref{flaththm} below.  First, we briefly recall the notion of a representation $(\otimes'_v\pi_v, \otimes_v'W_v)$ of a group $H = \prod' H_v$ that is the restricted direct product with respect to subgroups $H'_v\subseteq H_v$.  Consider an infinite set of vector spaces $W_v$ indexed by a set $\Sigma$, and suppose that for all but finitely many $v\in \Sigma$ outside a finite subset of $S\subseteq \Sigma$, we have chosen a vector $\xi_v^\circ\in W_v$. 
 Let $\otimes'W_v$ be the restricted tensor product of the vector spaces $W_v$, i.e.\ the space $\otimes'W_v$ is spanned by vectors of the form $\otimes_v \xi_v$ with $\xi_v = \xi_v^\circ$ for all but finitely many $v$.
Let $(\pi_v, W_v)$ be a representation of a group $H_v$, and let $H'_v$ be a subgroup of $H_v$.  Let $H$ be the restricted direct product of the groups $H_v$ with respect to the subgroups $H'_v$.  Suppose that for all but finitely many $v$, there exists a vector $\xi_v^0\in W_v$ such that $H'_v$ fixes the vector $\xi_v^0$.  Then we define a representation $(\otimes'_v\pi_v, \otimes_v'W_v)$ of $H$ by
\begin{align*}
(\otimes_v'\pi_v)((g_v)_v) (\otimes \xi_v) = \otimes_v \pi_v(g_v)(\xi_v).
\end{align*}

\begin{prop}[Tensor Product Theorem \cite{flath}]\label{flaththm}
If $\pi$ is an irreducible, admissible representation of a connected reductive algebraic group $H$, then $\pi$ is isomorphic to a restricted tensor product
\begin{align*}
\pi\cong \otimes'_v \pi_v,
\end{align*}
where for all but finitely many nonarchimedean $v$, $\pi_v$ is an irreducible, admissible representation of $H(\realfield_v)$ that contains a nonzero $\cpct_v$-fixed vector $\xi_v^0$, and for each archimedean $v$, $\pi_v$ is an irreducible, admissible $(\mathfrak{g}, \cpct_v)$-module.
\end{prop}

In our work with the doubling method in Section \ref{sec:doubling}, we will rely on the factorization guaranteed by Proposition \ref{flaththm}.  In practice, our restricted tensor product will be taken with respect to $\cpct$.

\subsection{$q$-expansions}\label{sec:qexp}
For the moment, we consider the case of automorphic forms on unitary groups of signature $(n,n)$ at infinity, and furthermore, we assume that the unitary group under consideration is quasi-split over $\realfield$.
We immediately see that for each $\alpha\in \hern(\IC)$, where $\hern$\index{$\hern$} denotes $n\times n$ Hermitian matrices,
\begin{align*}
\begin{pmatrix}
1_n & \alpha\\
0 & 1_n
\end{pmatrix}
\end{align*}
  is an element of $U(n,n)(\IR)$ and can be identified with the tuple $\begin{pmatrix}
1_n & \alpha_\sigma\\
0 & 1_n
\end{pmatrix}\in\prod_{\sigma\in\mathcal{T}_{\realfield}}U(n,n)(\IR)$ via
\begin{align}\label{map:alphsigma}
\alpha\mapsto \left(\alpha_\sigma\right)_{\sigma\in\mathcal{T}_{\realfield}},
\end{align}
where $\alpha_\sigma$ is the matrix whose $ij$-th entry is $\tau(\alpha_{ij})$ with $\alpha_{ij}$ the $ij$-th entry of $\alpha$ and $\tau$ the unique element of $\Sigma_\cmfield$ extending $\sigma$.  There is a $\ZZ$-lattice $M\subseteq \hern(\IC)$ such that $\begin{pmatrix}1_n &\alpha\\
0 & 1_n
\end{pmatrix}$ is in $\Gamma$ for all $\alpha\in M$.  So if $f$ is an $X$-valued automorphic form of level $\Gamma$, then for all $\alpha\in M$, we have
\begin{align*}
f(z+\alpha) = f(z)
\end{align*}
for all $z\in \CH = \prod_{\sigma\in\mathcal{T}_{\realfield}}\CH_n$.  
Let
\begin{align*}
M^\vee:=\left\{h\in \hern(\IC)\mid \tr_{\realfield/\IQ}(\tr(hM))\subseteq \ZZ\right\}.
\end{align*}
Then as explained in \cite[Lemma A1.4]{sh} (see also \cite[Section 5.6]{shar}), $f$ has a Fourier expansion
\begin{align}\label{equ:Fourierexpnz}
f(z)=\sum_{h\in M^\vee}c(h)\mathbf{e}(hz),
\end{align}
with $c(h)\in X$, $\mathbf{e}(hz) = e^{2\pi i\sum_{\sigma\in\mathcal{T}_{\realfield}} \tr(h_\sigma z_\sigma)}$ for each $h
 \in 
M^\vee$ (with $h_\sigma$ defined as in \eqref{map:alphsigma}), and $z = \left(z_\sigma\right)_{\sigma\in \mathcal{T}_{\realfield}}\in \prod_{\sigma\in\mathcal{T}_{\realfield}}\CH_\sigma$.  Sometimes we set
\begin{align}\label{equ:qtohz}
q^h:=\mathbf{e}(hz),
\end{align}
in analogue with how we write $q= e^{2\pi iz}$ in the Fourier expansion of a modular form (e.g.\ as in Equation \eqref{equ:G2k}).

By \cite[Proposition 5.7]{shar}, if $\realfield\neq\IQ$ or the signature is not $(1,1)$, then the Fourier coefficients $c(h)$ of any (holomorphic) automorphic form are $0$ unless $h = \left(h_\sigma\right)_{\sigma\in \mathcal{T}_{\realfield}}$ has the property that $h_\sigma$ is nonnegative at each $\sigma\in \mathcal{T}_{\realfield}$.  An automorphic form is a cusp form if the Fourier coefficients are nonzero only at positive definite matrices.

\begin{rmk}
In this section, we have focused on Fourier expansions of automorphic forms on unitary groups  that are of signature $(n,n)$ at infinity and that, furthermore, are quasi-split over $\realfield$.  More generally, though, for unitary groups not necessarily meeting these conditions, automorphic forms have {\em Fourier--Jacobi expansions}.  Fourier--Jacobi expansions play a similar role to Fourier expansions, except that the coefficients are not necessarily numbers and instead are certain functions called {\em theta functions}.  Like in the case of the Fourier expansions discussed above, these expansions are exponential sums.  Unlike the expansions above, though, the exponential function is only applied to a subset of the coordinates parametrizing the symmetric space associated to the unitary group.  (The theta function is then evaluated on other coordinates parametrizing the symmetric space.) While this is beyond the scope of this manuscript, a clear and concise summary of Fourier--Jacobi expansions, including their precise form, is provided on \cite[p. 1104]{garrettFJ}.
\end{rmk}

\begin{rmk}\label{rmk:qexp1}
In analogue with the algebraic $q$-expansion principle for modular forms (which says that algebraic modular forms are determined by their $q$-expansions), Kai-Wen Lan has proved an algebraic Fourier--Jacobi principle for unitary groups \cite[Proposition 7.1.2.14]{la}.  Lan has also proved that the algebraically defined Fourier--Jacobi coefficients agree with the analytically defined Fourier--Jacobi coefficients \cite{lancomparison}.  These facts will be important in our discussion of algebraicity for Eisenstein series in Section \ref{sec:algEseries}
\end{rmk}

\section{Automorphic $L$-functions for unitary groups}\label{autLfcnssection}

One reason we care about automorphic forms and automorphic representations is that they can be convenient tools for proving results about $L$-functions attached to particular arithmetic data (such as Galois representations), as mentioned in Section \ref{sec:intro}.  Like in familiar cases, e.g.\ the Riemann zeta function from Section \ref{sec:intro}, the $L$-functions with which we work have Euler products (analogous to $\zeta(s) = \prod_{p \mbox{ prime }}(1-1/p^s)^{-1}$ for $\Real(s)>1$), functional equations (analogous to $Z(s)= Z(1-s)$, where $Z(s)=\pi^{-\frac{s}{2}}\Gamma\left(\frac{s}{2}\right)\zeta(s)$), and meromorphic continuations to $\IC$.

\subsection{$L$-functions}
Given a number field $L$, for each unramified finite place $v$, let $\Frob_v$ denote a(n arithmetic) Frobenius conjugacy class in $\Gal(\overline{L}/L)$.  Let $S$ denote a finite set of places of $K$ containing the set $S_\infty$ of archimedean places of $K$ and the set $S_\ram$ of ramified places of $K$.  By the Chebotarev density theory, each continuous Galois representation
\begin{align*}
\rho: \Gal(\overline{L}/L)\rightarrow \GL_n(\IC)
\end{align*}
is completely determined by the set of elements $\rho(\Frob_v)$ with $v\nin S$.
The Euler product for the Artin $L$-function associated to $\rho$ is
\begin{align*}
L^S(s, \rho):=\prod_{v\nin S}L_v(s, \rho),
\end{align*}
where
\begin{align*}
L_v(s, \rho):=\det\left(1-\rho(\Frob_v)q_v^{-s}\right)^{-1},
\end{align*}
with $q_v$ the number of elements in the residue field at $v$.  Since $\rho(\Frob_v)$ is a semisimple conjugacy class and since any semi-simple conjugacy class $A$ is completely determined by its characteristic polynomial
\begin{align*}
\det(t-A),
\end{align*}
we see that $L_v(s, \rho)$ is is independent of our choice of representative for $\rho(\Frob_v)$.  Under certain conditions, the Langlands conjectures assert that suitable cuspidal automorphic representations $\pi$ have associated Galois representations $\rho_\pi$.  (In the one-dimensional case, class field theory achieves this.  For an introduction to the Langlands conjectures, beyond the scope of this manuscript, see e.g.\, \cite{arthurgelbart, langlands}.)  More precisely, the prediction is that under suitable conditions, a certain semisimple conjugacy class $\sigma_v(\pi)$ associated to $\pi = \otimes_v \pi_v$ (in the Langlands dual group ${ }^LG$, as explained in, e.g.\, \cite{cogdellLG, borel}) is the same as the conjugacy class of $\rho_\pi(\Frob_v)$, in which case we obtain as an immediate consequence that
\begin{align*}
L^S(s, \rho_\pi)& = L^S(s, \pi):=\prod_vL_v(s, \pi_v)\\
L_v(s, \pi_v)&:=\det(1-\sigma_v(\pi)q^{-s})^{-1}
\end{align*}
where $\sigma_v(\pi)$ is the semi-simple conjugacy class associated to $\pi_v$.
\begin{rmk}
In general, the $L$-function and associated Euler factors also take as input a representation $r$ of the Langlands dual group ${ }^LG$, i.e.\ one considers $L^S(s, \pi, r):=\prod_v\det(1-r(\sigma_v(\pi))q^{-s})^{-1}$.  Going forward, though, we will work with the doubling method, which only concerns the standard representation.  Hence, in this manuscript, we always take $r$ to be the standard representation, and we omit it from the notation in the input to our $L$-function.  
\end{rmk}  
If you are familiar with $L$-functions associated to modular forms, then you probably expect information about an associated Hecke algebra to show up in the definition of the local Euler factors $L_v(s, \pi_v)$.  Indeed, the semi-simple conjugacy class associated to $\pi_v$ is associated to information from a Hecke algebra, which is generated by double coset operators, in analogue with the situation from modular forms.  That is, for $g\in G(\adeles_f)$ and open compact subgroups $\cpct_1$ and $\cpct_2$, we define the {\em Hecke operator} $[\cpct_1g\cpct_2]$ from automorphic forms of level $\cpct_1$ to level $\cpct_2$ to be the action of the double coset $\cpct_2 g\cpct_1 = \sqcup g_i \cpct_1$ given by
\begin{align*}
[\cpct_2 g \cpct_1]f (h)= \sum_i f(hg_i).
\end{align*}
As noted in \cite[Equation (22)]{HELS}, we have
\begin{align*}
[\cpct_2 g \cpct_1] f = \sum_i \left[g_i\right]^\ast f,
\end{align*}
with $[g_i]^\ast$ the pullback of the map \eqref{equ:bracketh}.
When the level $\cpct$ is clear from context, we set
\begin{align*}
T(g) := [\cpct g\cpct ].
\end{align*}
For additional details of Hecke operators on automorphic forms on unitary groups, see, e.g.\, \cite[Sections 2.6.8 and 2.6.9]{HELS}.  There are also some standard choices of Hecke operators, generalizing the familiar Hecke operators $T_\ell$ from modular forms to the setting of unitary groups.

In analogue with the case of modular forms (on $\GL_2$), the Hecke operators generate a Hecke algebra.  We briefly summarize some key aspects of Hecke algebras here. (See, e.g.\, \cite[Section 1.4]{bellaiche} for a more detailed summary of Hecke algebras in our setting.)  Given a field $L$ of characteristic $0$ and a place $v$ of $\realfield$, we denote by $\CH(G(\realfield_v), \cpct_v, L)$ the algebra of compactly supported $L$-valued functions that are both left and right $\cpct_v$-invariant, together with the operation $*$ defined for $f_1, f_2\in\CH(G(\realfield_v), \cpct_v, L)$ by $(f*h)(g) = \int_{G(\realfield_v)} f(x)h(x^{-1}g)\,dx$. This is the {\em Hecke algebra} of $G$ with respect to $\cpct_v$ over $L$.  (For a more detailed introduction, see, e.g.\, \cite[Sections 3 and 4]{getz} or \cite[Section 2.1]{bellaiche}.) Given a smooth representation of $V$ of $G(\realfield_v)$ over $L$ (i.e.\ $V$ is an $L$-vector space on which $G(\realfield_v)$ acts continuously and such that each vector of $V$ is invariant under some compact open subgroup of $G(\realfield_v)$), $V^{\cpct_v}$ has a natural structure of a $\CH(G(\realfield_v), \cpct_v, L)$-module.  If $v$ does not ramify in $\cmfield$ and $\cpct_v$ is a maximal compact hyperspecial subgroup of $G(\realfield_v)$, then $\CH(G(\realfield_v), \cpct_v, L)$ is a commutative algebra.  If, furthermore, $V$ is an unramified representation of $G(\realfield_v)$ (i.e.\ $\dim V^{\cpct_v} = 1$), then the Hecke algebra $\CH(G(\realfield_v), \cpct_v, L)$ acts on $V^{\cpct_v}$ through a character
\begin{align*}
\CH(G(\realfield_v), \cpct_v, L)\rightarrow L.
\end{align*}
Thanks to a structure theorem proved by Satake, we have that $\CH(G(\realfield_v), \cpct_v, L)$ is finitely generated as an $L$-algebra, and, furthermore, when we specialize to the case of a representation $\pi$ like above, the collection of the values of these characters on a set of generators for $\CH(G(\realfield_v), \cpct_v, L)$ is an element of the semisimple conjugacy class $\sigma_v(\pi)$ associated to $\pi$ above \cite{satake}.
\begin{rmk}
For context, we remark briefly on how this description relates to the familiar cases of $\GL_1$ (Hecke characters) and $\GL_2$ (modular forms).  In the case of $GL_1$, then $\pi = \otimes_v\pi_v$ is a Hecke character $\adeles^\times\rightarrow \IC^\times$ (where $\adeles^\times$\index{$\adeles^\times$} denotes the ideles over $\IQ$).  At each $v$ where $\pi_v$ is unramified, $\pi_v$ is completely determined by its value on a uniformizer $\varpi_v$, and the set of values $\sigma_v(\pi): = \pi_v(\varpi_v)$ completely determines $\pi$.

In the case of $\GL_2$, we consider an irreducible cuspidal automorphic representation $\pi = \otimes_p\pi_p$ generated by a holomorphic cuspform $f(q) = q+\sum_{n\geq 2}a_nq^n$ of weight $k$ and level $1$.  (You can generalize to newforms of level $N$, if you would like.)  In this case, we have $T_p f =a_p f$ for all prime numbers $p$, and $\sigma_p(\pi_p)$ is in the conjugacy class of $\diag(\alpha_p, \beta_p)\in \GL_2(\IC)$, with $\alpha_p\beta_p = p^{k-1}$ and $\alpha_p+\beta_p = a_p$, i.e.\ the Euler factor at $p$ is the familiar Euler factor at $p$ for the $L$-function attached to the cusp form $f$.
\end{rmk}

\subsection{Strategy for proving algebraicity of certain values of $L$-functions}\label{sec:algtools}
In Section \ref{sec:intro}, we observed a connection between rationality of certain values of particular $L$-functions (e.g.\ the Riemann zeta function) and certain Eisenstein series.  For example, in Equation \eqref{equ:G2k}, we observed that the Riemann zeta function arises as the constant term of the Eisenstein series $G_{2k}$, and we noted that rationality of $\zeta(1-2k)$ then followed from properties of $G_{2k}$.  

From Section \ref{sec:intro}, it seems like if we want to prove rationality of $L$-functions more generally, an approach might be: {\em Try to relate the $L$-function in question to an Eisenstein series}.  That is an awfully vague strategy, though.  Supposing some version of this approach even works more generally, which Eisenstein series should we use, and in what sense should we ``relate'' our $L$-function to this Eisenstein series?

Motivated by the example of Shimura's proof in \cite{shimura-CM} of algebraicity of certain values of the Rankin--Selberg convolution that we recounted in Section \ref{sec:intro}, we propose the following recipe for investigating rationality properties:
\begin{enumerate}
\item{Find a Petersson-style pairing of automorphic forms (integrated against an Eisenstein series) that factors into an Euler product, has a functional equation, and can be meromorphically continued to all of $\IC$.}
\item{Prove the rationality of Eisenstein series occurring in that pairing.}
\item{Express a familiar automorphic $L$-function in terms of that pairing, similarly to Equation \eqref{equ:RSconvmf}, to obtain an expression analogous to Expression \eqref{equ:Dmfgalg}.}
\end{enumerate}
In addition to the case of Rankin--Selberg convolutions of modular forms from the introduction, this approach has been carried out in a number of cases, including by Shimura for Rankin--Selberg convolutions of Hilbert modular forms \cite{shimura-hilbert} and by Shimura's PhD student Jacob Sturm for a higher-rank generalization of the Rankin--Selberg convolution \cite{sturm-critical}.   We refer to that method as the {\em Rankin--Selberg method}.  

Although it turns out that the precise pairing in the Rankin--Selberg method does not work in certain settings (including for unitary groups, the main focus of this manuscript), we have substitutes in certain situations. For unitary groups, we have the {\em doubling method}, which is discussed in detail in Section \ref{sec:doubling}.  In addition to Shimura's work in the setting of unitary groups (e.g.\ in as compiled in \cite{sh, shar}), Harris has proved extensive results employing the doubling method to investigate the algebraicity of values of $L$-functions associated to automorphic representations of unitary groups (in particular, \cite[Theorem 3.5.13]{harriscrelle}), as well as associated results about rationality of Eisenstein series, including \cite{harrisannals, harrisbirkhauser, harriscohom1, harriscohom2, harrisdsv}.  One of the keys that enables proofs of algebraicity of certain values of $L$-functions (and associated Eisenstein series) is the realization of automorphic forms as sections over an integral model for a PEL-type moduli space, as in Section \ref{sec:sheafaut}.  

This recipe has also been carried out with certain other pairings.  For example, Harris proved algebraicity of critical values of $L$-functions attached to Siegel modular forms in \cite{hasv}, again relying heavily on the geometry of an associated Shimura variety.  Recently, using a pairing introduced by Aaron Pollack in \cite{pollack1}, the author, Giovanni Rosso, and Shrenik Shah have proved algebraicity of critical values of Spin $L$-functions for $GSp_6$.  Perhaps surprisingly, that pairing involves working with an Eisenstein series defined on a group that has no known Shimura variety or moduli problem.  As these examples show, the above recipe is powerful, even though specific details of how it is carried out in different situations can vary significantly.

\begin{rmk}
Finding a pairing in terms of which one can express an $L$-function, as required for the above three-step recipe, is highly non-trivial and is a serious research problem on its own.  So even though we just provided several examples where the recipe has been successfully carried out, one is not guaranteed to be in a position to carry out even the first step.  Furthermore, even if one has such a pairing, it is not guaranteed that the pairing will be suitable for obtaining results about algebraicity.  (Such pairings are often useful for studying analytic aspects of $L$-functions, hence the interest in them even when they seem not to be well-suited to proving algebraicity results.)  For example, the pairing introduced in \cite{BumpGinzburg} for Spin $L$-functions for $\GSp_{2n}$, for $n=3, 4, 5$, appears not to be amenable to proving algebraicity results.  That is also the case for the approach to ``twisted doubling'' in  \cite{CFGK}.  In general, proofs of algebraicity rely on the pairing and its input having an algebraic or geometric interpretation.
\end{rmk}

\begin{rmk}
Here, we are focusing on the use of Eisenstein series as a tool for proving algebraicity of certain values of $L$-functions.  Eisenstein series also play a key role in governing analytic behavior, like the functional equation and meromorphic continuation, of $L$-functions.  This is the case, for example, with the {\em Langlands--Shahidi method}, which realizes the reciprocals of certain $L$-functions in the constant terms of Fourier expansions of Eisenstein series \cite{shahidibook, shahidi}.  There are also other approaches to investigating algebraicity, e.g.\ \cite{HaRa}.
\end{rmk}

\begin{rmk}\label{rmk:padicrecipe}
As noted in Section \ref{sec:whythesetopics}, all known methods for constructing $p$-adic $L$-functions are adaptations of the specific techniques used to prove algebraicity results for the corresponding $\IC$-valued $L$-functions.  In fact, Haruzo Hida's approach to constructing $p$-adic Rankin--Selberg $L$-functions in \cite{hi85} builds directly on Shimura's proof of algebraicity of Rankin--Selberg convolutions summarized above.  Similarly, the construction of $p$-adic $L$-functions for unitary groups in \cite{HELS, EW} builds on the work with the doubling method.  As noted in Remark \ref{rmk:damerell}, the doubling method specializes in its simplest case to a variant of {\em Damerell's formula}, a key ingredient in the proof of algebraicity for $L$-functions associated to Hecke characters of CM fields, which in turn led to Katz's construction of $p$-adic $L$-functions associated to Hecke characters of CM fields \cite{kaCM}.
\end{rmk}

\begin{rmk}
Our focus is on algebraicity of the values of $L$-functions.  In analogue with Expression \eqref{equ:Dmfgalg}, we are especially interested in deriving results of the form
\begin{align*}
\frac{L(\pi, m)}{\langle \varphi, \varphi\rangle}\in \pi^c\bar{\IQ},
\end{align*}
where $L(\pi, m)$ is the value at some integer $m$ of an $L$-function associated to an irreducible cuspidal automorphic representation $\pi$ that contains a cusp form $\varphi$ and $\pi^c$ denotes a power of the transcendental number $\pi$.  (Results of this form arise from applications of {\em integral representations of $L$-functions}, which express certain automorphic $L$-functions as integrals of automorphic forms, like in the doubling method in Section \ref{sec:doubling} below and the Rankin--Selberg method.  This approach is seen, for example, in the beginnings of the introductions to \cite{PSS, FuMo}.) In Section \ref{sec:intro}, we also mentioned Deligne's conjectures about the meanings of the values of $L$-functions at certain points.  Although our focus in the next sections will be on proving algebraicity results like those above, we briefly bring up this still more challenging aspect now.  Deligne defined an integer $m$ to be {\em critical} for an $L$-function if neither $m$ nor the point symmetric to it with respect to the central point of the $L$-function are poles of the $\Gamma$-factors of the $L$-function (analogues of the factors that show up in the functional equation for the Riemann zeta function like at the beginning of Section \ref{autLfcnssection}).  Deligne predicts that the {\em critical values} (i.e.\ values at the critical points) of an $L$-function associated to a {\em motive} $M$ are rational multiples of the {\em period} of $M$ (an algebraic invariant of $M$ coming from cohomology).  In general, the values of automorphic $L$-functions are not readily seen to satisfy Deligne's conjecture, even when we have algebraicity results.  By exploiting geometry in the unitary setting, though, Harris explained connections between his algebraicity results and motivic periods in the case of $\realfield =\IQ$ in \cite{harriscrelle}, and this work was later extended by his PhD student Lucio Guerberoff to the case of $\realfield$ of degree $>1$ in \cite[Theorem 4.5.1]{guerberoff}.  These results and related ones are surveyed in \cite{harrislin}.
\end{rmk}

\begin{rmk}
In the above three-step recipe, we skipped over the fact that the Eisenstein series with which we need to work are not necessarily holomorphic.  At the beginning of this manuscript, in Equation \eqref{equ:G2k}, we recalled the holomorphic Eisenstein series $G_{2k}$.  Shortly after that, though, we encountered a non-holomorphic Eisenstein series in Equation \eqref{equ:RSconvmf}.  The prototypical example from the setting of modular forms of weight $\lambda$, level $N$, and character $\chi$ is 
\begin{align*}
E_{\lambda, N}(z, s, \chi) = \sum_{(0, 0)\neq(m, n)\in\ZZ\times\ZZ} \frac{\chi(n)}{(mNz+n)^\lambda |mNz+n|^{2s}},
\end{align*}
which converges for $\re(2s)>2-\lambda$.
In fact, the Eisenstein series denoted by $E$ in Equation \eqref{equ:RSconvmf} is the holomorphic modular form $E^\ast_{\lambda, N}(z, \chi)$ defined to be specialization to $s=0$ of $E^\ast_{\lambda, N}(z, s, \chi):=\frac{E_{\lambda, N}(z, s, \chi)}{2L(2s+\lambda, \chi)}$, but we also need to work with Eisenstein series at points $s$ where the Eisenstein series $E^\ast_{\lambda, N}(z, s, \chi)$ is not holomorphic.  Such Eisenstein series can be obtained via application of the Maass--Shimura operator $\delta_\lambda^{(r)}$ from Section \ref{sec:intro}, which is defined by $\delta_\lambda^{(r)}:=\delta_{\lambda+2r-2}\circ\cdots \circ \delta_{\lambda+2}\circ\delta_\lambda$ with $\delta_\lambda:=\frac{1}{2\pi i}\left(\frac{\lambda}{2iy}+\frac{\partial}{\partial z}\right)$.  We have 
\begin{align*}
E^\ast_{\lambda+2r, N}(z, -r, \chi) = \frac{\Gamma(\lambda)}{\Gamma(\lambda+r)}(-4\pi y)^r\delta_\lambda^{(r)}E_{\lambda, N}^\ast(z, \chi),
\end{align*}
i.e.\ $E^\ast_{\lambda+2r, N}(z, -r, \chi)$ is a scalar multiple of the term $\delta_\lambda^{(r)}E$ in Expression \eqref{equ:RSconvmf} from Section \ref{sec:intro}.  

At $s\neq 0$, the Eisenstein series $E_{\lambda, N}^\ast(z, s, \chi)$ are clearly not holomorphic, although they are $\ci$.  This might seem potentially problematic for carrying out our recipe concerning algebraicity.  Fortunately, though, the operators $\delta_\lambda$ can be reformulated geometrically over a modular curve $\moduliint_N(\IC)$, and that geometric formulation provides enough structure to preserve algebraic aspects of modular forms.  In brief, the map of sheaves 
\begin{align*}
\mathcal{E}_\lambda\rightarrow\mathcal{E}_{\lambda+2}
\end{align*}
corresponding to $\delta_\lambda$ is induced by the composition of maps
\begin{align}\label{diffopcomp}
\uo^{\otimes k}\hookrightarrow (\hdrone)^{\otimes k}\xrightarrow{\nabla} (\hdrone)^{\otimes k}\otimes\Omega_{\moduliint_N}\isomto(\hdrone)^{\otimes k}\otimes\uo^{\otimes 2}\twoheadrightarrow\uo^{k+2},
\end{align}
where $\hdrone:={\mathbf R}^1_{\pi_\ast}\left(\Omega^\bullet_{\CA/\moduliint_N}\right)$, $\nabla$ is the Gauss--Manin connection, the isomorphism is the identity map tensored with the Kodaira--Spencer morphism, and the surjection is the projection onto the first factor of the Hodge decomposition $\uo\oplus H^{0, 1}$.  Each of these maps is algebraic and can be defined over any $\mathcal{O}_{\reflex, (p)}$-algebra, except for the final one.  The final map has enough structure, though, that it preserves essential algebraic structure.  In particular, when applied to a holomorphic modular form, $\delta_\lambda$ preserves algebraicity at CM points, which is essential for proving algebraicity via the aforementioned Damerell's formula.  In addition, the {\em holomorphic projection} $H(g\delta_\lambda^{(r)}E)$ is a holomorphic modular form with the property $\langle f, H(g\delta_\lambda^{(r)}E)\rangle = \langle f, g\delta_\lambda^{(r)}E\rangle$, so $\langle f, H(g\delta_\lambda^{(r)}E)\rangle$ still becomes a scalar multiple of $\langle f, f \rangle$ and we still obtain Expression \eqref{equ:Dmfgalg}.

The composition of maps \eqref{diffopcomp} can naturally be formulated over our PEL moduli space $\moduliint_\cpct$ to construct differential operators on unitary groups, and these operators preserve similar algebraic properties of automorphic forms in this setting.  We do not elaborate on them here.  Note, though, that more book-keeping is involved in this setting, the resulting forms in this case can be vector-valued, and the Kodaira--Spencer map in this case is
\begin{align*}
\Omega_{\moduliint_{\cpct}}\isomto\otimes_{\tau\in\mathcal{T}_{\realfield}}\uo_\tau^+\otimes\uo_\tau^-.
\end{align*}
The Maass--Shimura differential operators and their algebraicity properties have been studied in detail for automorphic forms on unitary groups and related cases in, e.g.\, \cite{hasv, ha86, EDiffOps, EFMV, EiMa2, shar, shclassical, shclassnearhol, kaCM, zheng}.
\end{rmk}

\begin{rmk}
In the remainder of this section, we highlight the ingredients needed to carry out our three-step recipe for proving algebraicity for automorphic $L$-functions associated to cuspidal automorphic representations of unitary groups.  We emphasize aspects that are unlikely to be viewed as ``straightforward'' extensions of the tools occurring in the setting of Shimura's work with Rankin--Selberg convolutions of modular forms, and consequently, a detailed introduction to the doubling method forms a large portion of the remainder of this section.
\end{rmk}

\subsection{The doubling method, in the setting of unitary groups}\label{sec:doubling}
We now introduce the aforementioned {\em doubling method}, which provides an integral representation of our $L$-functions, i.e.\ a global integral that unfolds to a product of local integrals in terms of which we can realize an Euler product for our $L$-functions.  The approach we present here is due to \cite{PSR, ga}.  There are also helpful notes in \cite{cogdell}.  The doubling method is an instance of a {\em pullback method} or Rankin--Selberg style integral, i.e.\ we will be integrating a pullback of an Eisenstein series against a pair of cusp forms.  As noted in \cite[Section 1]{PSR}, the doubling method is based on the Rankin--Selberg method (a generalization to $\GL_n$ of the Rankin--Selberg convolution from Equation \eqref{equ:RSconvmf}), but unlike the Rankin--Selberg method, it does not require uniqueness of Whittaker models (and hence allows us to work with unitary groups).  The doubling method can also be formulated algebraically geometrically so that we can study algebraic aspects of values of $L$-functions. 

\subsubsection{Setup for the doubling method}\label{sec:doublingsetup}
To begin, let 
\begin{align*}
U := U(V, \langle, \rangle)
\end{align*}
 be as in Definition \ref{unitarygpdefn} (with $V$ an $n$-dimensional vector space over the CM field $\cmfield$).  Let $W=V\oplus V$.  We define a Hermitian pairing $\langle, \rangle_W$ on $W$ by
\begin{align*}
\langle (u, v), (u', v')\rangle_W:=\langle u, u'\rangle_V-\langle v, v'\rangle_V
\end{align*}
for all $u, v, u', v'\in V$.  Following the conventions of \cite[Equation (1.1.7)]{sh}, we sometimes denote the pairing $\langle, \rangle_W$ by 
\begin{align}\label{opluspairing}
\langle,\rangle\oplus-\langle, \rangle. 
\end{align}\index{$\langle,\rangle\oplus-\langle, \rangle$}\ignorespacesafterend  Let 
\begin{align*}\index{$U_W$}
U_W:=U(W, \langle, \rangle_W).
\end{align*}  Then $U_W$ is a unitary group of signature $(n,n)$ at each place, and we have an embedding
\begin{align}\label{doublingembedding}
U\times U = U(V, \langle, \rangle)\times U(V, -\langle, \rangle)\hookrightarrow U_W,
\end{align}
via $(g, g')(u, v) := (gu, g'v)$ for all $g, g'\in G$ and $u, v\in V$.
This {\em doubling} of our group is what leads to the name {\em doubling method}.
(This is the formulation of unitary groups in the original setup for the doubling method in \cite{PSR, ga}.  In Section \ref{sec:doublingPELvarieties}, we will also address the relationship with the unitary groups associated to PEL data, by giving a set of PEL data that induces an analogous embedding.)  \begin{rmk}
The doubling method is formulated in terms of integrals over the groups $U(\adeles_{\realfield})$ and $U(\realfield_v)$ and subquotients of these groups.  For these integrals, we work with a Haar measure on these groups.  For the purposes of this manuscript, we do not need to be more precise.  The reader seeking more information about appropriate choices of normalizations could consult \cite[Section 1.4.3]{HELS}.
\end{rmk}

Let $P$\index{$P$ Siegel parabolic subgroup of $U_W$} be the Siegel parabolic subgroup of $U_W$ preserving the maximal isotropic subspace
\begin{align*}\index{$V^\Delta$}
V^\Delta:=\left\{(v, v)\mid v\in V\right\}\subset W.
\end{align*}
We also define
\begin{align*}\index{$V_\Delta$}
V_\Delta:=\left\{(v, -v)\mid v\in V\right\}\subset W.
\end{align*}
With respect to the decomposition $W=V^\Delta\oplus V_\Delta$, we have that for each $\realfield$-algebra $R$,
\begin{align}\label{equ:PorR}
P(R)=\left\{\begin{pmatrix}A & 0\\ 0 & { }^t\bar{A}^{-1}\end{pmatrix}\begin{pmatrix}1_n & X\\ 0& 1_n\end{pmatrix}\mid A\in  GL_{\cmfield\otimes_{\realfield} R}(V\otimes_{\realfield} R) \mbox{ and } X\in\hern(\cmfield\otimes_{\realfield}R)\right\}.
\end{align}
Given a $\realfield$-algebra $R$ and a character $\psi$ of $(\cmfield\otimes_{\realfield}R)^\times$, we view $\psi$ as a character of $P(R)$ via
\begin{align*}
\begin{pmatrix}A & B\\ 0 & { }^t\bar{A}^{-1}\end{pmatrix}\mapsto\psi(\det A).
\end{align*}

\subsubsection{Siegel Eisenstein series}
Let $\chi: \cmfield^\times\backslash\adeles_\cmfield^\times\rightarrow \IC^\times$ be a unitary Hecke character.  Given $s\in\IC$, let 
\begin{align*}\index{$f_{s,\chi}$}
f_{s,\chi}: U_W(\adeles_{\realfield})\rightarrow \IC
\end{align*}
 be an element of
\begin{align*}\index{$I(s, \chi)$}
I(s, \chi):=\Ind_{P(\adeles_{\realfield})}^{U_W(\adeles_{\realfield})}(\chi| \cdot |^{-s}):=\left\{f: U_W(\adeles_{\realfield})\rightarrow \IC | f(ph) = \chi(p)|p|^{-s+n/2}f(h)\right\},
\end{align*}
with the absolute value here denoting the adelic norm.  (In practice, we usually restrict ourselves to the smooth, $\cpct$-finite functions in $I(s, \chi)$.)
 The elements $\alpha\in U_W(\adeles_{\realfield})$ act on elements 
 \begin{align*}
 f\in \Ind_{P(\adeles_{\realfield})}^{U_W(\adeles_{\realfield})}(\chi| \cdot |^{-s})
 \end{align*}
  via
\begin{align*}
(\alpha f)(h) = f(h\alpha)
\end{align*}
for all $h\in U_W(\adeles_{\realfield})$.  The induced representation $I(s, \chi)$ factors as a restricted tensor product $\otimes_v\Ind_{P(\realfield_v)}^{U_W(\realfield_v)}\left(\chi_v|\cdot|^{-s}\right)$, with the induced representation for the local groups defined analogously.

We define a {\em Siegel Eisenstein series} on $U_W$ by
\begin{align}\index{$E_{f_{s, \chi}}$}\label{equ:siegeleseries}
E_{f_{s, \chi}}(h):=\sum_{\gamma\in P(\realfield)\backslash U_W(\realfield)}f_{s, \chi}(\gamma h).
\end{align}
This series converges absolutely and uniformly for $\Real(s)>\frac{n}{2}$.  In addition, $E_{f_{s, \chi}}$ is a meromorphic function of $s$ and is an automorphic form in $h$.  Define a Hecke character $\check{\chi}$ by
\begin{align*}
\check{\chi}(x):=\chi(\bar{x})^{-1},
\end{align*}
let $N_P$ denote the unipotent radical of $P$, and let $w$ be the Weyl element interchanging $V^\Delta$ and $V_{\Delta}$.  The Eisenstein series $E_{f_{s, \chi}}(h)$ satisfies the functional equation
\begin{align*}
E_{f_s, \chi}(h) = E_{M(s, \chi)f_{s, \chi}}(h),
\end{align*}\index{$M(s, \chi)$}\ignorespacesafterend
where $M(s, \chi): I(s, \chi)\rightarrow I(1-s, \check{\chi})$ is the intertwining operator defined for $\Real(s)>\frac{n}{2}$ by
\begin{align*}
[M(s, \chi)f_{s, \chi}](h)&=\int_{N_P(\adeles_{\realfield})}f_{s, \chi}(wnh)\,dn\\
&=\int_{\hern(\adeles_\cmfield)} f_{s,\chi}\left(w\begin{pmatrix}1& X\\ 0 & 1\end{pmatrix} h\right)\,dX.
\end{align*}

\subsubsection{The doubling integral}
Let $\pi$ be an irreducible cuspidal automorphic representation of $U$, and let $\pi'$ denote the contragredient of $\pi$.  Given $\varphi\in\pi$ and $\varphi'\in\pi'$, we define the doubling integral by
\begin{align}\label{doublintegral}\index{$Z\left(\varphi, \varphi', f_{s, \chi}\right)$ the doubling integral}
Z\left(\varphi, \varphi', f_{s, \chi}\right):= \int_{[U\times U](\realfield)\backslash[U\times U](\adeles_{\realfield})}E_{f_{s, \chi}}(g, h)\varphi(g)\varphi'(h)\chi^{-1}(\det h)\,dg\,dh.
\end{align}
\noindent The notation $[U\times U]$\index{$[U\times U]$} here denotes the image of $U\times U$ inside $U_W$.
\begin{rmk}
The character $\chi$ is absent from the original treatment in \cite{PSR}, but it is a straightforward exercise to show that if we define the Eisenstein series as above to include $\chi$, then the doubling integral must take this form.  Indeed, we include $\chi$ in our analysis below.
\end{rmk}
The integral $Z\left(\varphi, \varphi', f_{s, \chi}\right)$ inherits the analytic properties of Eisenstein series.  In particular, $Z\left(\varphi, \varphi', f_{s, \chi}\right)$ extends to a meromorphic function of $s$ and satisfies the functional equation
\begin{align*}
Z\left(\varphi, \varphi', f_{s, \chi}\right)=Z\left(\varphi, \varphi', M(s, \chi)f_{s, \chi}\right).
\end{align*}
The doubling integral $Z\left(\varphi, \varphi', f_{s, \chi}\right)$ plays a role in our setting analogous to the role played by the Rankin--Selberg integral for $\GL_n$.

\begin{rmk}\label{rmk:damerell}
Note that in the case of definite groups (i.e.\ signature$(n,0)$), we can choose our input data so that the integral unfolds into a finite sum of values of automorphic forms and can be reinterpreted as values of those automorphic forms at CM points of the corresponding Shimura variety on which the Eisenstein series is defined.  In the special case where $n=1$, this allows us to recover a variant of {\em Damerell's formula}, an expression of values $L(0, \chi)$, with $\chi$ a Hecke character of type $A_0$ on a CM field, as a finite sum of values of a Hilbert modular form at CM Hilbert--Blumenthal abelian varieties.
Damerell initiated the study of the algebraicity properties of these values, and this study was later completed by Goldstein--Schappacher \cite{GS1, GS2}, Shimura \cite{shimura-CM}, and Weil \cite{weil1}.  For those seeking more details, note that a nice history of the development of Damerell's formula is provided in \cite[\S5]{harder-schappacher}.
\end{rmk}

\subsubsection{Unfolding the doubling integral into an Euler product}

\begin{thm}\label{doubletwotoone}We have the equality
\begin{align*}
Z\left(\varphi, \varphi', f_{s, \chi}\right)=\int_{U(\adeles_{\realfield})}f_{s, \chi}(g, 1)\langle\pi(g)\varphi, \varphi'\rangle \,dg,
\end{align*}
where
\begin{align}\label{equ:invtpairing}
\langle \varphi, \varphi'\rangle:=\int_{U(\realfield)\backslash U(\adeles_{\realfield})}\varphi(g)\varphi'(g)\,dg.
\end{align}
\end{thm}
\begin{rmk}
The pairing $\langle, \rangle$\index{$\langle, \rangle$ unique $U$-invariant pairing} is the unique $U$-invariant pairing of $\pi$ with $\pi'$, up to a constant multiple.  We likewise also write $\langle, \rangle$ for the corresponding integral over $U(\realfield_v)$.
\end{rmk}

\begin{rmk}
It is natural to wonder what we might get if we replace $\pi'$ by some other irreducible cuspidal automorphic representation $\tilde{\pi}\ncong\pi'$ and $\varphi'\in \tilde{\pi}.$  In that case, we would have $\langle \pi(g)\varphi, \varphi'\rangle=0$, and as a consequence of Theorem \ref{doubletwotoone}, we would then have that $Z\left(\varphi, \varphi', f_{s, \chi}\right)$ is identically zero.\end{rmk}

\begin{rmk}\label{rmkdconds}
Going forward, we suppose we have factorizations $\pi=\otimes_v'\pi_v$ and $\pi'=\otimes'_v\pi_v'$ (where these restricted tensor products are over the places of $\realfield$), and we suppose $\varphi=\otimes_v'\varphi_v$ with $\varphi_v\in\pi_v$ and $\varphi'=\otimes_v'\varphi_v'$ with $\varphi_v'\in\pi_v'$.  Let $S$ be the set of all finite places of $\realfield$ that ramify in $\cmfield $ or where $\pi$, $\pi'$, and $\chi$ are ramified.  For all finite places $v\nin S$, let $\varphi_v\in \pi_v$ and $\varphi_v'\in\pi_v'$ be normalized $\cpct$-fixed vectors such that $\langle \varphi_v, \varphi_v'\rangle = 1$.
We also suppose that $f_{s, \chi} = \otimes_v'f_{s, \chi_v}$, with $f_{s, \chi_v}\in I_v(s, \chi):=\Ind_{P(\realfield_v)}^{H(\realfield_v)}\left(\chi_v|\cdot|^{-s}\right)$.\index{$I_v(s, \chi)$}
\end{rmk}

\begin{cor}\label{coroEP}
The integral $Z\left(\varphi, \varphi', f_{s, \chi}\right)$ factors as an Euler product (over all places $v$ of $\realfield$):
\begin{align*}
Z\left(\varphi, \varphi', f_{s, \chi}\right)=\langle \varphi, \varphi'\rangle\prod_vZ_v(\varphi_v, \varphi_v', f_{s, \chi_v}),
\end{align*}
where
\begin{align*}
Z_v(\varphi_v, \varphi'_v, f_{s, \chi_v})=\frac{\int_{U(\realfield_v)}f_{s, \chi_v}(g_v, 1)\langle\pi_v(g_v)\varphi, \varphi'\rangle \,dg_v}{\langle \varphi_v, \varphi_v'\rangle},
\end{align*}
with the denominator equal to $1$ for all $v\nin S$.
\end{cor}

\begin{proof}[Proof of Corollary \ref{coroEP}]
By the uniqueness of the $U$-invariant pairing between $\pi$ and $\pi'$, we have
\begin{align*}
\langle \pi(g)\varphi,  \varphi'\rangle=c\prod_v\langle \pi_v(g_v)\varphi_v, \varphi_v'\rangle,
\end{align*}
with $c = \frac{\langle \varphi, \varphi'\rangle}{\prod_v\langle\varphi_v, \varphi_v'\rangle}.$
The corollary now follows immediately from Theorem \ref{doubletwotoone}.
\end{proof}

\begin{proof}[Proof of Theorem \ref{doubletwotoone}]
The theorem will follow from an analysis of the orbits of $U\times U$ acting on $X:=P\backslash U_W$, which we now explain.  
We write $[U\times U]^\gamma$ for the stabilizer of a point $\gamma\in X$.
We re-express the Eisenstein series $E_{f_{s, \chi}}(h)$ as
\begin{align*}
E_{f_{s, \chi}}(h)=\sum_{[\gamma]\in P(\realfield)\backslash U_W(\realfield)/[U\times U](\realfield)}\left(\sum_{[\gamma_0]\in[U\times U]^\gamma(\realfield)\backslash[U\times U](\realfield)}f_{s, \chi}(\gamma\gamma_0 h)\right),
\end{align*}
where, for each $\gamma\in U_W(\realfield)$, $[U\times U]^\gamma(\realfield)$ is the stabilizer of $P(\realfield)\gamma\in P(\realfield)\backslash U_W(\realfield)$ under the right action of $[U\times U](\realfield)$ and $[\gamma]$ is the orbit of $P(\realfield)\gamma$ in $P(\realfield)\backslash U_W(\realfield)$ under the right action of $[U\times U](\realfield)$.

Inserting this expression for the Eisenstein series into the doubling integral, we have
\begin{align*}
Z\left(\varphi, \varphi', f_{s, \chi}\right)
&=\sum_{[\gamma]\in P(\realfield)\backslash U_W(\realfield)/[U\times U](\realfield)}\sum_{[\gamma_0]\in[U\times U]^\gamma(\realfield)\backslash[U\times U](\realfield)}\\
&\left(\int_{[U\times U](\realfield)\backslash[U\times U](\adeles_{\realfield})}f_{s, \chi}\left(\gamma\gamma_0(g, h)\right)\varphi(g)\varphi'(h)\chi^{-1}(\det h)\,dg\,dh\right)\\
&=\sum_{[\gamma]\in P(\realfield)\backslash U_W(\realfield)/[U\times U](\realfield)}I(\gamma),
\end{align*}
where
for each $\gamma\in U_W(\realfield)$, 
\begin{align*}
I(\gamma) := \int_{[U\times U]^\gamma(\realfield)\backslash[U\times U](\adeles_{\realfield})}f_{s, \chi}\left(\gamma(g, h)\right)\varphi(g)\varphi'(h)\chi^{-1}(\det h)\,dg\,dh.
\end{align*}
Note that for each $\gamma\in U_W(\realfield)$,
\begin{align*}
[U\times U]^\gamma(\realfield)&=\left\{(g, h)\in [U\times U](\realfield) \mid P(\realfield)\gamma (g, h) = P(\realfield)\gamma  \right\}\\
&=\left\{(g, h)\in [U\times U](\realfield) \mid \gamma (g, h)\gamma^{-1} \in P(\realfield) \right\}.
\end{align*}

First we consider the case of the identity $\gamma=\gamma_0=1\in U_W(\realfield)$.  The stabilizer of the identity $\gamma_0=1\in U_W(\realfield)$ is
\begin{align*}
[U\times U]^{\gamma_0}(\realfield)=P(\realfield)\cap [U\times U](\realfield) =\left\{(g, g)\mid g\in U(\realfield)\right\}=:U^\Delta(\realfield).
\end{align*}
In this case, we have 
\begin{align*}
f_{s, \chi}\left(\gamma_0(g, h)\right) = f_{s, \chi}\left((g, h)\right) = f_{s, \chi}((h, h)(h^{-1}g, 1)) = \chi(\det(h))f_{s, \chi}(h^{-1}g, 1).
\end{align*}
So
\begin{align*}
I(\gamma_0)& =\int_{U^\Delta(\realfield)\backslash [U\times U](\adeles_{\realfield})}f_{s, \chi}\left(h^{-1}g, 1\right)\varphi(g)\varphi'(h)\,dg\,dh.
\end{align*}
Noting that
\begin{align*}
U\times U\cong & U^\Delta \times (U\times 1)\cong U\times U\\
(g, h)\leftrightarrow &(h, h)(h^{-1}g, 1)\leftrightarrow (h, h^{-1}g)
\end{align*}
and writing $g_1 = h^{-1}g$,
we have
\begin{align*}
I(\gamma_0)& =\int_{U(\adeles_{\realfield})}\int_{U(\realfield)\backslash U(\adeles_{\realfield})}f_{s, \chi}\left(g_1, 1\right)\pi(g_1)\varphi(h)\varphi'(h)\,dh\,dg_1\\
& = \int_{U(\adeles_{\realfield})}f_{s, \chi}\left(g_1, 1\right)\langle \pi(g_1) \varphi, \varphi' \rangle \,dg_1.
\end{align*}

The other orbits are {\em negligible}, i.e.\ the stabilizer of a point in the orbit contains the unipotent radical of a proper parabolic subgroup of $[U\times U](\realfield)$ as a normal subgroup.  (This definition of {\em negligible} comes from \cite[\S1, p.2 ]{PSR}.)  We will now show that for each negligible orbit $[\gamma]$,
\begin{align*}
I(\gamma) = 0,
\end{align*}
thus completing this proof (and also justifying the name {\em negligible}).  For the remainder of this proof, suppose that $\gamma$ belongs to a negligible orbit, i.e.\ there is a proper parabolic subgroup of $U\times U$ whose unipotent radical $N^\gamma$ is a normal subgroup of $[U\times U]^\gamma$.  We write
\begin{align*}
I(\gamma) = \int_{[U\times U]^\gamma(\adeles_{\realfield})\backslash[U\times U](\adeles_{\realfield})}I(\gamma, h_1, h_2)\,dh_1\,dh_2,
\end{align*}
where
\begin{align*}
&I(\gamma, h_1, h_2)\\
&=\int_{[U\times U]^\gamma(\realfield)\backslash[U\times U]^\gamma(\adeles_{\realfield})}f_{s, \chi}\left(\gamma(g_1, g_2)(h_1, h_2)\right)\varphi(g_1h_1)\varphi'(g_2h_2)\chi^{-1}(\det(g_2h_2))\,dg_1\,dg_2.
\end{align*}
We will now prove that $I(\gamma, h_1, h_2)=0$, thus completing the proof.  Let
\begin{align*}
M:=N^\gamma\backslash [U\times U]^\gamma.
\end{align*}
We further decompose that integral as
\begin{align*}
I(\gamma, h_1, h_2) = \int_{M(\realfield)\backslash M(\adeles_{\realfield})}I(\gamma, h_1, h_2, m_1, m_2)\,dm_1\,dm_2,
\end{align*}
where
\begin{align*}
I(\gamma, h_1, h_2, m_1, m_2)&=\int_{N^\gamma(\realfield)\backslash N^\gamma(\adeles_{\realfield})}F_{s, \chi, \gamma, h_1, h_2, m_1, m_2}(n_1, n_2)\,dn_1\,dn_2,\\
F_{s, \chi, \gamma, h_1, h_2, m_1, m_2}(n_1, n_2)&:=\\
f_{s, \chi}(\gamma(n_1, n_2)&(m_1, m_2)(h_1, h_2))\varphi(n_1m_1h_1)\varphi'(n_2m_2h_2)\chi^{-1}(\det(n_2m_2h_2)).
\end{align*}
Since $N^\gamma$ is the unipotent radical of a proper parabolic subgroup of $U\times U$, we can write
\begin{align*}
N^\gamma=N_1\times N_2,
\end{align*}
with each subgroup $N_i$ the unipotent radical of a parabolic subgroup of $U$.  So noting that $\det n_i=1$ and writing $\gamma (n_1, n_2)= p\gamma$ for some $p\in P$, we have
\begin{align*}
&I(\gamma, h_1, h_2)
\\&=\int_{M(\realfield)\backslash M(\adeles_{\realfield})}f_{s, \chi}(p\gamma(m_1, m_2)(h_1, h_2))I_1(m_1, h_1)I_1(m_2, h_2)\chi^{-1}(\det(m_2 h_2))\,dm_1\,dm_2,
\end{align*}
where, for $i=1,2$,
\begin{align*}
I_i(m_i, h_i) = \int_{N_i(\realfield)\backslash N_i(\adeles_{\realfield})}\varphi_i(n_i m_i h_i) \,dn_i,
\end{align*}
with $\varphi_1=\varphi$ and $\varphi_2=\varphi'$.
Since $N$ is nontrivial, at least one of the subgroups $N_i$ is nontrivial.  So since $\varphi$ and $\varphi'$ are cusp forms, $I_i(m_i, h_i)=0$ for at least one of $i=1,2$.  So $I(\gamma)=0$ for each $\gamma$ in a negligible orbit.

Finally, to conclude the proof of Theorem \ref{doubletwotoone}, we need to show that if $\gamma\nin[1]$, then $\gamma\in X=P\backslash U_W$ is in a negligible orbit.  
First, note that since $U_W$ acts transitively on the space of maximal isotropic subspaces of $W$ and $P$ stabilizes the maximal isotropic subspace $V^\Delta$, we identify $X = P\backslash U_W$ with the variety that parametrizes the maximal isotropic subspaces of $W$.  Under this identification, the $U\times U$-orbit of $P1\in P\backslash U_W$ corresponds to the $U\times U$-orbit of the maximal isotropic subspace $V^\Delta$.  We need to show that the other orbits are negligible, i.e.\ the stabilizer contains the unipotent radical of a proper parabolic subgroup of $U\times U$ as a normal subgroup.  Let $V^+=V\times\langle0\rangle\subseteq W$, and $V^-=\langle0\rangle\times V\subseteq W$.  By \cite[Lemma 2.1]{PSR}, for any maximal isotropic subspace $L\subseteq W$, we have $\dim_K L\cap V^+=\dim_K L\cap V^-$, and furthermore, the $U\times U$-orbits of maximal isotropic subspaces parametrized by $X$ are the spaces $X_d\subseteq X$ defined by
\begin{align*}
X_d := \left\{L\subseteq W \mid L \mbox{ is a maximal isotropic subspace such that }\dim_K(L\cap V^+) = d\right\}
\end{align*}
Observe that the orbit of $V^\Delta$ is then $X_0$.  

For the remainder of the proof, fix $d>0$, and let $L\in X_d.$  Let $R\subset U\times U$ be the stabilizer of $L$.  To complete the proof, we will find a proper parabolic subgroup of $U\times U$ whose unipotent radical is normal in $R$.  Let $P^{\pm}$ denote the parabolic subgroup of $U$ preserving the flag $V\supset \pi^{\pm}(L) \supset L^{\pm}$, where $\pi^{\pm}$ is the orthogonal projection of $W$ onto $V^\pm$ and $L^\pm:=L\cap V^\pm$.  Note that since $d>0$, $P^{\pm}$ is a proper parabolic subgroup of $U$.  So $P^+\times P^-$ is a proper parabolic subgroup of $U\times U$.  Let $N_d=N^+\times N^-$, with $N^\pm$ the unipotent radical of $P^\pm$.  To conclude the proof, it suffices to show that $N_d$ is a normal subgroup of $R$.  Note that $R\subset P^+\times P^-$.  So since $N_d$ is normal in $P^+\times P^-$, it suffices to show that $N \subset R.$  Note that $N^\pm$ induces the identity on $\pi^{\pm}(L)/L^\pm$.  Now consider the projections $\pi_{\pm}: \pi^\pm(L)\rightarrow \pi^\pm(L)/L^\pm$.  Then $\pi_{\pm}(n v) = \pi_{\pm} (v)$ for all $v\in L^{\pm}$ and $n\in N_d$.  Finally, note that for the maximal isotropic subspace, there is an isometry $\iota: \pi^{\pm}(L)/L^\pm\rightarrow \pi^{\mp}(L)/L^{\mp}$ determined by $\iota(\pi_+ v_+) = \pi_-v_-$ if and only if $(v_+, v_-)\in L$.  So for all $n=(n_+, n_-)\in N_d$, $\iota(\pi_+ \left(n_+v_+\right)) =\iota (\pi_+v_+) = \pi_-(v_-) = \pi_-(n_-v_-).$  So $N_d\subset R$, as desired.
\end{proof}

\begin{rmk}
The doubling method actually applies not only to unitary groups but to classical groups more generally.  Indeed, our approach above often did not use features unique to unitary groups but relied primarily on an embedding $U\times U\hookrightarrow U_W$ of a group into its ``doubled group'' and the following two properties, which can also be extended to other classical groups:
\begin{enumerate}
\item{The stabilizer of the identity $\gamma_0=1$ is $U^\Delta$.}
\item{All the orbits other than the orbit of $\gamma_0=1$ are negligible.}
\end{enumerate}
For example, explicit constructions for symplectic and orthogonal groups are also discussed in \cite[Section 2]{PSR}.  In practice, \cite{PSR} treats the input and construction axiomatically and then shows how this axiomatic treatment can be applied for each of the classical groups.\end{rmk}

\subsubsection{From local doubling integrals to Euler factors for standard Langlands $L$-functions}\label{sec:doubletostd}
We already saw that $Z(f_{s, \chi}, \varphi, \varphi')$ satisfies a functional equation, but we want to relate it to a familiar $L$-function.   In Sections \ref{sec:nsarch} through \ref{sec:archfac}, we give a brief overview of the realization of the Euler factors $L(s, \pi, \chi)$ for the standard Langlands $L$-function in terms of the local integrals $Z_v(\varphi_v, \varphi'_v, f_{s, \chi_v})$.  For the remainder of the manuscript, we set 
\begin{align*}
f_v:=f_{s, \chi_v}.
\end{align*}
  
As noted in \cite[Section 4.1]{MRLnonarch}, at all finite places $v\nin S$, the strategy for computing the local integrals arising in the doubling method for unitary groups (as well as for symplectic groups) is to reduce them to integrals computed by Godement and Jacquet for $\GL_n$ \cite{jacquet}.  For certain other groups, the computations reduce to other integrals computed by Godement and Jacquet in \cite{godement-jacquet}.

\subsubsection{Factors at split nonarchimedean places}\label{sec:nsarch}
At split finite places $v$, $U(\realfield_v)\cong \GL_n(\realfield_v)$, as in Isomorphism \eqref{UGLniso}.  In \cite[Section 3]{PSR} (see also \cite[Section 2.3]{cogdell} and \cite[Section 4.2.1]{HELS}), Ilya Piatetski-Shapiro and Stephen Rallis explain how to choose a local Siegel section $f_v$ so that the computation reduces to integrals computed by Godement and Jacquet for $\GL_n$ in \cite{godement-jacquet}.
At these places, we have
\begin{align}\label{equ:halfshift}
d_{n,v}\left(s, \chi_v\right)Z_v\left(\varphi_v, \varphi_v', f_v, s\right) = L_v\left(s+\frac{1}{2}, \pi_v, \chi_v\right),
\end{align}
where
\begin{align*}
d_{n, v}\left(s, \chi_v\right) = d_{n, v}\left(s\right)=\prod_{r=0}^{n-1}L_v\left(2s+n-r, \chi_v\mid_{\realfield_v}\eta_v^r\right),
\end{align*}
$\eta_v$ is the character on $\realfield_v$ attached by local class field theory to the extension $\cmfield_w/\realfield_v$ (with $w$ a prime of $\cmfield$ lying over $v$).  (N.B.\ In spite of how it might look, the shift by $1/2$ in Equation \eqref{equ:halfshift} above and in Equation \eqref{equ:JShalfshift} below are not typos, but rather are consistent with the corresponding statements in \cite[Theorem 3.1 and end of Section 3]{JSLi}.)  Here $L_v\left(s+\frac{1}{2}, \pi_v, \chi_v\right):=L(s, BC_{\cmfield/\realfield}(\pi_v)\otimes (\chi_v\circ\det))$, where $BC_{\cmfield/\realfield}(\pi_v)$ denotes the base change of $\pi_v$ from $U/\realfield$ to $\GL_n/\cmfield$ (as developed in \cite{labesse}).

\subsubsection{Factors at inert nonarchimedean places}
At inert nonarchimedean primes, the unitary group is not isomorphic to a general linear group.  So we will need a different approach from the one we just employed for split finite primes.  Note that there is a close relationship between the symplectic group $\Sp_{2n}$ and the unitary group $U(n,n)$ of signature $(n,n)$.  In \cite[Section 6]{PSR}, Piatetski-Shapiro and Rallis computed the local integrals for $\Sp_{2n}$.  (They did this by choosing a local section $f_v$ that allows them to reduce their calculation to a calculation of the aforementioned integrals for $\GL_n$ computed Godement and Jacquet.)  In \cite[Section 3]{JSLi}, Jian-Shu Li completed the computation of the local integrals for unitary groups by reduction to these calculations for $\Sp_{2n}$ completed by Piatetski-Shapiro--Rallis, and similarly to the result at split finite places above, we obtain
\begin{align}\label{equ:JShalfshift}
d_{n,v}\left(s, \chi_v\right)I_v\left(\varphi_v, \varphi_v', f_v, s\right) = L_v\left(s+\frac{1}{2}, \pi_v, \chi_v\right).
\end{align}

\begin{rmk}
For those interested in constructing $p$-adic $L$-functions, note that one must modify the input at $p$ and compute the corresponding integrals at $p$.  The approach in \cite{HELS} assumes the prime $p$ splits, so the local integrals in that paper reduce immediately to (elaborate!) calculations for general linear groups.  In the case of $p$ inert, at least for unitary groups of signature $(a,a)$, in the spirit of Li's calculation mentioned above, a starting point would be to adapt the calculations carried out by Zheng Liu for symplectic groups in \cite{liuJussieu}, as discussed in \cite[Section 4.1]{MRLnonarch}.
\end{rmk}

\subsubsection{Factors at ramified nonarchimedean places}
At ramified places, it is possible to construct a section $f_v\in  \Ind_{P(\realfield_v)}^{U(\realfield_v)}\left(\chi_v|\cdot|^{-s}\right)$ such that
$\int_{U(\realfield_v)}f_v(g, 1)\langle \pi_v(g)\varphi_v, \varphi'_v\rangle dg$ is constant (see, e.g.\ \cite[p. 47]{PSR} or \cite[Section 4.2.2]{HELS}).

\subsubsection{Factors at archimedean places}\label{sec:archfac}
The author and Liu computed the archimedean zeta integrals in terms of a product of $\Gamma$-factors, without restriction on the weights or signature \cite{EL}.  For over 30 years, though, the question of how to compute them remained open, except in special cases.  The first progress in this direction was due to Paul Garrett, who showed that the archimedean zeta integrals are algebraic up to a particular power of the number $\pi$ \cite{ga06}.  In addition, for certain choices of Siegel sections $f_v$, he computed the archimedean integrals when a certain piece of the weight of the cuspidal automorphic representation $\pi$ is one-dimensional.  (Prior to Garrett's contribution, Shimura had addressed the case of scalar weight \cite{sh, shar}.)  The work in \cite{EL} builds on Liu's work for the symplectic case in \cite{liu-archimedean}, as well as work for unitary groups of signature $(n,1)$ in \cite{liu21, liun1}.

\subsection{The doubling method, revisited from the perspective of algebraic geometry}\label{sec:doublingPELvarieties}
Section \ref{sec:doubling} introduced the doubling method, as it was originally presented in \cite{PSR, ga}.  This formulation is not only important for obtaining an Euler product, but also for establishing key analytic properties, such as the functional equation and meromorphic continuation of the $L$-function.  On the other hand, if we want to study algebraic aspects of the values of the $L$-function, it is not necessarily immediately apparent how to translate this analytic formulation (in terms of integrals) into algebraic information.  

The key is to reformulate the doubling pairing in terms of PEL data and the corresponding moduli spaces.  This is the approach employed by Harris in his study of critical values of $L$-functions associated to automorphic representations of unitary groups in \cite{harriscrelle}, and it is also the approach used in the constructions of $p$-adic $L$-functions for unitary groups in \cite{HELS}.  This approach is also the main topic of \cite{harrisdsv}.

Now, the groups arising in the full PEL moduli problem are {\em general} unitary groups, i.e.\ have similitude factors.  As seen in \cite[Sections 3.1 and 3.2]{harriscrelle}, though, it is straightforward to adapt the doubling method to the case of similitude groups.  The inclusion $U\times U\hookrightarrow U_W$ from \eqref{doublingembedding} extends to an inclusion
\begin{align}\label{GUincldouble}\index{$GU_W$}
G(U\times U)\hookrightarrow GU_W:=GU(W, \langle, \rangle_W),
\end{align}
where
\begin{align*}
G(U\times U):=\left\{(g, h)\in GU(V, \langle, \rangle)\times GU(V, -\langle, \rangle)\mid \nu(g) = \nu(h)\right\}.
\end{align*}
Then the doubling integral defined in Equation \eqref{doublintegral} is replaced by a similar integral (with the input automorphic forms now defined on similitude groups) but over 
\begin{align*}
\left(Z\left(\adeles_{\realfield}\right)\left(G\left(U\times U\right)\right)(\realfield)\right)\backslash \left(G\left(U\times U\right)\right)\left(\adeles_{\realfield}\right),
\end{align*}
where $Z$ denotes the identity component of the center and is identified with the center of $GU(V, \langle, \rangle)=GU(V, -\langle, \rangle).$
Likewise, the invariant pairing from Equation \eqref{equ:invtpairing} is replaced by an integral over 
\begin{align*}
Z(\adeles)GU_V(\realfield)\backslash GU_V(\adeles_{\realfield}), 
\end{align*}
where $GU_V :=GU(V, \langle, \rangle)$.

We also note that the definition of the Eisenstein series from Equation \eqref{equ:siegeleseries} can be extended to $GU_W$.  Similarly to the unitary setting above, we denote by $GP$ the parabolic subgroup of $GU_W$ preserving $V^\Delta$ in $V_\Delta\oplus V^\Delta$.  In place of Equation \eqref{equ:PorR}, we have
\begin{align}\label{equ:GPorR}
GP(R)=\left\{\begin{pmatrix} \lambda A & 0\\ 0 &  { }^t\bar{A}^{-1}\end{pmatrix}\begin{pmatrix}1_n & X\\ 0& 1_n\end{pmatrix}\mid \lambda\in R^\times, A\in  GL_{\cmfield\otimes_{\realfield} R}(V\otimes_\realfield R), X\in\hern(\cmfield\otimes_{\realfield}R)\right\}.
\end{align}
Given a $\realfield$-algebra $R$ and a character $\psi$ of $(\cmfield\otimes_{\realfield}R)^\times$, we obtain a character of $GP(R)$ via
\begin{align*}
p=\begin{pmatrix} \lambda A & B\\ 0 & { }^t\bar{A}^{-1}\end{pmatrix}\mapsto\psi(\det A(p))|\lambda(p)|_{\realfield}^{-ns},
\end{align*}
where $A(p)=A$ and $\lambda(p)=\lambda$.
Correspondingly, as detailed in, e.g.\, \cite[Section 3.1]{apptoSHLvv}, we replace $I(s, \chi)$ from above by the induction from $GP(\adeles_{\realfield})$ to $GU_W(\adeles_{\realfield})$ of
\begin{align*}
p\mapsto \chi(\det A(p))|\lambda(p)|_{\realfield}^{-ns}|\det A(p)|^{-s+n/2},
\end{align*}
and we sum over $GP(\realfield)\backslash GU_W(\realfield)$ instead of $ P(\realfield)\backslash U_W(\realfield)$ in the definition of the Eisenstein series.

\begin{rmk}
From the discussion in Section \ref{unitarycon}, it is straightforward to write the doubling integral and subsequent discussion in terms of algebraic groups defined over $\IQ$.
\end{rmk}

\subsubsection{PEL data for the doubling method}
Ultimately, we are interested in studying algebraicity of values of the above $L$-functions. For this, it will be useful to relate the above approach to the PEL data and moduli spaces introduced in Section \ref{PELdatasection}.  Similarly to \cite[Section 3.1]{HELS}, we choose PEL data of unitary type that will induce an embedding of moduli spaces corresponding to the inclusion \eqref{GUincldouble}.   
In particular, we fix the following PEL data of unitary type, following the conventions of Section \ref{unitarycon}:
\begin{align*}
\mathfrak{D}_+&:=\left(\cmfield, \ast, \OK, V, \langle, \rangle_\IQ, L, h \right) \mbox{is a PEL datum of unitary type associated with $(V, \langle, \rangle)$}\\
\mathfrak{D}_-&:=\left(\cmfield, \ast, \OK, V, -\langle, \rangle_\IQ, L, z\mapsto h(\bar{z}) \right) \\
\mathfrak{D}_{+,-}&:=\left(\cmfield\times\cmfield, \ast\times\ast, \OK\times\OK, V\oplus V, \langle, \rangle_\IQ\oplus-\langle, \rangle_\IQ, L\oplus L, z\mapsto h(z)\oplus h(\bar{z}) \right), \\
&\mbox{ where $\langle, \rangle_\IQ\oplus-\langle, \rangle_\IQ$ is defined as in \eqref{opluspairing}}\\
\mathfrak{D}&:=\left(\cmfield, \ast, \OK, V\oplus V,  \langle, \rangle_\IQ\oplus-\langle, \rangle_\IQ, L\oplus L, z\mapsto h(z)\oplus h(\bar{z})\right)
\end{align*}
We write $G_+$, $G_-$, $G_{+,-}$, and $G$ for each of the algebraic groups corresponding, respectively, to this PEL data, as in Equation \eqref{GofR}.  The subscripts have been chosen to emphasize the connection with the unitary similitude groups immediately above, in particular:
\begin{itemize}
\item{$G_+$ is defined in terms of a pairing $\langle, \rangle$}
\item{$G_-$ is defined in terms of $-\langle, \rangle$ and is canonically isomorphic to $G_+$}
\item{$G_{+, -}$ is canonically embedded in $G_+\times G_-$}
\item{$G$ is defined in terms of $\langle, \rangle\oplus -\langle, \rangle$, and $G_{+, -}$ is canonically embedded inside $G$}
\end{itemize}
These four groups play analogous roles to the groups $GU(V, \langle, \rangle)$, $GU(V, -\langle, \rangle)$, $G(U\times U)$, and $GU(W, \langle, \rangle_W)$, respectively, from above.  We choose compatible compact open subgroups $\cpct_\pm$, $\cpct_{+, -}$, $\cpct$ inside the groups $G_{\pm}$, $G_{+,-}$, $G$.  We continue to use the subscripts here to refer to objects corresponding to these PEL data.  Like in \cite[Equation (38)]{HELS}, for each $\CO_{\reflex, (p)}$-scheme $S$ (or each $\reflex$-scheme $S$, depending on which moduli problem we are working with), this induces natural $S$-morphisms
\begin{align}
\moduliint_{+, -, \cpct_{+, -}}&\rightarrow\moduliint_\cpct\label{moduliincl1}\\
\moduliint_{+, -, \cpct_{+, -}}&\rightarrow \moduliint_{+,\cpct_+}\times_S\moduliint_{-,\cpct_-}.\label{moduliincl2}
\end{align}
Via the map \eqref{moduliincl1}, we can also pullback (i.e.\ restrict) automorphic forms, viewed as global sections of the vector bundle $\CE^\rho$, on $\moduliint_\cpct$ to obtain automorphic forms on $\moduliint_{+, -, \cpct_{+, -}}$.  Via the map \eqref{moduliincl2}, we can evaluate those automorphic forms on products of pairs of abelian varieties with PEL structure.  As an exercise to check their understanding, the reader might write down the corresponding maps of abelian varieties with PEL structure induced by these two maps of moduli spaces.

This setup, together with the algebraicity of the Eisenstein series discussed in Section \ref{sec:algEseries}, enables us to view the doubling integral as a pairing between algebraic automorphic forms over the moduli space $\moduliint_{+, -, \cpct_{+, -}}$.

\subsection{Algebraicity of Eisenstein series}\label{sec:algEseries}

The strategy outlined above for studying rationality (or algebraicity) of values of automorphic $L$-functions relies on the rationality (or algebraicity) of the Eisenstein series that are input to the doubling method.  Thanks to the algebraicity properties of the Maass--Shimura operators mentioned earlier that can be used to study Eisenstein series at points $s$ where they might not be holomorphic, we focus our discussion here on holomorphic Eisenstein series.  In the case of modular forms, the algebraic $q$-expansion principle tells you that modular forms $f$ are determined by their $q$-expansions (i.e.\ value of $f$ at the Tate curve with the canonical differential, as in \cite[Sections 1.1 and 1.2]{katzmodular} or \cite{conradqexp}).  More precisely, we have the following proposition.
\begin{prop}[Corollary 1.6.2 of \cite{katzmodular}]\label{prop:qexpprinciple}
Let $f$ be a holomorphic modular form of level $n$, defined over a $\ZZ[1/n]$-algebra $S$.  Suppose the $q$-expansion coefficients of $f$ lie in a $\ZZ[1/n]$-subalgebra $R\subseteq S$.  Then $f$ is defined over $R$.
\end{prop}
Because the algebraic $q$-expansions of a modular form $f$ agree with the analytic $q$-expansions (i.e.\ Fourier expansions) of $f$ (as noted in \cite[Equation (1.7.6)]{kaCM}), we can use the Fourier coefficients of a modular form to determine that it is defined over, say, $\IQ$ or some localization of a ring of integers.  In other words, to determine that $f$ is defined over some $\ZZ[1/n]$-algebra $R$, we can complete the following two steps:
\begin{enumerate}
\item{Determine the Fourier expansion of $f$.}
\item{Observe that the coefficients of $f$ are contained in $R$ (and apply Proposition \ref{prop:qexpprinciple}).}\label{item:step2q}
\end{enumerate}

Fortunately, as noted in Remark \ref{rmk:qexp1}, Lan has proved an analogous algebraic Fourier--Jacobi expansion principle for automorphic forms on unitary groups \cite[Proposition 7.1.2.14]{la}, and he has also shown that their analytically defined Fourier--Jacobi expansions also agree with algebraically defined Fourier--Jacobi expansions \cite{lancomparison}.  (For unitary groups of signature $(n,n)$ at each archimedean place, like in Section \ref{sec:qexp}, we can write this in a variable $q$ that makes it reasonable to call this a {\em $q$-expansion principle} again.)  This takes care of Step \eqref{item:step2q} in our setting.  Still, we are left with a potentially challenging problem: choosing an Eisenstein series and determining its Fourier cofficients.  The choice of Eisenstein series is influenced by our use of the doubling method introduced in Section \ref{sec:doubling}.  

To assist with computing the Fourier coefficients, we have a convenient result of Shimura.  Like above, we will work with Siegel Eisenstein series $E_f:=E_{f_{s, \chi}}$ associated to $f:={f_{s, \chi}}\in I(s, \chi)$ factoring as $f = \otimes_v f_v$, with $f_v\in I_v(s, \chi)$.  Recall that $E_f$ is an automorphic form on a  group that is of signature $(n,n)$ at each real place, which can be identified with the matrices preserving $\eta=\eta_n$ at each real place (as in Remark \ref{rmk:preservemx}) and that is assumed to be quasi-split.  In this case, $E_f$ has an adelic Fourier expansion, similar to the one in Equation \eqref{equ:Fourierexpnz}, whose Fourier coefficients each factor as a product of local Fourier coefficients. 
More precisely, we have the following result for $E_f^\ast(x):=E_f(x\eta^{-1}_f)$, where $\eta_f:=\prod_{v\ndivides \infty}\eta$.
\begin{thm}[Section 18.10 of \cite{sh}]\label{prop:sh1810}
\begin{align*}
E_f^\ast\left(q\right) = \sum_{\beta\in\hern(\cmfield)}c_\beta q^\beta,
\end{align*}
with $c_\beta$ complex numbers that factor as a precisely determined constant multiple of a product (over all places $v$ of $\realfield$) of local Fourier coefficients $c_{\beta, v}$ defined to be the Fourier transform of $f_v$ at $\beta$.  
\end{thm}
Choosing a section $f$ that factors over places of $\realfield$ not only played a role in producing the Euler product for our $L$-function, but also plays a crucial role in breaking our global Fourier coefficients into a product of local Fourier coefficients.
\begin{rmk}\label{rmk:Fexpnsrmk}
We stated Theorem \ref{prop:sh1810} in such a way as to highlight the fact that each Fourier coefficient decomposes as a product local factors, at the cost of omitting the setup necessary to give more precise statements about adelic Fourier expansions.  Since we do not need those details anywhere else in this manuscript, we simply refer the reader to \cite[Section 2.2.4]{apptoSHL} or \cite[Section 3.1]{apptoSHLvv} for the details specific to our situation here or to \cite[Section 18]{sh} for a more general and more detailed discussion of the adelic Fourier coefficients of Eisenstein series.  In {\em loc. cit.} and \cite{shconfluent}, Shimura also computed local Fourier coefficients for these Eisenstein series (for specific choices of $f_v$), which one can use to determine the ring of definition of the Eisenstein series.  For other purposes, such as constructing $p$-adic families of Eisenstein series and $p$-adic $L$-functions (like in \cite{apptoSHL, apptoSHLvv, HELS}), one needs different choices at primes dividing $p$ from those chosen by Shimura, and determining them constitutes a significant step in the $p$-adic situation.  We also note that Shimura's rationality results for Eisenstein series were further extended by Harris in \cite{harrisbirkhauser, harrisannals}.
\end{rmk}

\begin{rmk}
At the risk of opening a can of worms, we briefly elaborate on details for the statement of Theorem \ref{prop:sh1810}.  The reader wishing to delve still further into this subject after reading this remark, though, should heed Remark \ref{rmk:Fexpnsrmk}.  While the equation in Theorem  \ref{prop:sh1810} is formatted to highlight connections with earlier material, the longer-form version of the equation is
\begin{align*}
E_f\left(\begin{pmatrix}1& m \\ 0&1\end{pmatrix}\begin{pmatrix}{ }^t\bar{h}^{-1} & 0\\ 0& h\end{pmatrix}\right) = \sum_{\beta\in\hern(\cmfield)} c_\beta(h)\mathbf{e}_{\adeles_{\realfield}}(\tr(\beta m)),
\end{align*}
for each $h = (h_v)_v\in \GL_n\left(\adeles_\cmfield\right)$ and $m\in \hern(\adeles_\cmfield)$, with $c_\beta(h)$ a complex number that depends only on $f$, $\beta$, and $h$.  Here, $\mathbf{e}_{\adeles_{\realfield}}((x)) = \prod_v\mathbf{e}_v(x_v)$ for each adele $x = (x_v)_v$, $\mathbf{e}_v(x_v) = e^{2\pi i x_v}$ if $v\divides \infty$, and $\mathbf{e}_v(x_v) = e^{-2\pi i y}$ with $y\in\IQ$ such that $\tr_{\realfield_v/\IQ_q}(x_v)-y\in\ZZ_p$ at each place $v$ over a rational prime $q$.  The product $c_\beta(h)$ then is given as a rational multiple of the product $\prod_v c_{\beta, v}(h)$, with $c_{\beta, v}$ the {\em Fourier transform} of $f_v$:
\begin{align*}
c_{\beta, v}(h):= \int_{\hern(\cmfield\otimes \realfield_v)}f_v\left(\begin{pmatrix}0 & -1\\ 1 & 0\end{pmatrix}\begin{pmatrix} 1 & m_v\\ 0 & 1
\end{pmatrix}\begin{pmatrix}{ }^t \bar{h}_v^{-1} & 0\\ 0 & h_v\end{pmatrix}\right)\mathbf{e}_v\left(-\tr \beta m_v\right)dm_v.
\end{align*}
\end{rmk}
It turns out that the numbers $c_{\beta, v}$ are completely determined by their values at $1$.  In the case of modular forms of weight $2k$ and level $1$, the sections $f_v$ can be chosen so that for $n\geq 1$, $c_{n, \infty}(1) = n^{2k-1}$ and $c_{n, p}(1) = \sum_{j=0}^{\ord_p(n)}p^{-j(2k-1)}$, so $c_n(1) = n^{2k-1}\prod_{p}\sum_{j=0}^{\ord_p(n)}p^{-j(2k-1)} = \sum_{d\divides n}d^{2k-1},$ the usual divisor function occurring in the Fourier expansion of the weight $2k$, level $1$ Eisenstein series.  For further details, the reader is again urged to consult the aforementioned references.

\bibliography{AWSbib}

\printindex
\end{document}